\documentclass[reqno,12pt]{amsart}
\usepackage{amsmath,amssymb,amsfonts,amsthm}
\usepackage{bm,bbm,mathtools}
\usepackage{comment,enumitem,multicol}
\usepackage{expl3}

\usepackage{array,booktabs,longtable}
\usepackage{caption,graphicx}
\usepackage[table]{xcolor}

\newlength{\margins}
\usepackage[
  top=\margins,
  bottom=\margins,
  left=\margins,
  right=\margins
]{geometry}

\usepackage[
  setpagesize=false,
  bookmarks=true,
  bookmarksdepth=section,
  bookmarksnumbered=true,
  colorlinks=false,
  hidelinks,
  pdftitle={Harder's conjecture and Hermitian automorphic forms},
  pdfauthor={Hidenori Katsurada and Nobuki Takeda},
  pdfkeywords={Hermitian automorphic forms,
    Harder's conjecture,
    Galois representation}
]{hyperref}

\allowdisplaybreaks
\makeatletter
\@namedef{subjclassname@2020}{\textup{2020} Mathematics Subject Classification}
\makeatother

\title{Harder's conjecture and Hermitian automorphic forms}

\author[H. Katsurada]{Hidenori KATSURADA}
\address[Hidenori Katsurada]{Department of Mathematics, Hokkaido University, Kita 10, Nishi 8,
Kita-ku, Sapporo, Hokkaido 060-0810, Japan;
and Muroran Institute of Technology, Mizumoto 27-1, Muroran, Hokkaido 050-8585, Japan}
\email{hidenori@muroran-it.ac.jp}

\author[N. Takeda]{Nobuki TAKEDA}
\address[Nobuki Takeda]{Department of Mathematics, Graduate School of Science, Kyoto University, Kitashirakawa Oiwake-cho, Sakyo-ku, Kyoto 606-8502, Japan}
\email{takeda.nobuki.z04@kyoto-u.jp}


\theoremstyle{definition}
\newtheorem{dfn}{Definition}[section]
\newtheorem{rem}[dfn]{Remark}

\theoremstyle{plain}
\newtheorem{prop}[dfn]{Proposition}
\newtheorem{conj}[dfn]{Conjecture}
\newtheorem{lem}[dfn]{Lemma}
\newtheorem{thm}[dfn]{Theorem}
\newtheorem{cor}[dfn]{Corollary}
\newtheorem{ex}[dfn]{Example}

\DeclareMathOperator{\tr}{Tr}

\ExplSyntaxOn
\clist_const:Nn \c_alphabet_clist
  {
    A,B,C,D,E,F,G,H,I,J,K,L,M,N,O,P,Q,R,S,T,U,V,W,X,Y,Z,
    a,b,c,d,e,f,g,h,i,j,k,l,m,n,o,p,q,r,s,t,u,v,w,x,y,z
  }

\cs_new:Npn \makealphabet #1 #2
  {
    \clist_map_inline:Nn \c_alphabet_clist
      { \exp_args:Nc \newcommand { #1 ##1 } { #2 { ##1 } } }
  }

\makealphabet{bb}{\mathbb}
\makealphabet{bf}{\mathbf}
\makealphabet{cal}{\mathcal}
\makealphabet{frak}{\mathfrak}

\clist_map_inline:nn
  {
    alg,Asai,Aut,BC,Cent,Cl,cont,cris,diag,disc,dR,End,
    Ext,Fil,Frob,Gal,GL,GSp,GU,Her,hol,Hom,HT,id,Ind,Lie,Lift,M,
    ord,PGL,PGSp,pr,proj,PSL,rad,Res,Sat,sh,SL,SO,Sp,sph,Si,prim,
    Spin,St,Sym,Tame,U,ur,vol,res
  }
  { \exp_args:Nc \newcommand { #1 } { \mathrm { #1 } } }

\ExplSyntaxOff

\newcommand{\bsk}{\boldsymbol{k}}
\newcommand{\typeIII}{\mathrm{I\!I\!I}}
\newcommand{\leftexp}[2]{{\vphantom{#2}}^{#1}{#2}}
\providecommand{\adj}[1]{{#1^*}}
\providecommand{\GU}[1]{\mathrm{GU}_{#1}}

\renewcommand{\Im}{\mathrm{Im}}

\numberwithin{equation}{section}
\makeatletter
\def\@startsection#1#2#3#4#5#6{%
  \if@noskipsec \leavevmode \fi
  \par \@tempskipa #4\relax
  \@afterindenttrue
  \ifdim \@tempskipa <\z@
    \@tempskipa -\@tempskipa
    \@afterindentfalse
  \fi
  \if@nobreak
    \everypar{}%
  \else
    \addpenalty\@secpenalty
    \addvspace\@tempskipa
  \fi
  \@ifstar
    {\@dblarg{\@sect{#1}{\@M}{#3}{#4}{#5}{#6}}}%
    {\@dblarg{\@sect{#1}{#2}{#3}{#4}{#5}{#6}}}%
}
\makeatother

\begin{document}

\begin{abstract}
  Let \(k\ge4\) and \(j\ge2\) be integers with \(j\) even.
  Harder's conjecture predicts a congruence between a primitive elliptic cusp form of weight \(2k+j-2\) and a degree-two vector-valued Siegel Hecke cusp eigenform of weight \({\det}^{k}\Sym^{j}\).
  We prove the corresponding Harder-type congruence for the spinor \(L\)-polynomials under explicit arithmetic hypotheses on a congruence prime and an auxiliary imaginary quadratic field.

  Our approach uses Hermitian automorphic forms on the quasi-split unitary group \(\U_{2,2}\).
  We construct Hermitian spin lifts of Siegel cusp forms and show that the Hermitian cusp form produced by a Hermitian Klingen--Eisenstein congruence lies in the image of this lift.
  The key step is to prove that the Galois representation attached to this Hermitian cusp form is conjugate invariant.
  For this,
  we combine the endoscopic classification for quasi-split unitary groups with Selmer-group vanishing results.
  This allows us to identify the Hermitian cusp form as the spin lift of a Siegel cusp eigenform and thereby obtain the Harder-type congruence.
  The argument applies uniformly to both even and odd \(k\).
\end{abstract}

\maketitle

\tableofcontents

\section{Introduction}
\label{sec:intro}

Harder's conjecture predicts congruences between elliptic cusp forms and Siegel cusp forms of degree two.
Let \(k\ge 3\) and let \(j\ge2\) be even.
For a primitive elliptic cusp form \(f\) of weight \(2k+j-2\) and a prime ideal \(\frakp\) of its Hecke field \(\bbQ(f)\) dividing the algebraic part \(L^{\alg}(k+j,f)\) of the critical value,
the conjecture predicts the existence of a Siegel cusp eigenform \(F\in\calS_{{\det}^{k}\Sym^{j}}(\Sp_2)\) and a prime ideal \(\frakP\) above \(\frakp\) such that
\[
  L_\ell(X,F,\Spin) \equiv L_\ell(X,f)(1-\ell^{k-2}X)(1-\ell^{k+j-1}X)\pmod{\frakP}
\]
for every prime \(\ell\).
Here \(\calS_{{\det}^{k}\Sym^{j}}(\Sp_2)\) denotes the space of degree-two Siegel cusp forms of weight \({\det}^{k}\Sym^{j}\).
The algebraic part of the critical value and the local Euler polynomials are normalized as in Subsections~\ref{subsec:l-values} and~\ref{subsec:harder-cong},
respectively.
This conjecture goes back to Harder \cite{Harder2008Congruence};
we use the formulation given in \cite[Conjecture~3.7]{Atobe2023HarderI}.

We first place our result in the context of previous work.
Ibukiyama related Harder's conjecture to Shimura-type correspondences and half-integral weight forms \cite{Ibukiyama2008ShimuraHarder,Ibukiyama2014ShimuraHarder}.
A different approach was developed by Chenevier and Lannes through automorphic forms arising from even unimodular lattices and Kneser neighbors of the Niemeier lattices \cite{Chenevier2019Automorphic};
using Arthur's classification,
they proved,
among other applications,
a congruence predicted by Harder.
More generally,
Bergstr\"om and Dummigan formulated Eisenstein congruences for split reductive groups \cite{Bergstrom2016Eisenstein}.

Atobe--Chida--Ibukiyama--Katsurada--Yamauchi formulated a stronger version of Harder's conjecture in terms of a Klingen--Eisenstein lift of the Duke--Imamo\u{g}lu--Ikeda lift and a certain lift of degree-two Siegel cusp forms,
and proved it in several cases \cite{Atobe2023HarderI}.
Their sequel \cite{Atobe2026HarderII} proves the conjecture under mild additional hypotheses when \(k\) is even and \(j\equiv0\pmod4\),
using congruences of Klingen--Eisenstein lifts together with Galois representations and Selmer groups.
For odd \(k\),
Katsurada--Lee proved Harder-type congruences in several cases by using Miyawaki lifts \cite{Katsurada2026HarderMiyawaki}.
The approach of the present paper is different:
we first produce a congruence on \(\U_{2,2}\) and then prove that the resulting Hermitian cusp form descends from a Siegel cusp form.
In particular,
the same argument treats even and odd \(k\) uniformly.
Moreover,
our criterion recovers several previously known cases of the Harder-type congruence,
including examples established in \cite{Atobe2023HarderI,Atobe2026HarderII,Katsurada2026HarderMiyawaki},
as explained in Section~\ref{sec:examples}.

Related congruence problems have also been studied at paramodular level.
Fretwell formulated an analogue of Harder's conjecture and provided numerical evidence at prime levels \cite{Fretwell2018Paramodular}; see also \cite{Fretwell2017Generic}.
Dummigan--Pacetti--Rama--Tornar\'{\i}a subsequently proved several examples of these paramodular congruences by using quinary forms and a Jacquet--Langlands type correspondence \cite{Dummigan2024Quinary}.

Dummigan \cite[\S~8]{Dummigan2014Eisenstein} proposed an approach to Harder's conjecture through Hermitian automorphic forms.
In that work,
Dummigan considered congruences between Hermitian Klingen--Eisenstein series and cusp forms on \(\U_{2,2}\) and observed that they should be functorial lifts of the congruences appearing in Harder's conjecture.
Dummigan also indicated that differential operators and pullback formulas may provide a way to study these Hermitian congruences.
The remaining problem is then to show that the Hermitian cusp form produced by the congruence lies in the image of the functorial lift from \(\SO(3,2)\simeq\PGSp_2\) to \(\U_{2,2}\).

The principal new step in the present paper is to prove this statement under explicit arithmetic hypotheses.
We fix an auxiliary imaginary quadratic field \(E/\bbQ\) and work with the quasi-split unitary group \(\U_{2,2}\) attached to \(E/\bbQ\).
We write \(D_E\) for the discriminant of \(E\),
\(\Cl(E)\) for its ideal class group,
and \(h_E=\#\Cl(E)\) for its class number.
Here \(\calS_\tau(\U_{2,2})\) denotes the space of Hermitian cusp forms of weight \(\tau\) defined in Subsection~\ref{subsec:automorphic-forms}.
Put
\[
  \tau_{k,j} = {\det}^{j/2+2}\Sym^{k-3} \boxtimes {\det}^{j/2+2}\Sym^{k-3}.
\]
A Hermitian Klingen--Eisenstein congruence established by the second author in \cite[Corollary~6.5]{Takeda2026Congruences} gives,
under explicit hypotheses,
a Hermitian Hecke cusp eigenform \(G\in\calS_{\tau_{k,j}}(\U_{2,2})\) satisfying
\begin{equation}
  \label{eq:intro-hermitian-congruence}
  G\equiv_{ev}[f]^2_{j+4}\pmod{\frakP}.
\end{equation}
Here \([f]^2_{j+4}\) is the Hermitian Klingen--Eisenstein lift associated with \(f\),
and \(\equiv_{ev}\) denotes congruence of all finite Hecke eigenvalues;
the precise definitions are given in Subsection~\ref{subsec:herm-cong}.
The central problem is to show that \(G\) lies in the image of the functorial lift from Siegel cusp forms associated with the spin \(L\)-homomorphism.
We refer to this lift as the Hermitian spin lift.

We first construct the relevant lift in the forward direction.
The construction uses the spin \(L\)-homomorphism and the endoscopic classification in~\cite{Mok2015Endoscopic}.
For every Hecke cusp eigenform \(F\) of weight \({\det}^{k}\Sym^{j}\) on \(\Sp_2\),
it gives a Hermitian Hecke cusp eigenform \(\Lift_E(F)\in\calS_{\tau_{k,j}}(\U_{2,2})\).
Then
\begin{equation}
  \label{eq:intro-spin-lift}
  L(s,\Lift_E(F),\St) = L\left(s+\frac{2k+j-3}{2},F,\Spin\right) L\left(s+\frac{2k+j-3}{2},F,\Spin\otimes\chi_E\right).
\end{equation}
Here \(\chi_E\) is the quadratic character associated with \(E/\bbQ\).
This construction is carried out in Proposition~\ref{prop:spin-lift}.
More generally, Dummigan conjectured a functorial lift from \(\SO(n+1,n)\) to \(\U_{n,n}\);
see \cite[\S~6]{Dummigan2014Eisenstein}.

The factorization \eqref{eq:intro-spin-lift} explains the connection with Harder's conjecture.
If the form \(G\) in \eqref{eq:intro-hermitian-congruence} is of the form \(G=\Lift_E(F)\),
then comparison with the explicit standard Euler factors of the Hermitian Klingen--Eisenstein lift gives
\[
  L_\ell(X,F,\Spin) \equiv L_\ell(X,f)(1-\ell^{k-2}X)(1-\ell^{k+j-1}X) \pmod{\frakP}
\]
for every prime \(\ell\).
Thus the main issue is the converse direction:
given the Hermitian cusp form \(G\) in \eqref{eq:intro-hermitian-congruence},
prove that \(G=\Lift_E(F)\) for some Siegel cusp eigenform \(F\).
Our main theorem establishes this converse under explicit arithmetic hypotheses and yields the following Harder-type congruence.
In condition~\((6)\) below,
\(\calG_0^{(2)}\) is the fixed set of adelic ideal-class representatives used to describe Fourier expansions on \(\U_{2,2}\),
and \(A_{[f]^2_{j+4}}(\gamma,S)\) denotes the corresponding normalized Fourier coefficient of the Hermitian Klingen--Eisenstein lift;
the coefficient lattice and its valuation are specified in Subsection~\ref{subsec:herm-cong}.

\begin{thm}[Main theorem]
  Let \(k \ge 4\) and \(j \ge 2\) be integers with \(j\) even.
  Let \(f\) be a primitive elliptic cusp form of weight \(2k+j-2\).
  Let \(E\) be an imaginary quadratic field.
  Let \(K/\bbQ(f)\) be a sufficiently large finite Galois extension.
  Assume that there exist a prime ideal \(\frakp\) of \(\bbQ(f)\)
  and a prime ideal \(\frakP\) of \(K\) lying above \(\frakp\),
  such that the following conditions are satisfied.
  \begin{enumerate}[itemsep=3pt]
    \item The rational prime \(p\) below \(\frakp\) satisfies \(p>2k+j-2\).

    \item After a suitable choice of a \(\Gal_\bbQ\)-stable lattice and scalar extension to \(K_\frakP\),
          the image of \(\Gal_{\bbQ(\zeta_{p^\infty})}\) under the Galois representation attached to \(f\) contains \(\SL_2(\bbZ_p)\).

    \item \(f\) is not congruent modulo \(\frakp\) to any other primitive elliptic cusp form of weight \(2k+j-2\).

    \item The prime \(\frakP\) does not divide \(\zeta(-j-1)\).

    \item \(\frakp\) divides \(L^{\alg}(k+j,f)\).

    \item There exist \(\gamma_0\in\calG_0^{(2)}\) and \(S_0\in\Her_2(E)_{>0}\), with notation as in Subsection~\ref{subsec:herm-cong},
          such that
          \[
            v_\frakp\!\left(
            L^{\alg}(k+j,f)A_{[f]^2_{j+4}}(\gamma_0,S_0)
            \right)=0.
          \]

    \item We have \(p\nmid D_E h_E\).

    \item The prime \(\frakP\) does not divide \(L(-j,\chi_E)\).

    \item The prime \(\frakP\) divides neither \(L^{\alg}(k+j,f\otimes\chi_E)\) nor \(L^{\alg}(k+j-1,f\otimes\chi_E)\).
  \end{enumerate}
  Then there exists a Hecke cusp eigenform \(F\in\calS_{{\det}^k\Sym^j}(\Sp_2)\) such that
  \[
    L_\ell(X,F,\Spin) \equiv L_\ell(X,f)(1-\ell^{k-2}X)(1-\ell^{k+j-1}X) \pmod{\frakP}
  \]
  for every rational prime \(\ell\).
\end{thm}
This is proved below as Theorem~\ref{thm:harder-cong}.

The first five assumptions are independent of the auxiliary imaginary quadratic field \(E\),
whereas conditions~\((6)\)--\((9)\) arise from the Hermitian Klingen--Eisenstein congruence and the subsequent descent argument over \(E\).
The auxiliary field \(E\) can be chosen with some freedom.
This flexibility suggests that a refinement of the argument may remove conditions~\((6)\)--\((9)\).

The divisibility \(\frakp\mid L^{\alg}(k+j,f)\) is the congruence condition predicted by Harder.
The other assumptions arise from the construction of the Hermitian Klingen--Eisenstein congruence and from the argument using the Galois representation attached to \(G\) to show that \(G\) lies in the image of the Hermitian spin lift.
They are effective and can be checked explicitly.
In Section~\ref{sec:examples},
we give practical criteria for the large-image hypothesis and verify the conditions in a substantial collection of cases with \(2k+j-2\le40\).

The proof has two main parts.
The first is automorphic.
We construct the Hermitian spin lift \(\Lift_E(F)\) and prove that a congruence
\[
  \Lift_E(F)\equiv_{ev}[f]^2_{j+4}\pmod{\frakP}
\]
implies the Harder-type congruence for the spinor \(L\)-polynomials.
This is Theorem~\ref{thm:harder}.

The second part treats the converse direction.
Fix a complex conjugation \(c\in\Gal_\bbQ\),
let \(\omega\) be the \(p\)-adic cyclotomic character,
and put \(\omega_E=\omega|_{\Gal_E}\).
For a representation \(R\) of \(\Gal_E\),
write \(R^c(g)=R(cgc^{-1})\).
Let \(R_G^\natural\) be the normalized \(p\)-adic Galois representation of the absolute Galois group \(\Gal_E\) of \(E\) attached to the Hermitian cusp form \(G\).
The congruence \eqref{eq:intro-hermitian-congruence} determines the residual semisimplification of \(R_G^\natural\) by
\[
  \overline R_G^{\natural,ss} \simeq \left( \overline\rho_f \oplus \overline\omega^{-k+2} \oplus \overline\omega^{-k-j+1} \right)\big|_{\Gal_E}.
\]
Here \(\overline\rho_f\) denotes the semisimplified residual representation attached to \(f\),
and \(\overline\omega\) denotes the mod \(\frakP\) cyclotomic character.

The unitary automorphic parameter already gives the conjugate self-duality
\[
  (R_G^\natural)^\vee
  \simeq
  (R_G^\natural)^c\omega_E^{2k+j-3}.
\]
The crucial point is to prove the stronger conjugate invariance
\begin{equation}
  \label{eq:intro-conj-inv}
  (R_G^\natural)^c\simeq R_G^\natural,
\end{equation}
which allows the associated cuspidal representation of \(\GL_4(\bbA_E)\) to descend to \(\GL_4(\bbA_\bbQ)\).

We first prove that \(R_G^\natural\) is absolutely irreducible,
and then establish \eqref{eq:intro-conj-inv}.
We follow the argument developed in \cite{Atobe2026HarderII},
together with the notion of the Bella\"{\i}che--Chenevier (BC) sign for self-dual and conjugate self-dual Galois representations \cite{Bellaiche2009Families,Bellaiche2011Sign}.
The use of this sign refines the method of \cite{Atobe2026HarderII}.
In particular,
it simplifies the proofs of several results and allows their hypotheses to be weakened,
as explained in Remark~\ref{rem:advantage-of-bc-sign} following Theorem~\ref{thm:conj-inv}.
If \(R_G^\natural\) were not conjugate invariant,
then a lattice argument applied to \(\Ind_{\Gal_E}^{\Gal_\bbQ}R_G^\natural\) would produce a non-zero residual extension between a constituent of \(\calB_f\) and a constituent of \(\calB_f\otimes\overline{\chi}_E\).
The local vanishing at primes dividing \(D_E\),
together with the Bloch--Kato Selmer-group vanishing results established in Section~\ref{sec:galois-selmer},
rules out all such classes;
for the pairs exchanged by the self-duality,
we also use the BC sign \(+1\) in the sense of Subsection~\ref{subsec:bc-signs}.

By the Chebotarev density theorem and strong multiplicity one,
the identity \eqref{eq:intro-conj-inv} implies \(\Pi_G^c\simeq\Pi_G\) for the associated cuspidal representation \(\Pi_G\) of \(\GL_4(\bbA_E)\).
The cyclic base-change theorem of Arthur--Clozel \cite[Chapter~3, Theorem~4.2]{Arthur1989BaseChange} then gives a cuspidal automorphic representation \(\Pi\) of \(\GL_4(\bbA_\bbQ)\) whose base change to \(E\) is \(\Pi_G\),
and we show that \(\Pi\) can be chosen unramified at every finite prime.
Using the conjugate self-duality of \(\Pi_G\) and the archimedean parameter,
we then show that \(\Pi\) is self-dual of symplectic type.
The endoscopic classification for \(\PGSp_2\) \cite{Arthur2013Endoscopic} consequently produces a Siegel cusp eigenform \(F\) with
\[
  G=\Lift_E(F).
\]
This is Theorem~\ref{thm:conj-inv}.
Combining the two parts gives Theorem~\ref{thm:harder-cong}.

We work throughout in the range \(k\ge4\),
with \(j\ge2\) even.
No parity condition is imposed on \(k\).
The imaginary quadratic field \(E\) is auxiliary and may be chosen subject to the arithmetic conditions in the theorem.

The restriction \(k\ge4\) is imposed in order to avoid a boundary case in the classification in~\cite{Mok2015Endoscopic}.
When \(k=3\),
the relevant Hermitian weight becomes scalar-valued.
Although the infinitesimal character at infinity remains regular,
the corresponding Arthur parameters may involve a non-trivial Arthur \(\SL_2\)-factor.
The Hermitian Klingen--Eisenstein congruence used here also applies in this scalar-valued case,
but the classification and the realization of \(G\) as a Hermitian spin lift require a separate discussion.
We therefore restrict to \(k\ge4\).

The paper is organized as follows.
In Section~\ref{sec:auto-forms},
we fix the automorphic conventions and the normalizations of the \(L\)-functions used throughout.
In Section~\ref{sec:spin-lifts},
we construct the Hermitian spin lift and show how a congruence involving such a lift implies the Harder-type congruence.
In Section~\ref{sec:galois-selmer},
we establish the results on Galois representations and Selmer groups needed in the proof and recall the Bella\"{\i}che--Chenevier sign.
In Section~\ref{sec:main-thm},
we prove the absolute irreducibility and conjugate invariance of \(R_G^\natural\),
show that the relevant Hermitian cusp form lies in the image of the Hermitian spin lift,
and deduce Theorem~\ref{thm:harder-cong}.
Finally,
Section~\ref{sec:examples} gives numerical examples.

\subsection*{Acknowledgements}
The authors are grateful to T.~Ikeda for his suggestions and support.
They also thank N.~Dummigan for helpful correspondence.

Takeda was supported by the Japan Science and Technology Agency (JST) SPRING Program,
Grant Number JPMJSP2110,
and by JSPS KAKENHI Grant Number JP26KJ1378.
Katsurada was supported by KAKENHI Grant Number JP24K06660.

\section*{Notation}
We denote by \(\M_{m,n}(R)\) the set of \(m\times n\) matrices with entries in \(R\).
In particular,
we put \(\M_n(R):=\M_{n,n}(R)\).
We denote by \(I_n\) the identity matrix in \(\M_n(R)\) and by \(e_{ij}\) the matrix unit whose \((i,j)\)-entry is \(1\).
For \(X\in \M_{m,n}(R)\),
we denote by \({}^tX\) the transpose of \(X\).
For \(X\in \M_n(R)\),
we denote its determinant and trace by \(\det(X)\) and \(\tr(X)\),
respectively.
We denote by \(\GL_n(R)\) the general linear group of degree \(n\) over \(R\).

For the fixed imaginary quadratic field \(E/\bbQ\),
let \(x\mapsto\overline{x}\) denote its non-trivial automorphism.
Whenever a ring \(R\) is equipped with an involution \(x\mapsto\overline{x}\),
for \(X=(x_{ij})\in \M_{m,n}(R)\),
we put \(\overline{X}=(\overline{x_{ij}})\) and \(\adj{X}={}^t\overline{X}\),
and define
\[
  \Her_n(R)=\{X\in\M_n(R)\mid \adj X=X\}.
\]
For \(R=E\otimes_\bbQ R_0\),
where \(R_0\) is a \(\bbQ\)-algebra,
and for every subring of such an \(R\) stable under this automorphism,
the involution is induced by the non-trivial automorphism of \(E/\bbQ\);
for \(R=\bbC\),
it is complex conjugation.
We denote by \(D_E\) the discriminant of \(E\),
by \(\Cl(E)\) its ideal class group,
and by \(h_E=\#\Cl(E)\) its class number.

Via the fixed embedding \(E\subset\bbC\),
we regard \(\Her_n(E)\) as a subset of \(\Her_n(\bbC)\).
For \(X\in \Her_n(\bbC)\),
we write \(X>0\) if \(X\) is positive definite and \(X\ge 0\) if \(X\) is positive semidefinite.
For a subset \(S\subset \Her_n(\bbC)\),
we write \(S_{>0}\) and \(S_{\ge 0}\) for the subsets of positive definite and positive semidefinite matrices in \(S\),
respectively.

If a group \(G\) acts on a vector space \(V\),
we denote by \(V^G\) the subspace of \(G\)-invariant vectors in \(V\).
We denote by \({\det}^k\) the one-dimensional representation of \(\GL_n(\bbC)\) given by the \(k\)-th power of the determinant,
and by \(\Sym^l\) the \(l\)-th symmetric power of the standard representation of \(\GL_n(\bbC)\).
If \(V\) is a finite-dimensional representation of a group \(G\) over a field \(K\),
we denote its contragredient representation by \(V^\vee=\Hom_K(V,K)\).
For a finite-length representation \(M\),
we denote by \(M^{ss}\) its semisimplification.

Let \(K\) be an algebraic number field,
and let \(\frakp\) be a prime ideal of \(K\).
We denote by \(K_\frakp\) the \(\frakp\)-adic completion of \(K\),
and by \(\calO_K\) and \(\calO_\frakp\) the rings of integers of \(K\) and \(K_\frakp\),
respectively.
Let \(p_\frakp\) be the rational prime below \(\frakp\).
We denote by \(\kappa_\frakp\) the residue field of \(K_\frakp\).
We put \(\calO_{(\frakp)}=K\cap \calO_\frakp\).
For \(x\in K\),
we say that \(\frakp\) divides \(x\),
and write \(\frakp\mid x\),
if \(x\in\frakp\calO_\frakp\).
We say that \(x\) is a \(\frakp\)-adic unit if \(x\in\calO_\frakp^\times\).

For a field \(K\),
we denote by \(\Gal_K\) the absolute Galois group of \(K\).
If \(K\) is a local or global field,
we denote by \(W_K\) its Weil group.

Fix a complex conjugation \(c\in\Gal_\bbQ\).
Since \(E\) is imaginary quadratic,
the restriction of \(c\) to \(E\) is the non-trivial automorphism of \(E/\bbQ\).
For a representation \(R\) of \(\Gal_E\),
we define its conjugate \(R^c\) by \(R^c(g)=R(cgc^{-1})\) for \(g\in\Gal_E\).

For an automorphic representation \(\Pi\) of \(\GL_n(\bbA_E)\),
we denote by \(\Pi^c\) its pullback under the automorphism of \(\GL_n(\bbA_E)\) induced by \(c\).
We use the same convention for Hecke characters of \(E\).
For a place \(w\) of \(E\),
we denote its conjugate by \(w^c\).

\section{Automorphic forms}
\label{sec:auto-forms}
\subsection{Unitary and symplectic groups}
\label{subsec:groups}

Let \(E\) be an imaginary quadratic extension of \(\bbQ\).
For a prime number \(p\),
we put
\[
  E_p=E\otimes_\bbQ \bbQ_p = \prod_{w\mid p}E_w,
\]
\[
  \calO_{E_p} = \prod_{w\mid p}\calO_{E_w}.
\]
We denote by \(\bbA_\bbQ\) the adele ring of \(\bbQ\),
and by \(\bbA_f\) and \(\bbA_\infty\) its finite and infinite parts.

\subsubsection{Unitary groups}
Let \(N\) be a positive integer and put
\[
  J^\U_N=
  \begin{cases}
    \begin{pmatrix}
      0   & -I_n \\
      I_n & 0
    \end{pmatrix}  & (N=2n),  \\
    \begin{pmatrix}
      0   & 0 & I_n \\
      0   & 1 & 0   \\
      I_n & 0 & 0
    \end{pmatrix} & (N=2n+1),
  \end{cases}
\]
where the blocks in the second matrix have sizes \(n,1,n\).
For every \(\bbQ\)-algebra \(R\),
define
\[
  G_N(R)=
  \left\{
  g\in\GL_N(E\otimes_\bbQ R)
  \ \middle|\
  \adj{g}J^\U_N g=J^\U_N
  \right\}.
\]
Its similitude group \(\widetilde G_N\) is defined by
\[
  \widetilde G_N(R)=
  \left\{
  g\in\GL_N(E\otimes_\bbQ R)
  \ \middle|\
  \adj{g}J^\U_N g=\nu(g)J^\U_N
  \text{ for some }\nu(g)\in R^\times
  \right\}.
\]
The resulting homomorphism \(\nu:\widetilde G_N\to\bbG_m\) is the similitude character,
and its kernel is \(G_N\).

Write
\[
  (a_N,b_N)=
  \begin{cases}
    (n,n)   & (N=2n),   \\
    (n+1,n) & (N=2n+1).
  \end{cases}
\]
We write
\[
  \U_{a_N,b_N}=G_N,
  \qquad
  \GU_{a_N,b_N}=\widetilde G_N,
\]
and use only these quasi-split unitary groups throughout the paper.

For the dual-group conventions,
put
\[
  \Phi_N=
  \begin{pmatrix}
    0          & 0          & \cdots                  & 0      & 1      \\
    0          & 0          & \cdots                  & -1     & 0      \\
    \vdots     & \vdots     & \reflectbox{\(\ddots\)} & \vdots & \vdots \\
    0          & (-1)^{N-2} & \cdots                  & 0      & 0      \\
    (-1)^{N-1} & 0          & \cdots                  & 0      & 0
  \end{pmatrix}.
\]

The dual and \(L\)-groups are
\[
  \widehat G_N=\GL_N(\bbC),
  \qquad
  \leftexp{L}{G_N}=\GL_N(\bbC)\rtimes W_{\bbQ},
\]
where the action of \(W_{\bbQ}\) factors through \(\Gal(E/\bbQ)\),
and the non-trivial element acts on \(\widehat G_N\) by
\[
  h\longmapsto \Phi_N\,{}^t h^{-1}\Phi_N^{-1}.
\]
This representative preserves the standard splitting of \(\GL_N(\bbC)\).
We also use
\[
  \leftexp{L}{(\Res_{E/\bbQ}\GL_N)} = (\GL_N(\bbC)\times\GL_N(\bbC))\rtimes W_{\bbQ},
\]
where \(W_E\) acts trivially on the dual group and the non-trivial element of \(\Gal(E/\bbQ)\) interchanges the two factors.

For the remainder of this subsection,
assume \(N\ge2\),
and write \((a,b)=(a_N,b_N)\) and \(r=a-b\).
Put
\[
  H_{a,b}=
  \begin{pmatrix}
    0            & 0   & -\sqrt{-1}I_b \\
    0            & I_r & 0             \\
    \sqrt{-1}I_b & 0   & 0
  \end{pmatrix},
\]
where the block sizes are \(b,r,b\),
and define the real Lie group
\[
  \U(a,b)=
  \left\{
  g\in\GL_N(\bbC)
  \ \middle|\
  \adj{g}H_{a,b}g=H_{a,b}
  \right\}.
\]
Fix an isomorphism
\[
  \iota_{N,\infty}:G_N(\bbR)\xrightarrow{\ \sim\ }\U(a,b),
\]
and henceforth use it to identify these real Lie groups.
If \(N=2n\),
then \(H_{n,n}=\sqrt{-1}J^\U_{2n}\),
so that \(\iota_{2n,\infty}\) may be taken to be the identity.

We define the Hermitian symmetric domain attached to \(\U(a,b)\) by
\[
  \calH^{(\mathrm{I})}_{a,b} = \left\{ Z=
  \begin{pmatrix}
    Z_1 \\
    Z_2
  \end{pmatrix}
  \in \M_{a,b}(\bbC) \ \middle|\
  \begin{array}{l}
    Z_1\in \M_b(\bbC), \ Z_2\in \M_{r,b}(\bbC), \\
    \sqrt{-1}({}^t\overline{Z_1}-Z_1)-{}^t\overline{Z_2}Z_2>0
  \end{array}
  \right\}.
\]
Here \(\M_{r,b}(\bbC)=0\) if \(r=0\).
For
\[
  g=
  \begin{pmatrix}
    A & B \\
    C & D
  \end{pmatrix}
  \in \U(a,b),
\]
with \(A\in \M_a(\bbC)\),
\(B\in \M_{a,b}(\bbC)\),
\(C\in \M_{b,a}(\bbC)\),
and \(D\in \M_b(\bbC)\),
the action of \(\U(a,b)\) on \(\calH^{(\mathrm{I})}_{a,b}\) is given by
\[
  gZ=(AZ+B)(CZ+D)^{-1}.
\]

For each finite place \(p\) of \(\bbQ\),
we define \(K_{a,b,p}=\U_{a,b}(\bbQ_p)\cap \GL_N(\calO_{E_p})\).
We put \(K_{a,b,f}=\prod_{p<\infty}K_{a,b,p}\).
We also write
\[
  \GU_{a,b}(\bbZ_p)=\GU_{a,b}(\bbQ_p)\cap\GL_N(\calO_{E_p}).
\]

Let \(K_{a,b,\infty}\) be the stabilizer in \(\U(a,b)\) of the element
\[
  \bfi_{a,b} =
  \begin{pmatrix}
    \sqrt{-1}I_b \\
    0_{r,b}
  \end{pmatrix}
  \in\calH^{(\mathrm{I})}_{a,b}.
\]
Then \(K_{a,b,\infty}\) is a maximal compact subgroup of \(\U(a,b)\),
and is isomorphic to \(\U(a)\times \U(b)\),
where \(\U(m)=\left\{g\in\GL_m(\bbC) \middle|\ {}^t\overline{g}g=I_m\right\}\) for \(m\ge1\).

We put \(\frakg_{a,b}=\Lie(\U(a,b))\),
and \(\frakk_{a,b}=\Lie(K_{a,b,\infty})\).
We write
\[
  \frakg_{a,b} = \frakk_{a,b}\oplus \frakp_{a,b}
\]
for the Cartan decomposition.
After complexification,
we have
\[
  \frakp_{a,b}^{\bbC} = \frakp_{a,b}^+\oplus \frakp_{a,b}^-.
\]
Via the differential of the orbit map at \(\bfi_{a,b}\),
we identify \(\frakp_{a,b}^+\) with the holomorphic tangent space of \(\calH^{(\mathrm{I})}_{a,b}\) at \(\bfi_{a,b}\),
and \(\frakp_{a,b}^-\) with the antiholomorphic tangent space.

\subsubsection{Symplectic groups}

Let \(n\) be a positive integer.
We put
\[
  J_n =
  \begin{pmatrix}
    0    & I_n \\
    -I_n & 0
  \end{pmatrix}.
\]
The symplectic group \(\Sp_{n}\) is the algebraic group over \(\bbQ\) defined by
\[
  \Sp_{n}(R) = \left\{ g\in \GL_{2n}(R) \ \middle|\ {}^t g J_n g=J_n \right\}
\]
for any \(\bbQ\)-algebra \(R\).

For each finite place \(p\) of \(\bbQ\),
we define \(K_{n,p}=\Sp_{n}(\bbZ_p)\).
We put \(K_{n,f}=\prod_{p<\infty}K_{n,p}\).

Let
\[
  \calH^{(\typeIII)}_n = \left\{ Z\in \M_n(\bbC) \ \middle|\ {}^tZ=Z,\ \Im(Z)>0 \right\}
\]
be the Siegel upper half space.
The group \(\Sp_{n}(\bbR)\) acts on \(\calH^{(\typeIII)}_n\) by
\[
  \begin{pmatrix}
    A & B \\
    C & D
  \end{pmatrix}
  Z = (AZ+B)(CZ+D)^{-1}.
\]
Let \(K_{n,\infty}\) be the stabilizer in \(\Sp_{n}(\bbR)\) of the element \(\bfi_n=\sqrt{-1}I_n\in\calH^{(\typeIII)}_n\).
Then \(K_{n,\infty}\) is a maximal compact subgroup of \(\Sp_{n}(\bbR)\),
and is isomorphic to \(\U(n)\).

We put \(\frakg_n=\Lie(\Sp_{n}(\bbR))\),
\(\frakk_n=\Lie(K_{n,\infty})\).
We write
\[
  \frakg_n = \frakk_n\oplus \frakp_n
\]
for the Cartan decomposition.
After complexification,
we have
\[
  \frakp_n^{\bbC} = \frakp_n^+\oplus \frakp_n^-.
\]
Via the differential of the orbit map at \(\bfi_n\),
we identify \(\frakp_n^+\) with the holomorphic tangent space of \(\calH^{(\typeIII)}_n\) at \(\bfi_n\),
and \(\frakp_n^-\) with the antiholomorphic tangent space.

\subsection{Hermitian and Siegel automorphic forms}
\label{subsec:automorphic-forms}

Let \(G\) be either \(\Sp_{n}\) or one of the quasi-split unitary groups \(\U_{a,b}=G_N\) fixed above.
Accordingly,
we write
\[
  (K_{G,f},K_{G,\infty},\frakg_G,\frakp_G^-) =
  \begin{cases}
    (K_{n,f},K_{n,\infty},\frakg_n,\frakp_n^-)             & (G=\Sp_{n}),  \\
    (K_{a,b,f},K_{a,b,\infty},\frakg_{a,b},\frakp_{a,b}^-) & (G=\U_{a,b}).
  \end{cases}
\]
For each place \(v\) of \(\bbQ\),
we write \(K_{G,v}\) for the corresponding local factor of \(K_{G,f}K_{G,\infty}\).
We write \(K_{G,\infty}^{\bbC}\) for the complexification of \(K_{G,\infty}\).

\begin{dfn}
  Let \((\rho,V)\) be a polynomial representation of \(K_{G,\infty}^{\bbC}\).
  A holomorphic automorphic form on \(G(\bbA_\bbQ)\) of weight \((\rho,V)\) is a smooth \(V\)-valued function \(f:G(\bbA_\bbQ)\rightarrow V\) satisfying the following conditions:
  \begin{enumerate}
    \item For all \(\gamma\in G(\bbQ)\),
          \(g\in G(\bbA_\bbQ)\),
          and \(k=k_f k_\infty\in K_{G,f}K_{G,\infty}\),
          \[
            f(\gamma gk) = \rho(k_\infty)^{-1}f(g).
          \]
    \item \(f\) is annihilated by \(\frakp_G^-\) through right derivation.
    \item \(f\) is of moderate growth.
    \item \(f\) is \(Z(U(\frakg_G^{\bbC}))\)-finite,
          where \(Z(U(\frakg_G^{\bbC}))\) is the center of the universal enveloping algebra of \(\frakg_G^{\bbC}\).
  \end{enumerate}
  We denote the space of such forms by \(\calA_\rho(G)\).
  When \(G=\Sp_{n}\),
  we call such a form a Siegel automorphic form of degree \(n\) and weight \((\rho,V)\).
  When \(G=\U_{a,b}\),
  we call such a form a Hermitian automorphic form of degree \((a,b)\) and weight \((\rho,V)\).
\end{dfn}

\begin{dfn}
  An automorphic form \(f\in\calA_\rho(G)\) is called a cusp form if
  \[
    \int_{N(\bbQ)\backslash N(\bbA_\bbQ)} f(ng)\,dn = 0
  \]
  for every \(g\in G(\bbA_\bbQ)\) and every unipotent radical \(N\) of a proper parabolic subgroup of \(G\).
  We denote the space of cusp forms by \(\calS_\rho(G)\).
\end{dfn}

We now define normalized Fourier coefficients for Hermitian automorphic forms on \(\U_{n,n}\).
For \(X\in\Her_n(E)\),
put
\[
  \bfn(X)=
  \begin{pmatrix}
    I_n & X   \\
    0   & I_n
  \end{pmatrix}.
\]
The same formula defines an element \(\bfn(X)\in \U_{n,n}(\bbA_\bbQ)\) for \(X\in\Her_n(\bbA_E)\).
We fix the standard additive character
\[
  \bfe_{\bbA}(x)=\exp(2\pi\sqrt{-1}x_\infty) \prod_{\ell<\infty}\bfe_\ell(x_\ell)
\]
of \(\bbA_\bbQ/\bbQ\),
whose finite local component \(\bfe_\ell\) has conductor \(\bbZ_\ell\).
For \(f\in\calA_\rho(\U_{n,n})\),
\(S\in\Her_n(E)\),
and \(g_f\in\U_{n,n}(\bbA_{\bbQ,f})\),
we define the normalized \(S\)-th Fourier coefficient of \(f\) at \(g_f\) by
\[
  A_f(g_f,S)
  =
  \exp(2\pi\tr(S))
  \int_{\Her_n(E)\backslash\Her_n(\bbA_E)}
  f(\bfn(X)g_f)\,
  \overline{\bfe_{\bbA}(\tr(SX))}\,dX,
\]
where \(g_f\) is embedded in \(\U_{n,n}(\bbA_\bbQ)\) with trivial archimedean component,
and \(dX\) is the self-dual Haar measure with respect to the pairing \((S,X)\mapsto \bfe_{\bbA}(\tr(SX))\).
Under the usual identification with classical Hermitian modular forms,
if \(g_f\) is an ideal-class representative and \(f_{g_f}\) is the corresponding classical component, then
\[
  A_f(g_f,S)=a(S;f_{g_f}).
\]
Thus the coefficients used below at the finite adelic representatives are precisely the classical Fourier coefficients.
See \cite[Section~2.2]{Takeda2026Congruences} for the comparison of the adelic and classical normalizations.

For later use, if
\[
  g_\infty=
  \begin{pmatrix}
    A & B \\
    C & D
  \end{pmatrix}
  \in\U(n,n),
\]
we put
\[
  M_n(g_\infty)
  =
  \left(
  C(\sqrt{-1}I_n)+D,
  \overline C(\sqrt{-1}I_n)+\overline D
  \right).
\]

Fix a \(K_{G,\infty}\)-invariant Hermitian inner product \(\langle\, ,\,\rangle\) on the representation space \(V\) of \(\rho\).
For each place \(v\) of \(\bbQ\),
we normalize the Haar measure \(dg_v\) on \(G(\bbQ_v)\) by \(\vol(K_{G,v})=1\).
The restricted product gives a Haar measure \(dg=\prod_v dg_v\) on \(G(\bbA_\bbQ)\).

For \(f,h\in\calA_\rho(G)\),
the Petersson inner product is
\[
  (f,h)_G = \int_{G(\bbQ)\backslash G(\bbA_\bbQ)} \langle f(g),h(g)\rangle\,dg.
\]
If either \(f\) or \(h\) is cuspidal,
this integral converges absolutely.

The degree one Siegel case recovers the usual theory of elliptic modular forms.
More precisely,
we identify \(\calA_{{\det}^{2\nu}}(\Sp_1)\) with the space \(M_{2\nu}(\SL_2(\bbZ))\) of elliptic modular forms of weight \(2\nu\),
and \(\calS_{{\det}^{2\nu}}(\Sp_1)\) with the space \(S_{2\nu}(\SL_2(\bbZ))\) of elliptic cusp forms of weight \(2\nu\).
Under this identification,
a form \(f\in\calS_{{\det}^{2\nu}}(\Sp_1)\) is regarded as an element of \(S_{2\nu}(\SL_2(\bbZ))\).
We write its Fourier expansion at \(\infty\) as
\[
  f(z) = \sum_{m=1}^{\infty}a(m;f)q^m,
\]
where \(q=e^{2\pi\sqrt{-1}z}\).
When we say that \(f\) is a primitive elliptic cusp form,
we mean that \(f\) is a Hecke eigenform in \(S_{2\nu}(\SL_2(\bbZ))\) with \(a(1;f)=1\).

Put \(\sigma_\nu={\det}^{\nu}\boxtimes{\det}^{\nu}\).
To regard an elliptic cusp form as a Hermitian automorphic form on \(\U_{1,1}\),
choose integral ideals \(\fraka_1=\calO_E,\ldots,\fraka_{h_E}\) representing the ideal class group \(\Cl(E)\) of \(E\).
Put
\[
  r_i=
  \begin{pmatrix}
    \sqrt{N_{E/\bbQ}(\fraka_i)} & 0                             \\
    0                           & 1/\sqrt{N_{E/\bbQ}(\fraka_i)}
  \end{pmatrix}.
\]
For \(f\in S_{2\nu}(\SL_2(\bbZ))\),
define
\[
  f_i=f\big|_{2\nu}r_i^{-1}.
\]
The adelic--classical correspondence in \cite[Section~2.2]{Takeda2026Congruences} gives an injective map
\begin{align*}
  \iota_E:S_{2\nu}(\SL_2(\bbZ)) & \rightarrow \calS_{\sigma_\nu}(\U_{1,1}),      \\
  f                             & \mapsto\iota_E(f)=(f_1,\ldots,f_{h_E})^\sharp.
\end{align*}
Here \((f_1,\ldots,f_{h_E})^\sharp\) denotes the adelization associated with the tuple \((f_1,\ldots,f_{h_E})\).
We write \(f\) also for \(\iota_E(f)\) when no confusion can arise.

\subsection{\texorpdfstring{Algebraic \(L\)-values}{Algebraic L-values}}
\label{subsec:l-values}

We first fix the local \(L\)-factor normalizations used throughout the paper,
and then recall the algebraicity results needed in the congruence and Selmer arguments.

Let \(\bsk=(k_1,\ldots,k_n)\) be a dominant integral weight with \(k_1\ge\cdots\ge k_n\ge0\),
and let \(\rho_{(n,\bsk)}\) denote the irreducible polynomial representation of \(\GL_n(\bbC)\) of highest weight \(\bsk\).
Let \(F\in \calS_{\rho_{(n,\bsk)}}(\Sp_{n})\) be a Hecke cusp eigenform.
For a prime number \(p\),
let \(\alpha_0(p),\alpha_1(p),\ldots,\alpha_n(p)\) be the \(p\)-Satake parameters of \(F\),
normalized so that
\[
  \alpha_0(p)^2\alpha_1(p)\cdots\alpha_n(p) = p^{k_1+\cdots+k_n-n(n+1)/2}.
\]
For a subset \(I\subset\{1,\ldots,n\}\),
we put \(\alpha_I(p)=\prod_{i\in I}\alpha_i(p)\),
with the convention \(\alpha_\emptyset(p)=1\).
We define the local spinor \(L\)-polynomial by
\[
  L_p(X,F,\Spin) = \prod_{I\subset\{1,\ldots,n\}} \left(1-\alpha_0(p)\alpha_I(p)X\right),
\]
and the spinor \(L\)-function is given by
\[
  L(s,F,\Spin) = \prod_p L_p(p^{-s},F,\Spin)^{-1}.
\]
For a Dirichlet character \(\chi\),
we define
\[
  L(s,F,\Spin\otimes\chi) = \prod_p L_p(\chi(p)p^{-s},F,\Spin)^{-1},
\]
where we use the convention \(\chi(p)=0\) if \(p\) divides the conductor of \(\chi\).

For a primitive elliptic cusp form \(f=\sum_{m=1}^\infty a(m;f)q^m\in S_{2\nu}(\SL_2(\bbZ))\),
we define the local \(L\)-polynomial by
\[
  L_p(X,f)=1-a(p;f)X+p^{2\nu-1}X^2,
\]
and define
\[
  L(s,f)=\prod_p L_p(p^{-s},f)^{-1}.
\]

We shall also use the standard \(L\)-function of Hermitian Hecke eigenforms.
Let \(F \in \calA_\rho(\U_{n,n})\) be a Hecke eigenform.
To fix its normalization,
we first specify the representation of the \(L\)-group that defines the local standard factor.
Throughout this paper,
\(\St\) denotes the composition
\[
  \St: \leftexp{L}{\U_{n,n}} \xrightarrow{\ \xi_{\mathbf{1}}\ } \leftexp{L}{(\Res_{E/\bbQ}\GL_{2n})} =
  (\GL_{2n}(\bbC)\times\GL_{2n}(\bbC)) \rtimes W_{\bbQ} \rightarrow \GL_{4n}(\bbC),
\]
where the action of \(W_{\bbQ}\) factors through \(\Gal(E/\bbQ)\).
Here,
\(\xi_{\mathbf{1}}\) is the standard base-change \(L\)-embedding associated with the trivial auxiliary character and is recalled explicitly in \eqref{eq:xi},
and the last arrow restricts on \(\GL_{2n}(\bbC)\times\GL_{2n}(\bbC)\) to the direct sum of the two standard representations.

For each prime \(p\) at which \(F\) is spherical,
let \(s_p(F)\) denote its Satake parameter,
viewed as a semisimple conjugacy class in the \(L\)-group,
with the normalization of \cite[Section~3]{Takeda2026Congruences}.
We define
\[
  L_p(X,F,\St) = \det\left(1-\St(s_p(F))X\right).
\]

Choose representatives of \(s_p(F)\) as follows.
If \(p\) splits in \(E\),
write the corresponding parameters as \(\alpha_1(p),\ldots,\alpha_{2n}(p)\).
If \(p\) is non-split in \(E\),
write them as \(\alpha_1(p),\ldots,\alpha_n(p)\).

With these conventions,
the same local standard \(L\)-polynomial is given explicitly by
\[
  L_p(X,F,\St) =
  \begin{cases}
    \displaystyle \prod_{i=1}^{2n} (1-\alpha_i(p)X)(1-\alpha_i(p)^{-1}X)   & (p \text{ is split in } E),    \\
    \displaystyle \prod_{i=1}^{n}(1-\alpha_i(p)X^2)(1-\alpha_i(p)^{-1}X^2) & (p \text{ is inert in } E),    \\
    \displaystyle \prod_{i=1}^{n} (1-\alpha_i(p)X)(1-\alpha_i(p)^{-1}X)    & (p \text{ is ramified in } E).
  \end{cases}
\]
Since the forms considered here are right \(K_{n,n,f}\)-invariant,
\(F\) is spherical at every finite prime.
The global standard \(L\)-function is then given by the Euler product
\[
  L(s,F,\St) = \prod_p L_p(p^{-s},F,\St)^{-1}.
\]

We recall the algebraicity results for the \(L\)-functions used below.
Let \(h\) be a normalized newform of weight \(2\nu\),
level \(N_h\),
and nebentypus \(\psi_h\).
Let \(\bbQ(h)\) be its Hecke field.
Let \(K\) be a number field containing \(\bbQ(h)\),
the values of all Dirichlet characters which occur below,
and all period quotients which occur below.
Let \(\frakP\) be a prime of \(K\) above \(p\).
We assume that \(p\) does not divide \(N_h\).

We first recall the algebraicity theorem of Shimura.
Fix Shimura periods \(u^\pm(h)\) as in \cite{Shimura1976SpecialValues,Shimura1977Periods} so that
\[
  u^+(h)u^-(h)=(h,h)_{\Gamma_0(N_h)}.
\]
For a primitive Dirichlet character \(\eta\),
we denote by \(G(\eta)\) its Gauss sum.
For an integer \(s\) with \(1\le s\le 2\nu-1\),
put \(\alpha(s,\eta)=(-1)^s\eta(-1)\).
We put
\[
  L^{\sh}(s, h\otimes\eta) = \frac{L(s,h\otimes\eta)}{(2\pi i)^sG(\eta)u^{\alpha(s,\eta)}(h)}.
\]
Here \(u^{\alpha(s,\eta)}(h)\) means \(u^+(h)\) if \(\alpha(s,\eta)=1\),
and \(u^-(h)\) if \(\alpha(s,\eta)=-1\).

\begin{prop}[\cite{Shimura1976SpecialValues,Shimura1977Periods}]
  \label{prop:shimura-alg}
  Let \(\eta\) be a primitive Dirichlet character,
  and let \(s\) be an integer with \(1\le s\le 2\nu-1\).
  Then \(L^{\sh}(s,h\otimes\eta)\in\bbQ(h)(\eta)\).
\end{prop}

We now replace the Shimura periods by the Kato periods.
Let \(\operatorname{per}_h\) denote the de Rham--Betti comparison isomorphism for the motive attached to \(h\),
after extension of scalars to \(\bbC\), with the normalization of \cite{Kato2004HodgeZeta}.
Write \(\Omega_B^\pm(h)\) for the periods defined by
\[
  \operatorname{per}_h(\omega_h)^\pm=\Omega_B^\pm(h)\gamma_h^\pm,
\]
where the superscripts \(\pm\) denote projection onto the eigenspaces of complex conjugation,
and \(\omega_h\) and \(\gamma_h^\pm\) are generators of the corresponding integral de Rham and Betti lattices in Kato's construction
\cite[Proposition~14.21 and \S~14.22]{Kato2004HodgeZeta}.
Following the convention of \cite{Atobe2026HarderII}, we put
\[
  \Omega^\pm(h)=(2\pi i)^{1-2\nu}\Omega_B^\mp(h).
\]
These periods are determined up to \(\frakP\)-adic units.
The ratios \(\Omega^\pm(h)/u^\pm(h)\) belong to \(K^\times\).
We therefore define
\[
  L^{\alg}(s,h\otimes\eta) :=\frac{L(s,h\otimes\eta)}{(2\pi i)^sG(\eta)\Omega^{\alpha(s,\eta)}(h)}\in K.
\]

The following fact about the Kato periods is known (see,
for example,
\cite[Lemma~9.1]{Atobe2026HarderII}).

\begin{lem}
  \label{lem:kato-periods}
  Let \(h\) be a primitive elliptic cusp form in \(S_{2\nu}(\SL_2(\bbZ))\).
  Fix a prime \(p\) satisfying \(p>2\nu-2\).
  Let \(\frakp\) be a prime ideal of the ring of integers of \(\bbQ(h)\) dividing \(p\).
  Assume that \(h\) is not congruent modulo \(\frakp\) to any other primitive elliptic cusp form in \(S_{2\nu}(\SL_2(\bbZ))\).
  Then the product of the Kato periods \(\Omega^+(h)\Omega^-(h)\) coincides with the Petersson inner product \((h,h)\) up to a \(\frakp\)-adic unit.
\end{lem}

We also record the compatibility of these periods with quadratic twisting.

\begin{lem}
  \label{lem:kato-twist-periods}
  Let \(h\in S_{2\nu}(\SL_2(\bbZ))\) be a primitive elliptic cusp form,
  and let \(\eta\) be a primitive quadratic Dirichlet character of conductor \(d\).
  Put \(g=h\otimes\eta\).
  Assume that \(p>2\nu\),
  that \(p\) does not divide \(d\),
  and that the residual Galois representation attached to \(h\) at \(\frakP\) is absolutely irreducible.
  Use the integral periods with the convention fixed above.
  Then there are \(u_\epsilon\in\calO_\frakP^\times\),
  for \(\epsilon=\pm\),
  such that
  \[
    \Omega^\epsilon(g)
    =u_\epsilon G(\eta)\Omega^{\epsilon\eta(-1)}(h).
  \]
  Consequently,
  for \(1\le s\le 2\nu-1\),
  \[
    L^{\alg}(s,h\otimes\eta)
    =u_{(-1)^s}\frac{L(s,g)}{(2\pi i)^s\Omega^{(-1)^s}(g)}.
  \]
\end{lem}

\begin{proof}
  Since \(g=h\otimes\eta\),
  the associated \(p\)-adic Galois representations satisfy
  \[
    V_g\simeq V_h\otimes\eta.
  \]
  By residual absolute irreducibility, the corresponding \(\Gal_\bbQ\)-stable lattices are homothetic.
  Since \(\eta\) is unramified at \(p\) and \(p>2\nu\),
  tensor compatibility and full faithfulness in the Fontaine--Laffaille category \cite[(14.17.2)]{Kato2004HodgeZeta} identify the corresponding integral de Rham structures up to the same homothety.

  For the Artin motive attached to \(\eta\),
  choose generators \(e_\eta\) and \(b_\eta\) of the corresponding \(\calO_\frakP\)-integral de Rham and Betti lattices such that
  \[
    \operatorname{per}_\eta(e_\eta)=G(\eta)b_\eta,
    \qquad
    c(b_\eta)=\eta(-1)b_\eta.
  \]
  Such a choice is possible since \(p\) does not divide the conductor of \(\eta\),
  so \(G(\eta)\) is a \(\frakP\)-adic unit.

  Tensor compatibility of the comparison maps then gives,
  for each
  \(\epsilon=\pm\),
  an element \(u_\epsilon\in\calO_\frakP^\times\) such that
  \[
    \Omega^\epsilon(g)
    =
    u_\epsilon G(\eta)
    \Omega^{\epsilon\eta(-1)}(h).
  \]
  The homothety factor is the same on the de Rham and Betti sides and therefore is absorbed into \(u_\epsilon\).
  The asserted identity for \(L^{\alg}\) follows immediately from its definition.
\end{proof}

We shall also use the following algebraicity result for the standard \(L\)-function of a Hermitian cusp form.
For \(\U_{n,n}\),
we identify \(K_{n,n,\infty}^{\bbC}\) with \(\GL_n(\bbC)\times\GL_n(\bbC)\) in the usual way.
By a dominant integral weight \((k_1,\ldots,k_n;l_1,\ldots,l_n)\),
we mean that \(k_1 \ge \cdots \ge k_n \ge 0\),
\(l_1 \ge \cdots \ge l_n \ge 0\),
\(k_i,l_i \in\bbZ\).
We use the same symbol \((k_1,\ldots,k_n;l_1,\ldots,l_n)\) for the irreducible polynomial representation of \(\GL_n(\bbC)\times\GL_n(\bbC)\) whose two highest weights are \((k_1,\ldots,k_n)\) and \((l_1,\ldots,l_n)\).

\begin{prop}[{\cite[Corollary~5.14]{Takeda2026Congruences}}]
  \label{prop:std-alg}
  Let \(F\) be a Hermitian Hecke cusp eigenform of degree \(n\) and weight given by a dominant integral weight \((k_1,\ldots,k_n;l_1,\ldots,l_n)\),
  and let \(\bbQ(F)\) denote its Hecke field.
  Let \(s\) be an integer with \(0\le s\le (k_n+l_n)/2-n\).
  Then
  \[
    L^{\alg}(s+1/2,F,\St) := \frac{L(s+1/2,F,\St)}{|D_E|^{n/2}\cdot\pi^{\sum_{i=1}^n(k_i+l_i)+2n(s+1/2)-n}\cdot(F,F)_{\U_{n,n}}} \in \bbQ(F).
  \]
\end{prop}

\begin{rem}
  \label{rem:degree-one-alg}
  Let \(h\in S_{2\nu}(\SL_2(\bbZ))\) be a primitive elliptic cusp form,
  let \(p\) be a rational prime,
  let \(\frakp\) be a prime ideal of \(\bbQ(h)\) above \(p\),
  and let \(s_0\) be an integer with \(0\le s_0\le\nu-1\).
  Then we have
  \[
    L(s_0+\tfrac{1}{2},\iota_E(h),\St) = L(s_0+\nu,h) L(s_0+\nu,h\otimes\chi_E).
  \]
  Since the adelic quotient has \(h_E\) components and each change of variables by \(r_i\) preserves the classical Petersson norm,
  \[
    (\iota_E(h),\iota_E(h))_{\U_{1,1}} =h_E(h,h)_{\SL_2(\bbZ)}.
  \]
  If \(p\) does not divide \(D_Eh_E\) and the assumptions of Lemma~\ref{lem:kato-periods} hold,
  there is a \(\frakp\)-adic unit \(u\) such that
  \[
    L^{\alg}(s_0+\tfrac{1}{2},\iota_E(h),\St) =u\cdot L^{\alg}(s_0+\nu,h) L^{\alg}(s_0+\nu,h\otimes\chi_E).
  \]
\end{rem}

\subsection{Harder's conjecture}
\label{subsec:harder-cong}

We recall Harder's conjecture in the form used in this paper.
The original congruence was proposed by Harder in \cite{Harder2008Congruence}.
The form below is the corresponding congruence for the spinor \(L\)-polynomials,
as in \cite[Conjecture~3.7]{Atobe2023HarderI}.

Let \(k\ge 3\) and \(j\ge 2\) be integers with \(j\) even.
Let
\[
  f=\sum_{m=1}^{\infty}a(m;f)q^m \in S_{2k+j-2}(\SL_2(\bbZ))
\]
be a primitive elliptic cusp form.
Let \(\frakp\) be a prime ideal of \(\bbQ(f)\).

For a Hecke eigenform \(F\in \calS_{{\det}^k\Sym^j}(\Sp_2)\) and a rational prime \(\ell\),
let \(T(\ell)\) denote the standard degree-two Siegel Hecke operator at \(\ell\) (see \cite[Section~3]{Atobe2023HarderI}),
and write \(\lambda_F(T(\ell))\) for the \(T(\ell)\)-eigenvalue of \(F\).

\begin{conj}[Harder's conjecture]
  \label{conj:harder}
  Assume that \(p_\frakp>2k+j-2\) and \(\frakp\mid L^{\alg}(k+j,f)\).
  Then there exist a Hecke cusp eigenform \(F\in \calS_{{\det}^k\Sym^j}(\Sp_2)\) and a prime ideal \(\frakP\) of \(\bbQ(f)\bbQ(F)\) lying above \(\frakp\) such that,
  for every prime number \(\ell\),
  we have
  \[
    L_\ell(X,F,\Spin) \equiv L_\ell(X,f) \left(1-\ell^{k-2}X\right) \left(1-\ell^{k+j-1}X\right) \pmod{\frakP}.
  \]
\end{conj}

We call the congruence in Conjecture~\ref{conj:harder} a Harder-type congruence for \(f\).

Conjecture~\ref{conj:harder} implies,
in particular,
\[
  \lambda_F(T(\ell)) \equiv a(\ell;f)+\ell^{k-2}+\ell^{k+j-1} \pmod{\frakP}
\]
for every prime number \(\ell\).
This is the Hecke-eigenvalue form of the Harder-type congruence.

\begin{rem}
  In the scalar-valued case \(j=0\) with \(k\) even,
  the congruence in Conjecture~\ref{conj:harder} can be made an equality for a special choice of \(F\).
  As noted in \cite[Remark~3.8(4)]{Atobe2023HarderI},
  this follows from the Saito--Kurokawa lift.
  More precisely,
  if \(F\in\calS_{{\det}^{k}}(\Sp_2)\) is the Saito--Kurokawa lift of the primitive elliptic cusp form \(f\in S_{2k-2}(\SL_2(\bbZ))\),
  then its spinor \(L\)-polynomial satisfies
  \[
    L_\ell(X,F,\Spin) = L_\ell(X,f)(1-\ell^{k-1}X)(1-\ell^{k-2}X)
  \]
  for every prime \(\ell\).
\end{rem}

\section{Hermitian spin lifts}
\label{sec:spin-lifts}
\subsection{Endoscopic classification for unitary groups}
\label{subsec:mok-class}

Retain the notation \(G_N=\U_{a_N,b_N}\) for the quasi-split unitary group in \(N\) variables fixed in Subsection~\ref{subsec:groups}.
For each finite prime \(p\),
define \(K_{N,p}=K_{a_N,b_N,p}\) and \(K_{N,f}=\prod_{p<\infty}K_{N,p}\).

We fix unitary Hecke characters \(\chi_{\pm}:\bbA_E^\times/E^\times \rightarrow \bbC^\times\) such that \(\chi_+|_{\bbA_{\bbQ}^\times}=\mathbf{1}_{\bbA_{\bbQ}^\times}\),
and \(\chi_-|_{\bbA_{\bbQ}^\times}=\chi_E\).
Here \(\mathbf{1}_{\bbA_{\bbQ}^\times}\) is the trivial character,
and \(\chi_E\) is the quadratic character associated with \(E/\bbQ\).

Although \(\chi_+\) can in fact be chosen arbitrarily as long as it satisfies the above conditions,
for the sake of simplicity in the following discussion,
we set \(\chi_+\) to the trivial character \(\mathbf{1}_{\bbA_E^\times}\).

The classification due to Mok \cite{Mok2015Endoscopic} describes the discrete automorphic spectrum of \(G_N\) in terms of conjugate self-dual Arthur parameters for \(\GL_N(\bbA_E)\),
together with the \(L\)-embedding determined by \(\chi_\kappa\),
where \(\kappa\in\{\pm\}\).
We regard \(\chi_\kappa\) as a character of \(W_E\) by global class field theory.

Choose \(w_c\in W_{\bbQ}\setminus W_E\) such that \(\chi_\pm(w_c^2)=\pm1\).
Then we define the \(L\)-embedding
\[
  \xi_{\chi_\kappa}: \leftexp{L}{G_N} \rightarrow \leftexp{L}{(\Res_{E/\bbQ}\GL_N)} = (\GL_N(\bbC)\times \GL_N(\bbC))\rtimes W_{\bbQ}
\]
by
\begin{equation}\label{eq:xi}
  \begin{cases}
    g\rtimes 1 \mapsto (g,{}^tg^{-1})\rtimes 1,                                                                      \\
    I_N\rtimes \gamma \mapsto (\chi_\kappa(\gamma)I_N, \chi_\kappa(\gamma)^{-1}I_N)\rtimes \gamma & (\gamma\in W_E), \\
    I_N\rtimes w_c \mapsto (\kappa \Phi_N,\Phi_N^{-1})\rtimes w_c.
  \end{cases}
\end{equation}
Here the action of \(W_{\bbQ}\) on the dual group factors through \(\Gal(E/\bbQ)\),
and \(\Phi_N\) is the anti-diagonal matrix defined above.
We follow Mok \cite[\S\S~2.3--2.4]{Mok2015Endoscopic} for global parameters.
Let
\[
  \psi^N = \boxplus_{i=1}^r\mu_i\boxtimes\nu(d_i), \qquad \sum_{i=1}^r m_id_i=N,
\]
be an element of the set \(\widetilde{\Psi}_{\mathrm{ell}}(N)\) of elliptic formal global parameters for \(\GL_N(\bbA_E)\).
Here \(\mu_i\) is a conjugate self-dual unitary cuspidal automorphic representation of \(\GL_{m_i}(\bbA_E)\),
\(\nu(d_i)\) is the \(d_i\)-dimensional irreducible algebraic representation of \(\SL_2(\bbC)\),
and the simple summands are mutually distinct.

Mok associates to \(\psi^N\) a group \(\calL_{\psi^N}\) and an \(L\)-homomorphism
\[
  \widetilde{\psi}^{\,N} : \calL_{\psi^N}\times\SL_2(\bbC) \rightarrow \leftexp{L}{(\Res_{E/\bbQ}\GL_N)}.
\]
An element \(\psi\in\Psi_2(G_N,\xi_{\chi_\kappa})\) is represented by a pair \(\psi=(\psi^N,\widetilde{\psi})\),
where
\[
  \widetilde{\psi}: \calL_{\psi^N}\times\SL_2(\bbC) \rightarrow \leftexp{L}{G_N}
\]
is an \(L\)-homomorphism,
considered up to \(\widehat{G}_N\)-conjugacy,
satisfying
\begin{equation}\label{eq:mok-factorization}
  \widetilde{\psi}^{\,N} = \xi_{\chi_\kappa}\circ\widetilde{\psi}.
\end{equation}
This is the definition of \(\Psi_2(G_N,\xi_{\chi_\kappa})\);
see \cite[Definitions~2.4.3, 2.4.5, and~2.4.7]{Mok2015Endoscopic}.
For the datum considered here,
\(\widetilde{\psi}\) is determined by \(\psi^N\),
so we identify \(\psi\) with \(\psi^N\) when no confusion can arise.
We call \(\psi\) simple if \(r=1\),
and generic if \(d_i=1\) for every \(i\).
Although \(\nu(1)\) is the trivial representation of \(\SL_2(\bbC)\),
we retain the factor \(\nu(1)\) in the notation for a generic Arthur parameter in order to make the triviality of its Arthur \(\SL_2\)-factor explicit.

Since \(\chi_+=\mathbf1\),
we call \(\xi_{\chi_+}\) the standard \(L\)-embedding and \(\psi^N\) the standard parameter.
For each place \(v\) of \(\bbQ\),
write \(\psi_v\) and \(\psi_v^N\) for the localizations of \(\psi\) and \(\psi^N\).
Then
\[
  \psi_v^N \sim \xi_{\chi_\kappa,v}\circ\psi_v,
\]
where \(\sim\) denotes equivalence of local parameters (see \cite[\S~2.2]{Mok2015Endoscopic}).

The factorization condition \eqref{eq:mok-factorization} can be checked on each simple summand as follows.

Let \(\mu\) be a conjugate self-dual cuspidal automorphic representation of \(\GL_m(\bbA_E)\).
There are two Asai \(L\)-functions,
\(L(s,\mu,\Asai^+)\) and \(L(s,\mu,\Asai^-)\).
Exactly one of them has a pole at \(s=1\) \cite[Theorem~2.5.4(a)]{Mok2015Endoscopic}.
We denote by \(\eta(\mu)\in\{\pm1\}\) the sign characterized by the condition that \(L(s,\mu,\Asai^{\eta(\mu)})\) has a pole at \(s=1\).

We define \(\delta(\mu)\in\{\pm1\}\) by \(\eta(\mu)=(-1)^{m-1}\delta(\mu)\).
For the \(i\)-th summand of \(\psi^N\) write
\begin{align*}
  \eta_i   & = \eta(\mu_i),   \\
  \delta_i & = \delta(\mu_i), \\
  N_i      & = m_i d_i.
\end{align*}

Following Mok \cite[Remark~2.4.6 and Definition~2.4.7]{Mok2015Endoscopic},
define \(\kappa_i=\delta_i(-1)^{N_i-m_i-d_i+1}\).
Then \eqref{eq:mok-factorization} holds if and only if
\[
  \kappa_i(-1)^{N_i-1}=\kappa(-1)^{N-1}
\]
for \(1\le i\le r\).
This condition can be rewritten as
\begin{equation}\label{eq:mok-sign}
  \delta_i = \kappa(-1)^{N-1+m_i+d_i}.
\end{equation}

For each \(\psi\in\Psi_2(G_N,\xi_{\chi_\kappa})\),
Mok attaches local Arthur packets \(\Pi_{\psi_v}(G_N(\bbQ_v))\) and the global packet
\[
  \Pi_\psi = \left\{ \pi={\bigotimes_v}' \pi_v \;\middle|\; \pi_v\in \Pi_{\psi_v}(G_N(\bbQ_v)) \right\}.
\]
Let \(\psi=(\psi^N,\widetilde{\psi})\) be an element of \(\Psi_2(G_N,\xi_{\chi_\kappa})\).
Its centralizer in the dual group is
\[
  S_\psi = \Cent_{\widehat{G}_N}(\Im\widetilde{\psi}).
\]
The global component group attached to \(\psi\) is
\[
  \calS_\psi = S_\psi/S_\psi^\circ Z(\widehat{G}_N)^\Gamma,
\]
as in \cite[Definition~2.4.8]{Mok2015Endoscopic},
where \(S_\psi^\circ\) is the identity component of \(S_\psi\),
and \(Z(\widehat{G}_N)^\Gamma\) is the group of Galois-invariant elements in the center of \(\widehat{G}_N\),
with \(\Gamma=\Gal(\overline{\bbQ}/\bbQ)\).
We denote by \(\epsilon_\psi:\calS_\psi\rightarrow \{\pm1\}\) the canonical sign character (\cite[Section~1.5]{Arthur2013Endoscopic}).

For each place \(v\) of \(\bbQ\),
let \(\psi_v\) be the localization of \(\psi\),
and define the local component group
\[
  \calS_{\psi_v} = S_{\psi_v}/S_{\psi_v}^\circ Z(\widehat{G}_N)^{\Gamma_v},
\]
where \(\Gamma_v=\Gal(\overline{\bbQ}_v/\bbQ_v)\).
Localization induces a natural map \(\calS_\psi\rightarrow\calS_{\psi_v}\).

Mok \cite[\S~2.5]{Mok2015Endoscopic} attaches to \(\psi_v\) a pairing
\[
  \langle\,\cdot\,,\,\cdot\,\rangle_v: \calS_{\psi_v}\times\Pi_{\psi_v}\rightarrow\{\pm1\}.
\]
Thus each \(\pi_v\in\Pi_{\psi_v}\) determines a character \(\langle\,\cdot\,,\pi_v\rangle_v:\calS_{\psi_v}\rightarrow\{\pm1\}\).
For \(\pi\in\Pi_\psi\), the product of these local characters defines \(\langle\,\cdot\,,\pi\rangle_\psi:\calS_\psi\rightarrow\{\pm1\}\).

\begin{prop}[{\cite[Theorem~2.5.2]{Mok2015Endoscopic}}]
  \label{prop:mok-global-multiplicity}
  For \(\psi\in\Psi_2(G_N,\xi_{\chi_\kappa})\) define
  \[
    \Pi_\psi(\epsilon_\psi) = \left\{ \pi\in\Pi_\psi \;\middle|\; \langle\,\cdot\,,\pi\rangle_\psi = \epsilon_\psi\right\}.
  \]
  The multiplicity formula is
  \[
    L^2_{\disc} \left( G_N(\bbQ)\backslash G_N(\bbA_\bbQ) \right) \simeq \bigoplus_{\psi\in\Psi_2(G_N,\xi_{\chi_\kappa})} \; \bigoplus_{\pi\in\Pi_\psi(\epsilon_\psi)} \pi.
  \]
\end{prop}

We recall the local spherical property.
We say that \(\psi_p\) is spherical if \(\Pi_{\psi_p}\) contains a \(K_{N,p}\)-spherical representation (see \cite[\S~7.1]{Mok2015Endoscopic}).
\begin{lem}[{\cite[Theorem~2.5.1 (a),\S~7.6]{Mok2015Endoscopic}, \cite{Shahidi1981Certain,Shahidi1990Plancherel}}]
  \label{lem:sphericalApacket}
  Let \(p\) be a finite prime,
  and let \(\psi_p\) be a local Arthur parameter for \(G_N(\bbQ_p)\).
  Assume that \(\Pi_{\psi_p}\) contains a \(K_{N,p}\)-spherical representation \(\pi_p^{\sph}\).
  Then \(\pi_p^{\sph}\) is the unique \(K_{N,p}\)-spherical member of \(\Pi_{\psi_p}\),
  and its character on the local component group is trivial.
\end{lem}

\subsection{Construction of the Hermitian spin lift}
\label{subsec:spin-lift}

Let \(k\ge4\) and let \(j\ge2\) be even.
Recall that
\[
  \widehat{\PGSp_2}=\Spin_5(\bbC)\simeq\Sp_2(\bbC),
\]
and that the four-dimensional spin representation is identified with the natural representation of \(\Sp_2(\bbC)\).
We use the \(\GL_4\)-representation associated with \(F\) via the four-dimensional spin representation;
see \cite[Proposition~9.1.4]{Chenevier2019Automorphic}.

For simplicity,
we put
\[
  \tau_{k,j} = {\det}^{j/2+2}\Sym^{k-3}\boxtimes{\det}^{j/2+2}\Sym^{k-3}.
\]
\begin{prop}\label{prop:spin-lift}
  For a Hecke cusp eigenform \(F\in\calS_{{\det}^{k}\Sym^{j}}(\Sp_2)\),
  there exists a Hecke cusp eigenform \(\Lift_E(F)\in\calS_{\tau_{k,j}}(\U_{2,2})\) such that
  \[
    L(s,\Lift_E(F),\St) = L\left(s+\frac{2k+j-3}{2},F,\Spin\right) L\left(s+\frac{2k+j-3}{2},F,\Spin\otimes\chi_E\right).
  \]
\end{prop}
\begin{proof}
  By \cite[Proposition~9.1.4]{Chenevier2019Automorphic},
  there exists an irreducible unitary cuspidal self-dual automorphic representation \(\Pi_F^{\Spin}\) of \(\PGL_4(\bbA_\bbQ)\),
  which we regard as a representation of \(\GL_4(\bbA_\bbQ)\) with trivial central character,
  such that
  \[
    L(s,\Pi_F^{\Spin},\St)
    =
    L\left(s+\frac{2k+j-3}{2},F,\Spin\right).
  \]
  Moreover,
  the eigenvalues of the infinitesimal character of \(\Pi_{F,\infty}^{\Spin}\) are \(\left\{\pm\frac{2k+j-3}{2},\pm\frac{j+1}{2}\right\}\).

  Put
  \[
    \Pi_{F,E}^{\Spin}
    =
    \BC_{E/\bbQ}(\Pi_F^{\Spin}).
  \]
  Then, \(\Pi_{F,E}^{\Spin}\) is cuspidal.
  Otherwise,
  the cuspidality criterion for quadratic base change \cite[Chapter~3]{Arthur1989BaseChange} would give
  \[
    \Pi_F^{\Spin}
    \simeq
    \Pi_F^{\Spin}\otimes\chi_E.
  \]
  Let \(q\mid D_E\).
  Since \(F\) has level one,
  \(\Pi_{F,q}^{\Spin}\) is unramified,
  whereas \(\Pi_{F,q}^{\Spin}\otimes\chi_{E,q}\) is ramified.
  This is a contradiction.

  We next determine the sign of \(\Pi_{F,E}^{\Spin}\).
  The four-dimensional spin representation is symplectic,
  so \(\Pi_F^{\Spin}\) is self-dual of symplectic type.
  Hence \(\Pi_{F,E}^{\Spin}\) is conjugate self-dual.
  For a quadratic base change,
  \[
    L(s,\Pi_{F,E}^{\Spin},\Asai^-)
    =
    L(s,\Pi_F^{\Spin},\wedge^2)
    L(s,\Pi_F^{\Spin},\Sym^2\otimes\chi_E).
  \]
  Since \(\Pi_F^{\Spin}\) is of symplectic type,
  the first factor has a pole at \(s=1\).
  Moreover,
  since \(\Pi_F^{\Spin}\not\simeq\Pi_F^{\Spin}\otimes\chi_E\),
  the standard pole criterion shows that
  \(L(s,\Pi_F^{\Spin},\Sym^2\otimes\chi_E)\) is holomorphic at \(s=1\),
  while its non-vanishing there follows from \cite[Theorem~1.1]{Shahidi1997Nonvanishing}.
  Thus
  \[
    \eta(\Pi_{F,E}^{\Spin})=-1.
  \]
  Since \(m=4\),
  we have
  \[
    \delta(\Pi_{F,E}^{\Spin})=+1.
  \]
  Applying \eqref{eq:mok-sign} with \(N=m=4\) and \(d=1\),
  we obtain \(\kappa=+\).

  Now put
  \[
    \psi_{F,E}^{N}
    =
    \Pi_{F,E}^{\Spin}\boxtimes\nu(1).
  \]
  This is a simple generic global parameter of degree \(4\) on the
  general linear side.
  Since
  \(\delta(\Pi_{F,E}^{\Spin})=+1\),
  it is compatible with \(\xi_{\chi_+}\),
  and hence determines a global Arthur parameter
  \[
    \psi_{F,E}
    \in
    \Psi_2(\U_{2,2},\xi_{\chi_+})
  \]
  whose associated standard parameter is \(\psi_{F,E}^{N}\).

  We choose a member of the corresponding global packet.
  At the archimedean place,
  the infinitesimal character of \(\psi_{F,E,\infty}\) has the eigenvalues recorded above.
  The corresponding local packet \(\Pi_{\psi_{F,E,\infty}}\) contains a holomorphic discrete series representation
  \[
    \pi_{U,\infty}^{\hol}\in\Pi_{\psi_{F,E,\infty}}
  \]
  whose lowest \(K_{2,2,\infty}\)-type has highest weight
  \[
    \left(
    \frac{2k+j}{2}-1,\,
    \frac{j}{2}+2;\,
    \frac{2k+j}{2}-1,\,
    \frac{j}{2}+2
    \right).
  \]
  This is precisely the highest weight defining \(\tau_{k,j}\).

  Let \(p\) be a finite prime.
  Since \(F\) has full level,
  \(\Pi_{F,p}^{\Spin}\) is unramified,
  and hence its local base change to the places of \(E\) above \(p\) is unramified.
  Thus the standard parameter \(\psi_{F,E,p}^{N}\) is unramified on \(W_{E_p}\).
  By the local description of spherical parameters for quasi-split unitary groups in \cite[\S~7.1]{Mok2015Endoscopic},
  the corresponding parameter \(\psi_{F,E,p}\) is induced from an unramified character of the standard minimal Levi \(M_0(\bbQ_p)\).
  In particular,
  \(\Pi_{\psi_{F,E,p}}\) contains a \(K_{2,2,p}\)-spherical member.
  By Lemma~\ref{lem:sphericalApacket},
  this member is unique and its local packet character is trivial.
  We denote it by
  \[
    \pi_{U,p}^{\sph}\in\Pi_{\psi_{F,E,p}}.
  \]

  Define
  \[
    \pi_U
    =
    \left(
    {\bigotimes_{p<\infty}}'
    \pi_{U,p}^{\sph}
    \right)
    \otimes
    \pi_{U,\infty}^{\hol}.
  \]
  Then
  \(\pi_U\in\Pi_{\psi_{F,E}}\).
  Since \(\psi_{F,E}\) is simple generic,
  its global component group is trivial.
  Thus both the canonical character
  \(\epsilon_{\psi_{F,E}}\)
  and the global packet character of \(\pi_U\) are trivial,
  so that
  \[
    \pi_U
    \in
    \Pi_{\psi_{F,E}}(\epsilon_{\psi_{F,E}}).
  \]
  Proposition~\ref{prop:mok-global-multiplicity}
  therefore shows that \(\pi_U\) occurs with multiplicity one in the
  discrete spectrum.
  Since \(\psi_{F,E}\) is simple generic,
  the resulting discrete automorphic representation is cuspidal.

  The holomorphic adelic realization of \(\pi_U\) gives a Hermitian
  Hecke cusp eigenform
  \[
    \Lift_E(F)
    \in
    \calS_{\tau_{k,j}}(\U_{2,2}).
  \]

  Finally,
  the standard parameter of \(\pi_U\) is
  \[
    \psi_{F,E}^{N}
    =
    \Pi_{F,E}^{\Spin}\boxtimes\nu(1).
  \]
  Hence
  \begin{align*}
    L(s,\Lift_E(F),\St)
     & =
    L(s,\Pi_{F,E}^{\Spin},\St)          \\
     & =
    L(s,\Pi_F^{\Spin},\St)
    L(s,\Pi_F^{\Spin}\otimes\chi_E,\St) \\
     & =
    L\left(
    s+\frac{2k+j-3}{2},
    F,
    \Spin
    \right)
    L\left(
    s+\frac{2k+j-3}{2},
    F,
    \Spin\otimes\chi_E
    \right).
  \end{align*}
  This proves the proposition.
\end{proof}

We call the resulting form \(\Lift_E(F)\) the Hermitian spin lift of \(F\) with respect to the quadratic extension \(E/\bbQ\).
\begin{rem}
  The existence of the lift \(\Lift_E(F)\) was conjectured by Dummigan in \cite[\S~6]{Dummigan2014Eisenstein},
  as a functorial lift from \(\SO(n+1,n)\) to \(\U_{n,n}\).
\end{rem}

\begin{prop}
  \label{prop:spin-surj}
  Let \(G \in \calS_{\tau_{k,j}}(\U_{2,2})\) be a Hecke cusp eigenform on \(\U_{2,2}\),
  and let \(\pi_G\) be the automorphic representation generated by \(G\).
  Assume the following conditions.
  \begin{enumerate}
    \item
          The global Arthur parameter \(\psi_G\) of \(\pi_G\) is simple generic,
          with standard parameter
          \[
            \psi_G^N=\Pi_G\boxtimes\nu(1)
          \]
          for a conjugate self-dual cuspidal automorphic representation \(\Pi_G\) of \(\GL_4(\bbA_E)\).

    \item
          There is a cuspidal automorphic representation \(\Pi\) of \(\GL_4(\bbA_\bbQ)\) whose base change to \(E\) is \(\Pi_G\).
          We assume that \(\Pi_q\) is unramified for every finite rational prime \(q\).
  \end{enumerate}
  Then there exists a Hecke cusp eigenform \(F\in \calS_{{\det}^k\Sym^j}(\Sp_2)\) such that
  \[
    G=\Lift_E(F).
  \]
\end{prop}

\begin{proof}
  Since \(\Pi_G=\BC_{E/\bbQ}(\Pi)\) is both \(c\)-invariant and conjugate self-dual, it is self-dual.
  The uniqueness of quadratic base change \cite[Chapter~3]{Arthur1989BaseChange} therefore gives
  \[
    \Pi^\vee\simeq\Pi
    \quad\text{or}\quad
    \Pi^\vee\simeq\Pi\otimes\chi_E.
  \]
  The second possibility cannot occur.
  Indeed,
  if \(q\mid D_E\),
  then \(\chi_{E,q}\) is ramified.
  Since \(\Pi_q\) is unramified by assumption,
  \(\Pi_q\otimes\chi_{E,q}\) is ramified,
  and hence cannot be isomorphic to the unramified representation \(\Pi_q^\vee\).
  Hence \(\Pi\) is self-dual.

  Put
  \[
    I_m=\Ind_{W_{\bbC}}^{W_{\bbR}}(z\mapsto(z/|z|)^m).
  \]
  The weight of \(G\) gives the Langlands parameter of \(\Pi_\infty\) as
  \[
    I_{2k+j-3}\oplus I_{j+1}.
  \]
  Since both indices are odd and the two summands are non-isomorphic,
  this parameter is of symplectic type.
  Hence the self-dual cuspidal representation \(\Pi\) is symplectic.

  By \cite[Theorem~8.5.2 and Proposition~9.1.4]{Chenevier2019Automorphic},
  there exists a Hecke cusp eigenform \(F\in\calS_{{\det}^k\Sym^j}(\Sp_2)\) for which the corresponding cuspidal representation \(\Pi_F^{\Spin}\) of \(\PGL_4(\bbA_\bbQ)\) satisfies
  \[
    \Pi_F^{\Spin}\simeq\Pi.
  \]

  By Proposition~\ref{prop:spin-lift}, the standard parameter of \(\Lift_E(F)\) is
  \begin{align*}
    \BC_{E/\bbQ}(\Pi_F^{\Spin})\boxtimes\nu(1)
     & = \BC_{E/\bbQ}(\Pi)\boxtimes\nu(1) \\
     & = \Pi_G\boxtimes\nu(1)             \\
     & = \psi_G^N.
  \end{align*}
  Thus \(G\) and \(\Lift_E(F)\) belong to the same global Arthur packet.
  Their finite local components coincide by Lemma~\ref{lem:sphericalApacket},
  and their archimedean components are the same holomorphic discrete series representation of lowest \(K_{2,2,\infty}\)-type \(\tau_{k,j}\).
  The usual adelic realization therefore shows that \(G\) and \(\Lift_E(F)\) differ by a non-zero scalar.
  Rescaling the latter, we obtain \(G=\Lift_E(F)\).
\end{proof}

To simplify notation,
for a Hermitian automorphic form \(h\in \calS_{\tau_{k,j}}(\U_{2,2})\),
we put
\begin{align*}
  L_p^\natural(X,h,\St) & = L_p(p^{(2k+j-3)/2}X,h,\St), \\
  L^\natural(s,h,\St)   & = L(s-(2k+j-3)/2,h,\St).
\end{align*}
With this normalization,
Proposition~\ref{prop:spin-lift} gives
\[
  L_\ell^\natural(X,\Lift_E(F),\St) =L_\ell(X,F,\Spin)L_\ell(\chi_E(\ell)X,F,\Spin).
\]

\subsection{Hermitian congruences}
\label{subsec:herm-cong}

This subsection recalls the Hermitian Klingen--Eisenstein congruence theorem used in the proof of the main result.
We first fix the notation needed to state the theorem,
and then specialize it to the weights arising from Harder's conjecture.

For Hecke eigenforms \(H_1,H_2\) of the same weight on \(\U_{2,2}\),
let \(K\) be a number field containing their Hecke eigenvalues,
and let \(\frakP\) be a prime ideal of \(K\).
We write
\[
  H_1\equiv_{ev}H_2\pmod{\frakP}
\]
if their eigenvalues are congruent modulo \(\frakP\) for all operators in the Hecke algebra associated with
\[
  {\bigotimes}'_{p<\infty} \calH\left( \GU_{2,2}(\bbQ_p)\cap\M_4(\calO_{E_p}), \GU_{2,2}(\bbZ_p) \right).
\]
See \cite[Sections~3 and~6]{Takeda2026Congruences} for details.

We regard each elliptic cusp form \(f\) as the Hermitian automorphic form \(\iota_E(f)\).
Whenever a prime \(\frakp\) is fixed,
the ideal representatives in this construction may be chosen prime to \(p_\frakp\),
so their \(p_\frakp\)-components are integral and satisfy the hypothesis in \cite[Theorem~6.2(4e)]{Takeda2026Congruences}.

For integers \(\nu>\mu\ge2\),
put
\[
  \sigma_{\mu,\nu} = {\det}^{\mu}\Sym^{\nu-\mu}\boxtimes {\det}^{\mu}\Sym^{\nu-\mu}.
\]
Put \(V_{\mu,\nu}=V_{\sigma_{\mu,\nu}}\).
For \(n=1,2\),
we use the automorphy factor \(M_n(g_\infty)\) defined above.
Let \(P_{2,1}\subset\U_{2,2}\) be the standard Klingen parabolic subgroup,
whose Levi factor is isomorphic to \(\Res_{E/\bbQ}\GL_1\times\U_{1,1}\).
Write \(N_{2,1}\) for its unipotent radical,
and use the Levi embeddings
\begin{align*}
  t_{2,1}(a)
   & =
  \begin{pmatrix}
    a & 0 & 0                & 0 \\
    0 & 1 & 0                & 0 \\
    0 & 0 & \overline a^{-1} & 0 \\
    0 & 0 & 0                & 1
  \end{pmatrix}, \\
  s_{2,1}\!\left(
  \begin{pmatrix}
      \alpha & \beta  \\
      \gamma & \delta
    \end{pmatrix}
  \right)
   & =
  \begin{pmatrix}
    1 & 0      & 0 & 0      \\
    0 & \alpha & 0 & \beta  \\
    0 & 0      & 1 & 0      \\
    0 & \gamma & 0 & \delta
  \end{pmatrix}.
\end{align*}
As in \cite[Definition~4.7]{Takeda2026Congruences},
we regard \(V_{\sigma_\nu}\) as a subspace of \(V_{\mu,\nu}\) through the standard inclusion associated with the lower-right Levi embedding.
For \(f\in\calS_{\sigma_\nu}(\U_{1,1})\),
choose an Iwasawa decomposition
\[
  g = t_{2,1}(a)u\,s_{2,1}(h)k,
\]
where \(a\in\bbA_E^\times\),
\(u\in N_{2,1}(\bbA_\bbQ)\),
\(h\in\U_{1,1}(\bbA_\bbQ)\),
and \(k\in K_{2,2,f}K_{2,2,\infty}\).
Define the normalized Klingen--Eisenstein section \(\varepsilon(f)\in\calA_{\sigma_{\mu,\nu}}(P_{2,1})\) by
\[
  \varepsilon(f)(g) = |a\overline{a}|_{\bbA_f}^{\mu} \sigma_{\mu,\nu}\!\left(M_2(g_\infty)\right)^{-1} \sigma_\nu\!\left(M_1(h_\infty)\right)f(h).
\]
This is the specialization to \(s=\mu\) of the section in \cite[Definition~4.7]{Takeda2026Congruences}.
The associated Hermitian Klingen--Eisenstein series is
\[
  [f]_{2\mu}^{2}(g) = \sum_{\gamma\in P_{2,1}(\bbQ)\backslash\U_{2,2}(\bbQ)} \varepsilon(f)(\gamma g) \in\calA_{\sigma_{\mu,\nu}}(\U_{2,2}).
\]
Let \(\calG_0^{(2)}\subset\U_{2,2}(\bbA_f)\) be the set of ideal class representatives fixed in \cite[Section~2.2]{Takeda2026Congruences}.
For \(\gamma\in\calG_0^{(2)}\) and \(S\in\Her_2(E)_{\ge0}\),
write \(A_{[f]_{2\mu}^{2}}(\gamma,S)\) for the corresponding normalized Fourier coefficient.

Using the standard polynomial realization of \({\det}^{\mu}\Sym^{\nu-\mu}\),
let \(V_{\mu,\nu,\bbZ}\) be the \(\bbZ\)-lattice in \(V_{\mu,\nu}\) spanned by the standard monomial bases in the two tensor factors.
It is stable under \(\GL_2(\bbZ)\times\GL_2(\bbZ)\).
If \(K\) is a number field,
\(\frakp\) is a prime of \(K\),
and \(0\ne a\in V_{\mu,\nu,\bbZ}\otimes_\bbZ K\),
we put
\[
  v_\frakp(a)=\max\left\{
  m\in\bbZ
  \ \middle|\
  a\in\frakp^m
  (V_{\mu,\nu,\bbZ}\otimes_\bbZ\calO_\frakp)
  \right\}.
\]
We use this notation only for vector-valued Fourier coefficients.

\begin{prop}[{\cite[Corollary~6.5]{Takeda2026Congruences}}]\label{prop:main_n2_r1}
  Let \(\mu\) and \(\nu\) be integers with \(\nu>\mu\ge 2\),
  and let \(f\) be a primitive elliptic cusp form of weight \(2\nu\).
  Assume that there exists a prime ideal \(\frakp\) of \(\bbQ(f)\) satisfying the following conditions.
  \begin{enumerate}
    \item The form \(f\) is not congruent modulo \(\frakp\) to any other primitive elliptic cusp form of weight \(2\nu\).

    \item The value \(\zeta(3-2\mu)L(4-2\mu,\chi_E)\) is a \(\frakp\)-adic unit.

    \item There exist \(\gamma_0\in\calG_0^{(2)}\) and \(S_0\in \Her_2(E)_{>0}\) such that
          \(\frakp\mid L^\alg((\mu-1)+1/2,f,\St)\) and
          \[
            v_\frakp\!\left(
            L^\alg((\mu-1)+1/2,f,\St)A_{[f]^2_{2\mu}}(\gamma_0,S_0)
            \right)=0.
          \]
    \item The rational prime \(p_\frakp\) satisfies \(p_\frakp>2\mu+1\) and does not divide \(D_E\).
  \end{enumerate}
  Then there exists a Hecke cusp eigenform \(G\in \calS_{\sigma_{\mu,\nu}}(\U_{2,2})\) such that
  \[
    G\equiv_{ev}[f]_{2\mu}^{2} \pmod{\frakP}
  \]
  for some prime ideal \(\frakP\) of \(K=\bbQ(f,G)\) lying above \(\frakp\).
\end{prop}

\begin{proof}
  Apply \cite[Corollary~6.5]{Takeda2026Congruences} with \(n=2\), \(r=1\), with the parameter \(\mu\) there equal to \(2\mu\) here,
  and with both auxiliary highest weights equal to \((\nu-\mu)\).
  With our choice of ideal-class representatives,
  its hypotheses are precisely conditions~\((2)\)--\((4)\).
  Thus, for some \(\frakP\mid\frakp\), there is a Hecke eigenform \(G\in\calA_{\sigma_{\mu,\nu}}(\U_{2,2})\) such that
  \[
    G\equiv_{ev}[f]^2_{2\mu}\pmod{\frakP},
  \]
  whose eigensystem is different from that of the basis element \([f]^2_{2\mu}\);
  see also \cite[Lemma~6.4]{Takeda2026Congruences}.

  It remains to show that \(G\) is cuspidal.
  In the present \(n=2\), \(r=1\) case,
  every non-cuspidal eigensystem in \(\calA_{\sigma_{\mu,\nu}}(\U_{2,2})\) is that of a Klingen--Eisenstein lift \([h]^2_{2\mu}\),
  with \(h\) a primitive elliptic cusp eigenform of weight \(2\nu\);
  see the decomposition used in \cite[Theorem~5.12]{Takeda2026Congruences}.
  Suppose that \(G\) has the eigensystem of \([h]^2_{2\mu}\).
  Comparing the explicit standard Euler factors gives, for
  \(\ell\nmid p_\frakp D_E\),
  \[
    L_\ell(X,h)L_\ell(\chi_E(\ell)X,h)
    \equiv
    L_\ell(X,f)L_\ell(\chi_E(\ell)X,f)
    \pmod{\frakP}.
  \]
  Hence, by the Chebotarev density theorem and Brauer--Nesbitt theorem,
  \[
    \overline\rho_h^{ss}
    \oplus
    \overline\rho_h^{ss}\otimes\overline\chi_E
    \simeq
    \overline\rho_f^{ss}
    \oplus
    \overline\rho_f^{ss}\otimes\overline\chi_E.
  \]
  For \(q\mid D_E\),
  the untwisted representations are unramified at \(q\),
  whereas their \(\overline\chi_E\)-twists are ramified.
  Therefore
  \[
    \overline\rho_h^{ss}
    \simeq
    \overline\rho_f^{ss}.
  \]
  Thus \(h\equiv f\pmod{\frakp}\), and condition~\((1)\) gives \(h=f\),
  contrary to the choice of \(G\).
  Hence \(G\) is cuspidal.
\end{proof}
\begin{rem}
  \label{rem:fourier-coefficient-condition}
  The Fourier-coefficient condition in~(3) can be studied using the primitive Fourier coefficient formula of \cite[Theorem~4.5]{Takeda2026PrimitiveFourier}.
  In the full-level degree-two case,
  \cite[Corollary~5.3]{Takeda2026PrimitiveFourier} shows that a \(\frakP\)-adic unit condition for the corresponding primitive coefficient yields a Fourier coefficient satisfying the required valuation condition after multiplication by the relevant algebraic standard \(L\)-value.
  In the specialization considered below,
  \cite[Corollary~5.6]{Takeda2026PrimitiveFourier} gives a sufficient condition for this unit property in terms of finitely many Fourier coefficients.
\end{rem}

Let \(k\ge 4\) and \(j\ge 2\) be integers with \(j\) even.
Put \(\mu=j/2+2\),
\(\nu=k+j/2-1\).
With this choice,
Proposition~\ref{prop:main_n2_r1} applies to a primitive elliptic cusp form \(f\) of weight \(2k+j-2\) whenever its hypotheses are satisfied.
For this choice of \(\mu\) and \(\nu\),
we have \(\sigma_{\mu,\nu}=\tau_{k,j}\) and \([f]_{2\mu}^{2}=[f]_{j+4}^{2}\).

If the cusp form \(G\) obtained in Proposition~\ref{prop:main_n2_r1} is the lift \(\Lift_E(F)\) of a Siegel cusp form constructed in the previous subsection,
then the congruence can be translated into the Harder-type congruence.

\begin{thm}
  \label{thm:harder}
  Let \(f\in S_{2k+j-2}(\SL_2(\bbZ))\) be a primitive elliptic cusp form,
  regarded as the Hermitian automorphic form \(\iota_E(f)\),
  and let \(F\in\calS_{{\det}^{k}\Sym^{j}}(\Sp_2)\) be a Hecke cusp eigenform.
  Let \(K\) be a number field containing the Hecke fields of \(f\) and \(F\),
  and let \(\frakP\) be a prime of \(K\).
  Assume that \(p_\frakP>2k+j-2\) and that \(p_\frakP\) does not divide \(D_E\).
  Assume also that
  \[
    \Lift_E(F)\equiv_{ev}[f]^2_{j+4}\pmod{\frakP}.
  \]
  Then we have
  \[
    L_\ell(X,F,\Spin) \equiv L_\ell(X,f)(1-\ell^{k+j-1}X)(1-\ell^{k-2}X) \pmod{\frakP}
  \]
  for every prime \(\ell\).
\end{thm}

To prove the theorem,
we recall the Galois representations attached to \(f\) and \(F\).
Put \(p=p_\frakP\).
We use geometric Frobenius elements \(\Frob_\ell\) throughout.

By the work of Deligne \cite{Deligne1971Formes},
there is a \(2\)-dimensional continuous semisimple Galois representation
\[
  \rho_f: \Gal_\bbQ \rightarrow \GL_2(K_\frakP)
\]
associated with the elliptic modular form \(f\),
such that,
for every prime \(\ell\ne p\),
the representation \(\rho_f\) is unramified at \(\ell\) and
\[
  L_\ell(X,f) = \det \left( 1-X\rho_f(\Frob_\ell) \right).
\]
Similarly,
by the work of Weissauer \cite{Weissauer2005Four},
the cuspidal automorphic representation of \(\GSp_4\) associated with \(F\) admits a \(4\)-dimensional continuous semisimple Galois representation
\[
  \rho_F: \Gal_\bbQ \rightarrow \GL_4(K_\frakP)
\]
such that,
for every prime \(\ell\ne p\),
the representation \(\rho_F\) is unramified at \(\ell\) and
\[
  L_\ell(X,F,\Spin) = \det \left( 1-X\rho_F(\Frob_\ell) \right).
\]

Here \(D_{\cris}(V)\) denotes Fontaine's crystalline Dieudonn\'e module
and \(\varphi_p\) its crystalline Frobenius;
see \cite{Fontaine1994CorpsPeriodes}.
Our conventions for crystalline representations and Hodge--Tate weights are given below in Definition~\ref{dfn:crys-rep} and the paragraph following it.

For \(\rho_f\),
the usual \(p\)-adic local--global compatibility for elliptic modular forms of level prime to \(p\) shows that \(\rho_f|_{\Gal_{\bbQ_p}}\) is crystalline and gives
\[
  L_p(X,f) = \det \left( 1-X\varphi_p \mid D_{\cris} \left( \rho_f|_{\Gal_{\bbQ_p}} \right) \right)
\]
For \(\rho_F\),
the Galois representation attached to \(\Pi_F^{\Spin}\) agrees with \(\rho_F\) by the Chebotarev density theorem and the Brauer--Nesbitt theorem,
since their characteristic polynomials agree at every prime \(\ell\ne p\).
Moreover,
\(\Pi_{F,p}^{\Spin}\) is unramified because \(F\) has full level.
The \(p\)-adic local--global compatibility theorem of Barnet-Lamb--Gee--Geraghty--Taylor \cite[Corollary~1.2]{BarnetLamb2014LocalGlobal} shows that \(\rho_F|_{\Gal_{\bbQ_p}}\) is crystalline and gives
\[
  L_p(X,F,\Spin) = \det \left( 1-X\varphi_p \mid D_{\cris} \left( \rho_F|_{\Gal_{\bbQ_p}} \right) \right).
\]

The Hodge--Tate weights of \(\rho_f\) are \(\{-(2k+j-3),0\}\),
and those of \(\rho_F\) are \(\{-(2k+j-3),-(k+j-1),-(k-2),0\}\).

We denote by \(\overline\rho_f^{ss}\) and \(\overline\rho_F^{ss}\) the semisimplifications of the residual representations of \(\rho_f\) and \(\rho_F\),
respectively.

For \(L=\bbQ\) or \(E\),
we denote by \(\frakA_\frakP(\Gal_L)\) the Grothendieck ring of finite-dimensional continuous representations of \(\Gal_L\) over \(\kappa_\frakP\).
By abuse of notation,
we use the same symbol for a residual representation and its class in the corresponding Grothendieck ring.
Addition is induced by direct sum,
and multiplication is induced by tensor product.
Throughout this paper,
all decompositions and equalities of residual \(\Gal_L\)-representations are understood in \(\frakA_\frakP(\Gal_L)\),
unless otherwise stated.

We denote by \(\overline\omega:\Gal_\bbQ\rightarrow\kappa_\frakP^\times\) the mod \(\frakP\) cyclotomic character.
With our geometric Frobenius convention,
we have \(\overline\omega(\Frob_\ell)=\ell^{-1}\) for every rational prime \(\ell\ne p\).

Put
\[
  \calB_f = \overline\rho_f^{ss} + \overline\omega^{-k+2} + \overline\omega^{-k-j+1} \in \frakA_\frakP(\Gal_\bbQ).
\]

Let \(I_p\subset\Gal_{\bbQ_p}\) be the inertia subgroup.
The following Fontaine--Laffaille lemma will be applied to both \(\rho_f|_{\Gal_{\bbQ_p}}\) and \(\rho_F|_{\Gal_{\bbQ_p}}\).

\begin{lem}
  \label{lem:FL-euler-reduction}
  Let \(V\) be a crystalline \(K_\frakP\)-representation of \(\Gal_{\bbQ_p}\),
  and let \(T\subset V\) be a stable \(\calO_\frakP\)-lattice.
  Assume that all Hodge--Tate weights of \(V\) lie in \([-(p-2),0]\),
  and that
  \[
    P_V(X)=\det\left(1-X\varphi_p\mid D_{\cris}(V)\right)
  \]
  belongs to \(\calO_\frakP[X]\).
  Then
  \[
    P_V(X)\equiv \det\left(1-X\Frob_p\mid (T/\frakP T)^{ss,I_p}\right) \pmod{\frakP}.
  \]
\end{lem}

\begin{proof}
  This is a standard consequence of \cite[Th\'eor\`eme~5.3]{Fontaine1982Representations};
  see also \cite[Chapter~2, \S2.1.1--\S2.1.2]{Ribet1999Serre}.
\end{proof}

\begin{proof}[Proof of Theorem~\ref{thm:harder}]
  The assumed congruence of Hecke eigenvalues gives
  \[
    L_\ell^\natural(X,\Lift_E(F),\St) \equiv L_\ell^\natural(X,[f]^2_{j+4},\St) \pmod{\frakP}
  \]
  for every prime \(\ell\ne p_\frakP\).
  By Proposition~\ref{prop:spin-lift},
  we have
  \[
    L_\ell^\natural(X,\Lift_E(F),\St) = L_\ell(X,F,\Spin) L_\ell(\chi_E(\ell)X,F,\Spin).
  \]
  On the other hand,
  the local standard \(L\)-polynomial of the Hermitian Klingen--Eisenstein lift is
  \[
    L_\ell^\natural(X,[f]^2_{j+4},\St) = L_\ell(X,f)L_\ell(\chi_E(\ell)X,f) A_\ell(X)A_\ell(\chi_E(\ell)X),
  \]
  where we put
  \[
    A_\ell(X)=(1-\ell^{k+j-1}X)(1-\ell^{k-2}X).
  \]
  Hence
  \[
    L_\ell(X,F,\Spin)L_\ell(\chi_E(\ell)X,F,\Spin) \equiv L_\ell(X,f)L_\ell(\chi_E(\ell)X,f) A_\ell(X)A_\ell(\chi_E(\ell)X) \pmod{\frakP}
  \]
  for every prime \(\ell\ne p_\frakP\).

  Let \(\overline\chi_E\) denote the reduction of \(\chi_E\) modulo \(\frakP\).
  By the Chebotarev density theorem and the Brauer--Nesbitt theorem,
  the preceding congruences imply
  \[
    (1+\overline\chi_E)\overline\rho_F^{ss}=(1+\overline\chi_E)\calB_f
  \]
  in the Grothendieck ring \(\frakA_\frakP(\Gal_\bbQ)\).

  Let \(q_0\mid D_E\) be a rational prime with \(q_0\ne p_\frakP\).
  Since \(F\) has full level,
  the representation \(\overline\rho_F^{ss}\) is unramified at \(q_0\).
  The constituents of \(\calB_f\) are unramified at \(q_0\),
  because \(q_0\ne p_\frakP\) and \(f\) has level one.
  On the other hand,
  \(\overline\chi_E\) is ramified at \(q_0\).
  Since every constituent of \(\overline\chi_E\calB_f\) and \(\overline\rho_F^{ss}\otimes\overline\chi_E\) is ramified at \(q_0\),
  comparing the constituents unramified at \(q_0\) gives
  \begin{equation}\label{eq:residual-congruence}
    \overline\rho_F^{ss}
    =
    \calB_f.
  \end{equation}
  Therefore,
  for every prime \(\ell\ne p_\frakP\),
  we have
  \[
    L_\ell(X,F,\Spin) \equiv L_\ell(X,f)(1-\ell^{k+j-1}X)(1-\ell^{k-2}X)\pmod{\frakP}.
  \]

  It remains to treat the prime \(\ell=p=p_\frakP\).
  Since \(p>2k+j-2\),
  the Hodge--Tate weights of both \(\rho_f\) and \(\rho_F\) lie in the Fontaine--Laffaille range.
  Thus Lemma~\ref{lem:FL-euler-reduction},
  applied to stable lattices in \(\rho_f\) and \(\rho_F\),
  gives
  \[
    L_p(X,F,\Spin) \equiv \det\left(1-X\Frob_p\mid(\overline\rho_F^{ss})^{I_p}\right) \pmod{\frakP}
  \]
  and
  \[
    L_p(X,f) \equiv \det\left(1-X\Frob_p\mid(\overline\rho_f^{ss})^{I_p}\right) \pmod{\frakP}.
  \]

  Since \(0<k-2<k+j-1<p-1\),
  the characters \(\overline\omega^{-k+2}\) and \(\overline\omega^{-k-j+1}\) are non-trivial on \(I_p\).
  Hence,
  by \eqref{eq:residual-congruence},
  \[
    (\overline\rho_F^{ss})^{I_p} = (\overline\rho_f^{ss})^{I_p},
  \]
  and therefore
  \[
    L_p(X,F,\Spin) \equiv L_p(X,f) \pmod{\frakP}.
  \]
  Finally,
  \[
    (1-p^{k+j-1}X)(1-p^{k-2}X) \equiv 1 \pmod{\frakP},
  \]
  so
  \[
    L_p(X,F,\Spin) \equiv L_p(X,f)(1-p^{k+j-1}X)(1-p^{k-2}X) \pmod{\frakP}.
  \]
  This proves the desired congruence at \(\ell=p\).
\end{proof}

In the following sections,
we study sufficient conditions for \(G\) in Proposition~\ref{prop:main_n2_r1} to be of the form \(G=\Lift_E(F)\) for some Hecke cusp eigenform \(F\) on \(\Sp_2\).

\section{Galois representations and Selmer groups}
\label{sec:galois-selmer}

In this section,
whenever \(K_\frakP\)-linear representations occur,
\(K\) denotes a number field containing their coefficient fields,
and \(\frakP\) denotes a prime ideal of \(K\).
Recall that \(\calO_\frakP\) and \(\kappa_\frakP\) denote the valuation ring and the residue field of \(K_\frakP\),
respectively.
We put \(p=p_\frakP\).

\subsection{Local conditions for Selmer groups}
\label{subsec:selmer-local-conditions}
We recall the local conditions used below.
Let \(L\) denote either \(E\) or \(\bbQ\),
and let \(v\) be a finite place of \(L\).

For \(v\mid p\),
let \(L_{v,0}\) be the maximal unramified subfield of \(L_v\).
We write \(B_{\cris}\) and \(B_{\dR}\) for Fontaine's crystalline and de Rham period rings,
respectively;
see \cite{Fontaine1994CorpsPeriodes}.

\begin{dfn}
  \label{dfn:crys-rep}
  Let \(v\mid p\),
  and let \(V\) be a finite-dimensional \(K_\frakP\)-representation of \(\Gal_{L_v}\).
  Put
  \[
    D_{\cris}(V)=\left(B_{\cris}\otimes_{\bbQ_p}V\right)^{\Gal_{L_v}}.
  \]
  We say that \(V\) is crystalline if
  \[
    \dim_{L_{v,0}\otimes_{\bbQ_p}K_\frakP}D_{\cris}(V)=\dim_{K_\frakP}V.
  \]
\end{dfn}

If \(V\) is crystalline,
then
\[
  D_{\dR}(V):=\left(B_{\dR}\otimes_{\bbQ_p}V\right)^{\Gal_{L_v}}
  \simeq L_v\otimes_{L_{v,0}}D_{\cris}(V).
\]
We use the convention
\[
  \HT(V)=\left\{j\in\bbZ\;\middle|\;
  \Fil^{-j}D_{\dR}(V)/\Fil^{-j+1}D_{\dR}(V)\ne0\right\}.
\]
Thus,
for the \(p\)-adic cyclotomic character
\(\omega:\Gal_\bbQ\rightarrow\calO_\frakP^\times\),
we have
\[
  \HT\left(\omega|_{\Gal_{L_v}}\right)=\{1\}.
\]

\begin{dfn}
  \label{dfn:short-crys}
  A crystalline \(K_\frakP\)-representation \(V\) of \(\Gal_{L_v}\),
  with \(v\mid p\),
  is called short crystalline if
  \(\Fil^0D_{\dR}(V)=D_{\dR}(V)\),
  \(\Fil^pD_{\dR}(V)=0\),
  and \(V'\otimes_{\bbQ_p}\bbQ_p(p-1)\) is ramified
  for every non-zero quotient \(V'\) of \(V\).
\end{dfn}

The two filtration conditions are equivalent to
\[
  \HT(V)\subset[-(p-1),0].
\]
In particular,
if
\[
  \HT(V)\subset[-(p-2),0],
\]
then \(V\) is short crystalline.

For a continuous \(\Gal_L\)-module \(M\),
write
\[
  H^1(L,M)=H^1(\Gal_L,M)
\]
for continuous Galois cohomology.
For a continuous \(\Gal_{L_v}\)-module \(M\),
we similarly write
\[
  H^1(L_v,M)=H^1(\Gal_{L_v},M).
\]
Put
\[
  H^1_{\ur}(L_v,M)
  =\ker\left(H^1(L_v,M)\longrightarrow H^1(I_v,M)\right),
\]
where \(I_v\subset\Gal_{L_v}\) is the inertia subgroup.

\begin{dfn}
  \label{dfn:local-bk}
  Let \(V\) be a finite-dimensional \(K_\frakP\)-representation of \(\Gal_{L_v}\).
  The Bloch--Kato finite local condition is
  \[
    H^1_f(L_v,V)=
    \begin{cases}
      H^1_{\ur}(L_v,V),                           & v\nmid p, \\[4pt]
      \ker\left(H^1(L_v,V)\longrightarrow
      H^1(L_v,V\otimes_{\bbQ_p}B_{\cris})\right), & v\mid p.
    \end{cases}
  \]
\end{dfn}

For \(v\mid p\),
this is the crystalline Bloch--Kato local condition
\cite{Bloch1990Tamagawa}.

Let \(V\) now be a finite-dimensional \(K_\frakP\)-representation of \(\Gal_L\),
and let \(T\subset V\) be a \(\Gal_L\)-stable \(\calO_\frakP\)-lattice.
Put \(A=V/T\) and \(\overline T=T/\frakP T\).

For every finite place \(v\) of \(L\),
define
\[
  H^1_f(L_v,A)
  =\Im\left(H^1_f(L_v,V)\longrightarrow H^1(L_v,A)\right).
\]
For \(v\mid p\),
let \(H^1_f(L_v,T)\) be the inverse image of \(H^1_f(L_v,V)\) under
\[
  H^1(L_v,T)\longrightarrow H^1(L_v,V).
\]
The residual finite local condition is
\[
  H^1_f(L_v,\overline T)=
  \begin{cases}
    H^1_{\ur}(L_v,\overline T),                                       & v\nmid p, \\[4pt]
    \Im\left(H^1_f(L_v,T)\longrightarrow H^1(L_v,\overline T)\right), & v\mid p.
  \end{cases}
\]

For any continuous \(\Gal_L\)-module \(M\) occurring below
for which the local finite conditions \(H^1_f(L_v,M)\) have been specified,
we define
\[
  H^1_f(L,M)
  =
  \ker\left(
  H^1(L,M)\longrightarrow
  \bigoplus_v\frac{H^1(L_v,M)}{H^1_f(L_v,M)}
  \right),
\]
where \(v\) runs over the finite places of \(L\).
\begin{lem}
  \label{lem:fl-fixed-lifts}
  Let \(L_v/\bbQ_p\) be finite and unramified,
  and let \(T_1,T_2\) be stable \(\calO_\frakP\)-lattices
  in short crystalline representations.
  Put \(\overline T_i=T_i/\frakP T_i\).
  Consider an extension
  \[
    0\longrightarrow\overline T_1\longrightarrow\overline X
    \longrightarrow\overline T_2\longrightarrow0
  \]
  and suppose that \(\overline X\) is a subquotient of the reduction
  of a stable lattice in a short crystalline representation.
  Then its class lies in
  \(H^1_f(L_v,\Hom_{\kappa_\frakP}(\overline T_2,\overline T_1))\),
  with the finite local condition defined using \(\Hom_{\calO_\frakP}(T_2,T_1)\).
\end{lem}
\begin{proof}
  The Fontaine--Laffaille subcategory associated with short crystalline representations is closed under subquotients and extensions,
  and the Fontaine--Laffaille functor on this subcategory is exact and fully faithful \cite[Th\'eor\`emes~3.3 and~6.1]{Fontaine1982Representations}.
  Thus the given extension corresponds to an extension of the reductions of the Fontaine--Laffaille modules attached to \(T_1,T_2\).
  The construction in the proof of \cite[Proposition~2.17]{Kalloniatis2019Flagged} lifts this extension within the same subcategory while keeping these two modules fixed.
  Hence the given extension lifts to a short crystalline extension of \(T_2\) by \(T_1\),
  and its class satisfies the asserted finite local condition.
\end{proof}

\begin{lem}
  \label{lem:quadratic-local-vanishing}
  Let \(q\mid D_E\),
  and let \(M_0\) be a finite-dimensional unramified \(\kappa_\frakP\)-representation of \(\Gal_{\bbQ_q}\).
  Assume that \(p\ne2\) and that \(p\) does not divide \(D_E\).
  Then
  \[
    H^1\left( \bbQ_q, M_0\otimes\overline\chi_E \right)=0.
  \]
\end{lem}
\begin{proof}
  Since \(q\ne p\),
  both \(M_0\) and \(M_0^\vee(1)\) are unramified at \(q\),
  whereas \(\overline\chi_E|_{I_q}\) is non-trivial.
  Hence
  \[
    H^0\left( \bbQ_q, M_0\otimes\overline\chi_E \right) = H^0\left( \bbQ_q, M_0^\vee(1)\otimes\overline\chi_E \right) = 0.
  \]
  By local Tate duality,
  \[
    H^2\left( \bbQ_q, M_0\otimes\overline\chi_E \right)=0.
  \]
  The local Euler--Poincar\'e characteristic formula therefore gives
  \[
    H^1\left( \bbQ_q, M_0\otimes\overline\chi_E \right)=0.
  \]

  For these standard local cohomological facts,
  see, for example,
  \cite[Chapter~VII, Sections~2--3]{Neukirch2008Cohomology}.
\end{proof}

\begin{lem}
  \label{lem:partitioned-lattice-selmer}
  Let \(R:\Gal_\bbQ\rightarrow\GL(V)\) be a continuous representation over \(K_\frakP\),
  and let \(T\subset V\) be a \(\Gal_\bbQ\)-stable \(\calO_\frakP\)-lattice.
  Put \(\overline T=T/\frakP T\).
  Assume that the following conditions hold.
  \begin{enumerate}
    \item \(R\) is absolutely irreducible.

    \item We have \(p\ne2\), the prime \(p\) does not divide \(D_E\),
          the representation \(R\) is unramified outside \(pD_E\),
          and \(R\) is short crystalline at \(p\).

    \item The irreducible constituents of \(\overline T^{ss}\) are pairwise non-isomorphic and admit a partition \(\calA\amalg\calC\) into two non-empty sets of irreducible \(\kappa_\frakP[\Gal_\bbQ]\)-modules such that
          \[
            \overline T^{ss}
            =
            \bigoplus_{S\in\calA\amalg\calC}S.
          \]
    \item For every \(q\mid D_E\) and every
          \[
            (S,S')\in(\calA\times\calC)\cup(\calC\times\calA),
          \]
          the module \(\Hom_{\kappa_\frakP}(S,S')\) is isomorphic to \(M_0\otimes\overline\chi_E\) for some finite-dimensional unramified
          \(\kappa_\frakP\)-representation \(M_0\) of \(\Gal_{\bbQ_q}\).
  \end{enumerate}
  Then,
  there exist \(S_a\in\calA\) and \(S_c\in\calC\) for which one of the groups
  \[
    H^1_f\left(\bbQ,\Hom_{\kappa_\frakP}(S_c,S_a)\right)
    \quad\text{and}\quad
    H^1_f\left(\bbQ,\Hom_{\kappa_\frakP}(S_a,S_c)\right)
  \]
  is non-zero.
\end{lem}
\begin{proof}
  Applying the lattice argument in the proof of \cite[Corollary~7.4]{Atobe2026HarderII} to the partition \(\calA\amalg\calC\),
  we obtain a non-split residual extension between a constituent of \(\calA\) and a constituent of \(\calC\).
  Since \(R\) is absolutely irreducible,
  after replacing the stable lattice and taking a residual subquotient,
  there is a non-split extension between a constituent in \(\calA\)
  and a constituent in \(\calC\).
  Its extension class is non-zero in one of the two global cohomology groups displayed in the statement.

  We verify the Bloch--Kato local conditions.
  At every prime \(\ell\nmid pD_E\),
  the class is unramified since \(R\) is unramified at \(\ell\).
  At \(p\),
  short crystallinity and \cite[Lemma~7.6]{Atobe2026HarderII}
  show that it lies in the residual finite local condition.
  If \(q\mid D_E\),
  assumption~(4) and Lemma~\ref{lem:quadratic-local-vanishing}
  show that the corresponding local cohomology group vanishes.
  Hence the local class at \(q\) is zero.
  Therefore the extension class lies in the asserted Bloch--Kato Selmer group.
\end{proof}

\subsection{Bella\"{\i}che--Chenevier signs}
\label{subsec:bc-signs}

We use the notion of sign for self-dual and conjugate self-dual Galois representations introduced in \cite[Section~1.8]{Bellaiche2009Families} and \cite[Section~1.1]{Bellaiche2011Sign}.
We refer to these signs as the Bella\"{\i}che--Chenevier signs,
or BC signs.
They should not be confused with the conjugate-orthogonal or conjugate-symplectic parity of automorphic parameters in \cite[Section~2.2]{Mok2015Endoscopic}.

Let \(K_0\) be a field of characteristic different from \(2\),
and let \(L\) be either \(\bbQ\) or \(E\).
For a character \(\mu:\Gal_L\rightarrow K_0^\times\),
with \(\mu^c=\mu\) when \(L=E\),
put
\[
  D_{\mu,L}(V) = \begin{cases}
    V^\vee\otimes\mu     & (L=\bbQ), \\
    (V^c)^\vee\otimes\mu & (L=E).
  \end{cases}
\]
Thus a self-duality,
or conjugate self-duality when \(L=E\),
of an absolutely irreducible representation \(\rho:\Gal_L\rightarrow\GL_n(K_0)\) has BC sign \(\epsilon\in\{\pm1\}\) if its representing matrix \(B\) may be chosen so that
\[
  {}^tB=\epsilon B.
\]
More explicitly,
\(B\) satisfies
\[
  {}^t\rho(g)B\rho(g)=\mu(g)B
\]
when \(L=\bbQ\),
and
\[
  {}^t\rho^c(g)B\rho(g)=\mu(g)B
\]
when \(L=E\).
By Schur's lemma,
the sign is independent of the choice of \(B\).

We also use the induced sign on extension classes.

\begin{dfn}
  \label{dfn:bc-extension-sign}
  Let \(S\) and \(T\) be finite-dimensional \(K_0\)-representations of \(\Gal_L\),
  and fix an isomorphism
  \[
    J:T\xrightarrow{\sim}D_{\mu,L}(S).
  \]
  Using the natural biduality isomorphism \(\iota_S:S\xrightarrow{\sim}D_{\mu,L}^2(S)\),
  put
  \[
    J^\sharp = D_{\mu,L}(J)\circ\iota_S: S\xrightarrow{\sim}D_{\mu,L}(T).
  \]

  Applying \(D_{\mu,L}\) to an extension
  \[
    0\longrightarrow S \longrightarrow X \longrightarrow T \longrightarrow0
  \]
  and transporting the resulting extension through \(J^\sharp\) and \(J\) defines an involution
  \[
    \tau_L: H^1\left( L, \Hom_{K_0}(T,S) \right) \longrightarrow H^1\left( L, \Hom_{K_0}(T,S) \right).
  \]
  An extension class \(x\) has BC sign \(\epsilon\in\{\pm1\}\) if
  \[
    \tau_L(x)=\epsilon x.
  \]

  If \(S\) and \(T\) are absolutely irreducible,
  then \(J\) is unique up to scalar,
  so the involution and its eigenspace decomposition are independent of the choice of \(J\).
\end{dfn}

For the representations considered below,
whenever the finite local conditions are defined,
the involution \(\tau_L\) is compatible with localization and preserves them;
when \(L=E\),
it carries the condition at \(v\) to that at \(v^c\).
Hence it induces an involution on the corresponding Selmer group,
and we write
\[
  H^1_f\left( L, \Hom_{K_0}(T,S) \right)^\epsilon = \ker\left( \tau_L-\epsilon \right).
\]

\subsection{Selmer groups for one-dimensional characters}
\label{subsec:one-dim}

From now on,
we fix positive integers \(k\) and \(j\) with \(k\ge 4\) and \(j\) even,
as in the preceding sections.
We also fix a primitive elliptic cusp form \(f\) of weight \(2k+j-2\).
We keep the notation \(p=p_\frakP\).
Throughout this subsection we assume that \(p\ne2\) and that \(p\) does not divide \(D_E\).

Recall that \(\overline{\omega}:\Gal_\bbQ\rightarrow\kappa_\frakP^\times\) denotes the mod \(\frakP\) cyclotomic character.
Let \(\widetilde{\omega}:\Gal_\bbQ\rightarrow\calO_\frakP^\times\) be its Teichm\"uller lift,
viewed as a Dirichlet character.

\begin{lem}
  \label{lem:omega-odd-selmer}
  Let \(a\) be an odd integer with \(1-p<a<0\).
  If \(L(0,\widetilde{\omega}^{-a})\) is a \(\frakP\)-adic unit,
  then
  \[
    H^1_f(\bbQ,\overline{\omega}^{a})=0.
  \]
\end{lem}

\begin{proof}
  This follows from the argument in \cite[Lemma~7.8(1)]{Atobe2026HarderII}.
\end{proof}

Recall that \(\overline{\chi}_E\) denotes the reduction of \(\chi_E\) modulo \(\frakP\).

\begin{lem}
  \label{lem:chiE-selmer}
  If \(p\) does not divide \(h_E\),
  then
  \[
    H^1_f(\bbQ,\overline\chi_E)=0.
  \]
\end{lem}

\begin{proof}
  Since \(p\ne2\),
  inflation--restriction gives an injection
  \[
    H^1(\bbQ,\overline\chi_E) \hookrightarrow H^1(E,\kappa_\frakP).
  \]
  If \(c\in H^1_f(\bbQ,\overline\chi_E)\),
  then \(\res_E(c)\) is unramified at every finite place of \(E\).
  Indeed,
  at \(q\mid D_E\) this follows from Lemma~\ref{lem:quadratic-local-vanishing},
  while at all other places it follows from the Bloch--Kato local conditions.
  Hence
  \[
    \res_E(c) \in H^1_f(E,\kappa_\frakP) \simeq \Hom(\Cl(E),\kappa_\frakP).
  \]
  Since \(p\) does not divide \(h_E\),
  the latter group is zero,
  and therefore \(c=0\).
\end{proof}

\begin{rem}
  Put \(w_E=\#\calO_E^\times\).
  The class number formula for the imaginary quadratic field \(E\) gives
  \[
    L(0,\chi_E)=\frac{2h_E}{w_E}.
  \]
  If \(p\ne2\) and \(p\) does not divide \(D_E\),
  then \(p\) does not divide \(w_E\).
  Hence,
  under these assumptions,
  the condition that \(p\) does not divide \(h_E\) is equivalent to the condition that
  \(L(0,\chi_E)\) is a \(\frakP\)-adic unit.

  Lemma~\ref{lem:chiE-selmer} also follows from the Bloch--Kato theorem for Dirichlet motives of Huber--Kings \cite{Huber2003BlochKato}.
\end{rem}

\begin{cor}
  \label{cor:bc-plus-rho-selmer}
  Let \(V\) be the underlying \(\kappa_\frakP\)-vector space of \(\overline\rho_f\),
  and put \(W=V\otimes\overline\chi_E\).
  Assume that \(p>2\) and that \(p\) does not divide \(h_E\).
  Then
  \[
    H^1_f\left( \bbQ, \Hom_{\kappa_\frakP}(W,V) \right)^+ = H^1_f\left( \bbQ, \Hom_{\kappa_\frakP}(V,W) \right)^+ =0.
  \]
\end{cor}

\begin{proof}
  Put \(\lambda=\det(V)\,\overline\chi_E\),
  and let \(\langle\ ,\ \rangle: V\times V\rightarrow\det(V)\) be the determinant pairing.
  Through these,
  we fix the duality isomorphisms \(W\simeq D_{\lambda,\bbQ}(V)\) and \(V\simeq D_{\lambda,\bbQ}(W)\).

  Put \(M=\End_{\kappa_\frakP}(V)\).
  Then the above isomorphisms identify
  \begin{equation}
    \label{eq:hom-end-twist-identification}
    \Hom_{\kappa_\frakP}(W,V) \simeq M\otimes\overline\chi_E.
  \end{equation}
  For each \(A\in M\),
  let \(A^\dagger\) be the adjoint of \(A\) with respect to \(\langle\ ,\ \rangle\).
  For \(\epsilon\in\{\pm\}\),
  put
  \[
    M^\epsilon = \{A\in M\mid A^\dagger=\epsilon A\}.
  \]
  Since \(p>2\),
  we have \(M=M^+\oplus M^-\).
  Since we have \(\dim V=2\),
  we note that
  \[
    A^\dagger=\tr(A)\id_V-A.
  \]
  Hence we have
  \[
    M^+=\kappa_\frakP\id_V\simeq\kappa_\frakP.
  \]

  Under \eqref{eq:hom-end-twist-identification},
  the involution \(\tau_\bbQ\) is induced by the \(\Gal_\bbQ\)-equivariant map
  \begin{align*}
    M\otimes\overline\chi_E & \longrightarrow M\otimes\overline\chi_E, \\
    A\otimes t              & \longmapsto A^\dagger\otimes t.
  \end{align*}
  Consequently,
  \begin{align*}
    H^1_f\left( \bbQ,\Hom_{\kappa_\frakP}(W,V) \right)^+
     & \simeq H^1_f\left( \bbQ,M^+\otimes\overline\chi_E \right) \\
     & \simeq H^1_f\left( \bbQ,\overline\chi_E \right).
  \end{align*}
  Lemma~\ref{lem:chiE-selmer} proves the vanishing for \(\Hom_{\kappa_\frakP}(W,V)\).
  The same argument with \(V\) and \(W\) interchanged proves the other vanishing.
\end{proof}

\subsection{Selmer groups for two-dimensional representations}
\label{subsec:two-dim}

We next exclude extensions between \(\overline\rho_f\) and the one-dimensional constituents.

\begin{prop}
  \label{prop:kato-input}
  Let \(h\in S_{2\nu}(\SL_2(\bbZ))\) be primitive,
  and let \((\rho_{h\otimes\chi_E},V_{h\otimes\chi_E})\) be the two-dimensional \(K_\frakP\)-representation attached to \(h\otimes\chi_E\).
  Let \(m\) be an integer with \(1\le m\le 2\nu-1\) and \(m\ne\nu\).
  Let \(T\subset V_{h\otimes\chi_E}(m)\) be a \(\Gal_\bbQ\)-stable \(\calO_\frakP\)-lattice,
  and put \(\overline T=T/\frakP T\).
  Assume that the following conditions hold:
  \begin{enumerate}
    \item \(p\) does not divide \(D_E\) and \(p>2\nu\).
    \item \(\overline T\) is absolutely irreducible and the image of \(\Gal_{\bbQ(\zeta_{p^\infty})}\) under \(\rho_{h\otimes\chi_E}\) contains \(\SL_2(\bbZ_p)\) after a suitable choice of lattice.
    \item \(L^{\alg}(2\nu-m,h\otimes\chi_E)\) is a \(\frakP\)-adic unit.
  \end{enumerate}
  Then
  \[
    H^1_f(\bbQ,\overline T)=0.
  \]
\end{prop}

\begin{proof}
  By Lemma~\ref{lem:kato-twist-periods},
  for each \(1\le s\le 2\nu-1\),
  there exists \(u_s\in\calO_\frakP^\times\) such that
  \[
    L^{\alg}(s,h\otimes\chi_E)
    =
    u_s
    \frac{L(s,h\otimes\chi_E)}
    {(2\pi i)^s\Omega^{(-1)^s}(h\otimes\chi_E)}.
  \]
  Hence our normalization has the same \(\frakP\)-adic valuation as the special value of \(h\otimes\chi_E\) normalized using the integral periods in~\cite{Kato2004HodgeZeta}.

  Put \(A=V_{h\otimes\chi_E}(m)/T\),
  where \(T\) is the chosen lattice.
  If \(H^1_f(\bbQ,\overline T)\ne0\),
  then the absolute irreducibility and non-triviality of \(\overline T\) give
  \[
    H^1_f(\bbQ,\overline T) \hookrightarrow H^1_f(\bbQ,A)[\frakP].
  \]
  Since the level of \(h\otimes\chi_E\) divides \(D_E^2\) and \(p\) does not divide \(D_E\),
  the argument in \cite[the proof of Theorem~7.10]{Atobe2026HarderII},
  based on \cite[Theorem~14.5(3) and Propositions~14.16 and~14.21]{Kato2004HodgeZeta},
  applies to the Dirichlet twist \(h\otimes\chi_E\).
  At every prime \(q\mid D_E\),
  the character \(\chi_E\) is ramified while \(h\) is unramified,
  so the corresponding local Euler factor of \(h\otimes\chi_E\) is \(1\).
  Moreover,
  Lemma~\ref{lem:quadratic-local-vanishing} shows that the residual local condition at \(q\) is automatic.
  It follows that
  \[
    \frakP \mid L^{\alg}(m,h\otimes\chi_E).
  \]
  Since \(p\) does not divide \(D_E\) and \(p>2\nu\),
  the functional equation shows that
  \[
    \frac{L^{\alg}(m,h\otimes\chi_E)}
    {L^{\alg}(2\nu-m,h\otimes\chi_E)}
    \in\calO_\frakP^\times.
  \]
  Hence \(\frakP\mid L^{\alg}(2\nu-m,h\otimes\chi_E)\),
  a contradiction.
\end{proof}

\begin{cor}
  \label{cor:offdiag-Q}
  Let \(k\ge 4\) and \(j\ge 2\) be integers with \(j\) even.
  Let \(f\in S_{2k+j-2}(\SL_2(\bbZ))\) be a primitive elliptic cusp form.
  Put \(\theta_1=\overline{\omega}^{-k+2}\) and \(\theta_2=\overline{\omega}^{-k-j+1}\).
  Assume that the following conditions hold:
  \begin{enumerate}
    \item The prime \(p\) does not divide \(D_E\), and \(p>2k+j-2\).

    \item After a suitable choice of a lattice,
          the image of
          \(\Gal_{\bbQ(\zeta_{p^\infty})}\) under the representation attached to \(f\otimes\chi_E\) contains \(\SL_2(\bbZ_p)\).

    \item The values \(L^{\alg}(k+j,f\otimes\chi_E)\) and
          \(L^{\alg}(k+j-1,f\otimes\chi_E)\) are \(\frakP\)-adic units.
  \end{enumerate}
  Then,
  for each \(i=1,2\) and each ordered pair
  \((\eta,\eta')\in\left\{\left(\theta_i\overline{\chi}_E,\overline{\rho}_f\right),\left(\overline{\rho}_f,\theta_i\overline{\chi}_E\right)\right\}\),
  we have
  \[
    H^1_f\left(\bbQ,\Hom_{\kappa_\frakP}(\eta,\eta')\right)=0.
  \]
\end{cor}

\begin{proof}
  We have
  \[
    \Hom_{\kappa_\frakP}\left(\theta_1\cdot\overline{\chi}_E,\overline{\rho}_f\right)
    =
    \overline{\rho}_{f\otimes\chi_E}\cdot\overline{\omega}^{k-2},
  \]
  and
  \[
    \Hom_{\kappa_\frakP}\left(\theta_2\cdot\overline{\chi}_E,\overline{\rho}_f\right)
    =
    \overline{\rho}_{f\otimes\chi_E}\cdot\overline{\omega}^{k+j-1}.
  \]
  For the opposite extension directions,
  we use the self-duality
  \[
    \overline{\rho}_{f\otimes\chi_E}^{\vee}
    \simeq
    \overline{\rho}_{f\otimes\chi_E}\cdot
    \overline{\omega}^{2k+j-3}.
  \]
  Hence
  \[
    \Hom_{\kappa_\frakP}\left(\overline{\rho}_f,\theta_1\cdot\overline{\chi}_E\right)
    =
    \overline{\rho}_{f\otimes\chi_E}\cdot\overline{\omega}^{k+j-1},
  \]
  and
  \[
    \Hom_{\kappa_\frakP}\left(\overline{\rho}_f,\theta_2\cdot\overline{\chi}_E\right)
    =
    \overline{\rho}_{f\otimes\chi_E}\cdot\overline{\omega}^{k-2}.
  \]

  Apply Proposition~\ref{prop:kato-input} to \(h=f\).
  For \(m=k-2\),
  we have \(2\nu-m=k+j\),
  so the first unit condition in assumption~\((3)\) gives the required vanishing.

  For \(m=k+j-1\),
  we have \(2\nu-m=k-1\).
  In this case we use the functional equation
  \[
    \frac{L^{\alg}(k+j-1,f\otimes\chi_E)}{L^{\alg}(k-1,f\otimes\chi_E)}
    =(-1)^kD_E^{-j}\frac{(k-2)!}{(k+j-2)!}.
  \]
  Under assumption~\((1)\),
  the right-hand side is a \(\frakP\)-unit.
  Hence the second unit condition in assumption~\((3)\) implies that
  \(L^{\alg}(k-1,f\otimes\chi_E)\) is a \(\frakP\)-adic unit,
  and Proposition~\ref{prop:kato-input} gives the remaining vanishings.
\end{proof}

\section{Proof of the main theorem}
\label{sec:main-thm}

Starting from a Hermitian Hecke cusp eigenform \(G\) congruent to \([f]^2_{j+4}\),
we attach a normalized Galois representation \(R_G^\natural\),
determine its residual semisimplification,
and prove its absolute irreducibility and conjugate invariance.
We then use cyclic base change and the endoscopic classification to show that \(G\) is the Hermitian spin lift of a Siegel cusp eigenform.

\subsection{Galois representations and residual congruences}
\label{subsec:main-galois}

Let \(G\in\calS_{\tau_{k,j}}(\U_{2,2})\) be a Hecke cusp eigenform,
and let \(p\) be a rational prime not dividing \(D_E\).
Put
\[
  \GU_{2,2}^{+}(\bbA_\bbQ) = \{g\in\GU_{2,2}(\bbA_\bbQ)\mid \nu(g)_\infty>0\}.
\]
By \cite[Lemma~5.1]{Takeda2026Congruences},
\(G\) extends to a level-one cusp form \(\widetilde G\) on \(\GU_{2,2}^{+}(\bbA_\bbQ)\)
with trivial character on the positive real scalar matrices.
Since
\[
  \GU_{2,2}(\bbA_\bbQ) = \GU_{2,2}(\bbQ)\GU_{2,2}^{+}(\bbA_\bbQ),
\]
we regard \(\widetilde G\) as a cuspidal automorphic form on \(\GU_{2,2}(\bbA_\bbQ)\).

Let \(\pi_G\) be the automorphic representation generated by \(G\).
Choose an irreducible cuspidal constituent \(\widetilde\pi_G\) of the automorphic representation generated by \(\widetilde G\) such that the restriction of \(\widetilde\pi_G\) to \(\U_{2,2}(\bbA_\bbQ)\) contains \(\pi_G\).
Then \(\widetilde\pi_{G,\ell}\) is unramified at every prime \(\ell\) not dividing \(D_E\).

\begin{lem}
  \label{lem:standard-parameter-G}
  The global Arthur parameter of \(\pi_G\) with respect to \(\xi_{\chi_+}\) has trivial Arthur \(\SL_2\)-factor.
  Write its standard parameter as
  \[
    \psi_G^N=\boxplus_{i=1}^r\Pi_i\boxtimes\nu(1),
  \]
  where the \(\Pi_i\) are conjugate self-dual cuspidal automorphic representations of \(\GL_{n_i}(\bbA_E)\).
  Then each \(\Pi_i\) is unramified at every finite place of \(E\).
\end{lem}

\begin{proof}
  With the normalization used here, the infinitesimal character of
  \(\pi_{G,\infty}\) is
  \[
    \left(
    \frac{2k+j-3}{2},
    \frac{j+1}{2},
    -\frac{j+1}{2},
    -\frac{2k+j-3}{2}
    \right).
  \]
  Its adjacent differences are \(k-2\), \(j+1\), and \(k-2\), all greater than \(1\).
  Since a summand \(\mu\boxtimes\nu(d)\) with \(d>1\) would contribute two entries differing by \(1\),
  the Arthur \(\SL_2\)-factor is trivial.
  Since \(G\) has full level,
  the standard local parameter is unramified at every finite prime \(q\) not dividing \(D_E\).
  If \(q\mid D_E\) and \(w\) is the place of \(E\) above \(q\),
  \cite[\S~7.1]{Mok2015Endoscopic} shows that the local parameter of \(\pi_{G,q}\) is induced from an unramified character of \(M_0(\bbQ_q)\simeq E_w^\times\times E_w^\times\).
  Since \(\chi_+=\mathbf1\),
  its standard parameter is unramified on \(W_{E_w}\).
  Thus each \(\Pi_i\) is unramified at every finite place of \(E\).
\end{proof}

Let \(\chi_{\widetilde\pi_G}\) be the central character of \(\widetilde\pi_G\).
Put \(\psi=\chi_{\widetilde\pi_G}^{c}\).
By \cite[Theorem~A]{Skinner2012Galois},
the weak base change of \(\widetilde\pi_G\) is of the form \((\psi,\tau)\),
where \(\tau\) is an automorphic representation of \(\GL_4(\bbA_E)\).
The compatibility with restriction to \(\U_{2,2}\) in \cite[\S~2.3]{Skinner2012Galois} identifies \(\tau\) with the standard base change of \(\pi_G\) at almost all finite places.
Strong multiplicity one therefore gives
\[
  \tau\simeq\boxplus_i\Pi_i.
\]

Fix an isomorphism \(\iota_p:\overline{\bbQ}_p\simeq\bbC\).
Let \(\rho_{\widetilde\pi_G}\) be the representation of \cite[Theorem~B]{Skinner2012Galois},
and let \(\rho_\psi\) be the \(p\)-adic character associated with \(\psi\) by class field theory,
with the same geometric normalization.
We put
\[
  R_G=\rho_{\widetilde\pi_G}\otimes\rho_\psi^{-1}:
  \Gal_E\longrightarrow\GL_4(\overline{\bbQ}_p).
\]
The construction in \cite[\S~4.3, (4.5)]{Skinner2012Galois} gives a decomposition into semisimple representations
\[
  R_G\simeq\bigoplus_{i=1}^r R_i.
\]
For every finite place \(w\) not lying above \(p\),
the Frobenius-semisimplified Weil--Deligne representation of \(R_i|_{\Gal_{E_w}}\) corresponds to \(\Pi_{i,w}\otimes|\det|_w^{-3/2}\).
The same construction gives crystallinity at \(v\mid p\) when \(\Pi_{i,v}\) is unramified.
By Lemma~\ref{lem:standard-parameter-G},
\(R_G\) is therefore unramified outside \(p\),
and crystalline at every place above \(p\).
We write \(\omega_E=\omega|_{\Gal_E}\).
We define
\[
  R_G^\natural=R_G\cdot\omega_E^{3-k-j/2}.
\]
Put \(R_i^\natural=R_i\cdot\omega_E^{3-k-j/2}\).
Then
\[
  R_G^\natural\simeq\bigoplus_{i=1}^r R_i^\natural,
\]
and each \(R_i^\natural\) is conjugate self-dual with multiplier \(\omega_E^{3-2k-j}\).
A direct calculation gives
\[
  (R_G^\natural)^\vee\simeq(R_G^\natural)^c\cdot\omega_E^{2k+j-3}.
\]
The character \(\omega_E\) is unramified away from \(p\),
so the properties of \(R_G\) established above show that \(R_G^\natural\) is unramified at every finite place of \(E\) not dividing \(p\) and crystalline at every place above \(p\).
At every place \(v\mid p\),
the Hodge--Tate weights of \(R_G^\natural|_{\Gal_{E_v}}\) are
\[
  \{-2k-j+3,-k-j+1,-k+2,0\}.
\]
If \(p>2k+j-2\),
these weights lie in \([-(p-2),0]\),
and hence \(R_G^\natural\) is short crystalline at every place above \(p\).

Choose a number field \(K\) containing the Hecke fields under consideration and a prime \(\frakP\) above \(p\) such that \(R_G^\natural\) is realized over \(K_\frakP\).
Let \(\overline R_G^{\natural,ss}\) denote its semisimplified reduction with respect to a stable \(\calO_\frakP\)-lattice.
We enlarge \(K\) and extend \(\frakP\) when necessary.

Let \(\ell\ne p\) be a rational prime which splits in \(E\),
and write \(\ell=ww^c\).
After choosing \(w\),
the split local group is identified with \(\U_{2,2}(\bbQ_\ell)\simeq \GL_4(\bbQ_\ell)\).
Let \(s_w(G)\in\GL_4(\bbC)\) be the corresponding Satake parameter.
We define the normalized split \(L\)-polynomial by
\[
  P_w^\natural(X,G) = \det\left(1-\ell^{k+j/2-3/2}s_w(G)X\right).
\]
With the convention for the standard \(L\)-factor of \(\U_{2,2}\),
the rational \(L\)-polynomial at \(\ell\) factors as
\[
  L_\ell^\natural(X,G,\St) = P_w^\natural(X,G)\,P_{w^c}^\natural(X,G).
\]
Local-global compatibility and the definition of \(R_G^\natural\) give,
after identifying coefficients by \(\iota_p\),
\[
  \det\left(1-R_G^\natural(\Frob_w)X\right) = P_w^\natural(X,G).
\]

We write \(\overline{\omega}_E=\overline{\omega}|_{\Gal_E}\).
We put
\[
  \calB_{f,E} = \overline{\rho}_f^{ss}|_{\Gal_E} + \overline{\omega}_E^{-k+2} + \overline{\omega}_E^{-k-j+1} \in \frakA_\frakP(\Gal_E).
\]

\begin{prop}
  \label{prop:res-class-E}
  If \(G\equiv_{ev}[f]^2_{j+4}\pmod{\frakP}\),
  then
  \[
    \overline{R}_G^{\natural,ss}=\calB_{f,E}.
  \]
\end{prop}

\begin{proof}
  Let \(\ell\) be a rational prime not dividing \(pD_E\) which splits in \(E\),
  and write \(\ell=ww^c\).
  The Satake isomorphism and the explicit split Euler factor of \([f]^2_{j+4}\) give
  \[
    P_w^\natural(X,G) \equiv P_w^\natural(X,[f]^2_{j+4}) \equiv \det\left(1-\calB_{f,E}(\Frob_w)X\right) \pmod{\frakP}.
  \]
  Thus, for both choices of \(w\mid\ell\),
  \[
    \det\left(1-\overline{R}_G^{\natural,ss}(\Frob_w)X\right) = \det\left(1-\calB_{f,E}(\Frob_w)X\right).
  \]
  These places have density one among the finite places of \(E\).
  The Chebotarev density theorem and the Brauer--Nesbitt theorem therefore give
  \(\overline R_G^{\natural,ss}=\calB_{f,E}\).
\end{proof}

\subsection{Absolute irreducibility}
\label{subsec:main-irreducibility}

\begin{lem}
  \label{lem:no-onedim}
  Assume that \(p\) is unramified in \(E\) and that \(p-1\) does not divide \(j+1\).
  Let \(\Delta\) be a one-dimensional \(p\)-adic representation of \(\Gal_E\).
  Assume that \(\Delta\) satisfies
  \[
    \Delta^c \simeq \Delta^\vee\otimes\omega_E^{3-2k-j}.
  \]
  Let \(\overline{\Delta}\) be the residual representation of \(\Delta\).
  Then
  \[
    \overline{\Delta} \ne \overline{\omega}_E^{-k+2} \quad \text{and} \quad \overline{\Delta} \ne \overline{\omega}_E^{-k-j+1}.
  \]
\end{lem}

\begin{proof}
  Reducing the conjugate self-duality modulo \(\frakP\),
  we obtain
  \[
    \overline{\Delta}^{\,c} = \overline{\Delta}^{\,\vee} \otimes \overline{\omega}_E^{3-2k-j}.
  \]
  Since \(\overline{\omega}_E\) is invariant under conjugation,
  if \(\overline{\Delta}=\overline{\omega}_E^{-k+2}\),
  then
  \[
    \overline{\omega}_E^{-k+2} = \overline{\omega}_E^{k-2} \overline{\omega}_E^{3-2k-j} = \overline{\omega}_E^{1-k-j},
  \]
  and hence
  \[
    \overline{\omega}_E^{j+1}=1.
  \]
  The same conclusion follows if \(\overline{\Delta}=\overline{\omega}_E^{-k-j+1}\).

  Since \(p\) is unramified in \(E\),
  the character \(\overline{\omega}_E\) has order \(p-1\).
  Thus \((p-1)\mid(j+1)\),
  contradicting the assumption.
\end{proof}

We also need the following vanishing result for extensions between the two character constituents.
Put \(\theta_1=\overline\omega^{-k+2}\) and \(\theta_2=\overline\omega^{-k-j+1}\).
In the following lemma, the superscript \(+\) refers to the involution induced by
\(D_{\lambda,\bbQ}\), where
\(\lambda=\theta_1\theta_2\overline\chi_E=\overline\omega^{-2k-j+3}\overline\chi_E\).

\begin{lem}
  \label{lem:paired-character-bc-plus-selmer}
  Assume that \(p>j+2\) and that \(p\) does not divide \(D_E\).
  Then
  \[
    H^1 \left( \bbQ, \Hom_{\kappa_\frakP}(\theta_2\overline\chi_E,\theta_1) \right)^{+} =0
  \]
  and
  \[
    H^1 \left( \bbQ, \Hom_{\kappa_\frakP}(\theta_1\overline\chi_E,\theta_2) \right)^{+} =0.
  \]
\end{lem}

\begin{proof}
  We prove the first assertion.
  Let
  \[
    0\rightarrow\theta_1\rightarrow X \rightarrow\theta_2\overline\chi_E\rightarrow0
  \]
  represent a class in the displayed extension group with BC sign \(+1\).
  By Definition~\ref{dfn:bc-extension-sign}, the sign condition gives a self-duality \(X\simeq X^\vee\otimes\lambda\).
  Choose a basis \(e_1,e_2\) adapted to the extension,
  and let \(J\) be the matrix representing this self-duality.
  Then
  \begin{equation}\label{eq:char-ext}
    {}^t\rho_X(g)J\rho_X(g)=\lambda(g)J
    \qquad (g\in\Gal_\bbQ).
  \end{equation}
  Since \(\theta_1^2\ne\lambda\),
  the upper-left entry of \(J\) is zero.
  After rescaling \(e_2\),
  we may write
  \begin{align*}
    J         & = \begin{pmatrix}
                    0 & 1 \\
                    1 & d
                  \end{pmatrix},                           \\
    \rho_X(g) & = \begin{pmatrix}
                    \theta_1(g) & u(g)                       \\
                    0           & \theta_2\overline\chi_E(g)
                  \end{pmatrix}.
  \end{align*}
  The lower-right entry of \eqref{eq:char-ext} gives
  \[
    u(g)=\frac d2(\theta_1(g)-\theta_2\overline\chi_E(g)).
  \]
  Replacing \(e_2\) by \(e_2-\frac d2e_1\) makes \(u(g)=0\) for every \(g\),
  so the extension splits.
  The second assertion follows by interchanging \(\theta_1\) and \(\theta_2\).
\end{proof}

\begin{rem}
  \label{rem:huber-kings-stronger-vanishing}
  By the result of Huber--Kings \cite{Huber2003BlochKato},
  the condition that both
  \(L(0,\chi_E\widetilde\omega^{-j})\) and
  \(L(0,\chi_E\widetilde\omega^{j+2})\) are \(\frakP\)-adic units
  gives the stronger vanishings
  \[
    H^1_f\left(\bbQ,\Hom_{\kappa_\frakP}(\theta_2\overline\chi_E,\theta_1)\right)=0
    \quad\text{and}\quad
    H^1_f\left(\bbQ,\Hom_{\kappa_\frakP}(\theta_1\overline\chi_E,\theta_2)\right)=0.
  \]
  We do not need these stronger results for our later purpose.
\end{rem}

\begin{lem}
  \label{lem:char-block}
  Put \(\delta_1=\overline\omega_E^{-k+2}\) and \(\delta_2=\overline\omega_E^{-k-j+1}\).
  Let \(X\) be a two-dimensional \(K_\frakP\)-linear representation of \(\Gal_E\).
  Assume the following conditions.
  \begin{enumerate}
    \item We have \(p>2k+j-2\),
          the prime \(p\) is unramified in \(E\),
          and \(\zeta(-j-1)\) is a \(\frakP\)-adic unit.

    \item \(X\) is unramified outside the places above \(p\) and short crystalline at every place above \(p\).

    \item \(X\) is conjugate self-dual with multiplier \(\omega_E^{-2k-j+3}\) and BC sign \(+1\).

    \item The representation \(X\) admits a \(\Gal_E\)-stable lattice \(T\) such that
          \[
            (T/\frakP T)^{ss}=\delta_1+\delta_2.
          \]
  \end{enumerate}
  Then \(X\) is not absolutely irreducible.
\end{lem}

\begin{proof}
  Suppose that \(X\) is absolutely irreducible.
  By Ribet's lattice lemma \cite[Proposition~2.1]{Ribet1976Modular},
  after replacing \(T\) by another \(\Gal_E\)-stable lattice,
  we may arrange that \(T/\frakP T\) fits into the following non-split exact sequence:
  \[
    0\rightarrow \delta_2 \rightarrow T/\frakP T \rightarrow \delta_1 \rightarrow 0.
  \]
  Its class gives a non-zero element \(x\in H^1(E,\overline\omega_E^{-j-1})\).
  Since \(X\) is unramified outside \(p\) and short crystalline at every place above \(p\),
  Lemma~\ref{lem:fl-fixed-lifts} gives \(x\in H^1_f(E,\overline\omega_E^{-j-1})\).

  Since \(X\) is absolutely irreducible and has BC sign \(+1\),
  applying \cite[Proposition~1.8.10(i)]{Bellaiche2009Families}
  to the above extension arising from a stable lattice in \(X\) gives
  \[
    \tau_E(x)=x,
  \]
  where \(\tau_E\) is the involution induced by conjugate duality.

  Since \(\overline\omega^{-j-1}(c)=-1\),
  for a cocycle \(z\),
  both the natural action of \(c\) and the involution \(\tau_E\) are given by
  \[
    z(g)\longmapsto -z(cgc^{-1}).
  \]
  Thus \(x\) is fixed by \(c\).

  Since \(p\ne2\),
  inflation--restriction,
  together with the compatibility of the Bloch--Kato local conditions
  with restriction to \(E\),
  gives
  \[
    H^1_f\left(
    \bbQ,
    \overline\omega^{-j-1}
    \right)
    \xrightarrow{\sim}
    H^1_f\left(
    E,
    \overline\omega_E^{-j-1}
    \right)^{c=1}.
  \]
  Hence \(x\) descends to a non-zero class in \(H^1_f(\bbQ,\overline\omega^{-j-1})\).

  On the other hand,
  \(p>j+3\) and the fact that \(\zeta(-j-1)\) is a \(\frakP\)-adic unit imply that \(L(0,\widetilde\omega^{j+1})\) is a \(\frakP\)-adic unit by the generalized Kummer congruences.
  Lemma~\ref{lem:omega-odd-selmer} therefore gives
  \[
    H^1_f\left(
    \bbQ,
    \overline\omega^{-j-1}
    \right)=0,
  \]
  which is a contradiction.
\end{proof}

\begin{prop}
  \label{prop:RG-irred}
  Let \(k\ge 4\) and \(j\ge 2\) be integers with \(j\) even.
  Let \(G \in \calS_{\tau_{k,j}}(\U_{2,2})\) be a Hecke cusp eigenform on \(\U_{2,2}\).
  Let \(R_G^\natural\) be the \(p\)-adic Galois representation of \(\Gal_E\) constructed above.
  Let \(f\) be a primitive elliptic cusp form of weight \(2k+j-2\).
  Choose \(K\) containing the Hecke fields of \(f\) and \(G\),
  and a prime \(\frakP\) of \(K\) such that \(R_G^\natural\) is realized over \(K_\frakP\).
  Assume the following conditions.
  \begin{enumerate}
    \item \label{itm:stable-congruence}
          \(G\equiv_{ev}[f]^2_{j+4}\pmod{\frakP}\).

    \item \label{itm:p-range} The prime \(p=p_\frakP\) satisfies \(p>2k+j-2\) and is unramified in \(E\).

    \item \label{itm:mult-free} After a suitable choice of a lattice and scalar extension to \(K_\frakP\),
          the image of \(\Gal_{\bbQ(\zeta_{p^\infty})}\) under the Galois representation attached to \(f\) contains \(\SL_2(\bbZ_p)\).

    \item \label{itm:char-special} The value \(\zeta(-j-1)\) is a \(\frakP\)-adic unit.
  \end{enumerate}
  Then \(R_G^\natural\) is absolutely irreducible.
\end{prop}

\begin{proof}
  Put \(\delta_1=\overline\omega_E^{-k+2}\), \(\delta_2=\overline\omega_E^{-k-j+1}\).
  By Proposition~\ref{prop:res-class-E} and condition~\((3)\),
  \[
    \overline R_G^{\natural,ss}
    =
    \overline\rho_f|_{\Gal_E}\oplus\delta_1\oplus\delta_2,
  \]
  where \(\overline\rho_f|_{\Gal_E}\) is absolutely irreducible.

  Write the standard parameter as
  \(\boxplus_i\Pi_i\),
  with associated representations \(R_i^\natural\).
  Conjugate duality permutes the absolutely irreducible constituents of these representations,
  and every fixed constituent has BC sign \(+1\) by
  \cite[Theorem~1.2 and Lemma~2.1]{Bellaiche2011Sign}.
  Since each \(\Pi_i\) is conjugate self-dual,
  a one-dimensional block is conjugate self-dual and reduces to
  \(\delta_1\) or \(\delta_2\),
  contradicting Lemma~\ref{lem:no-onedim}.
  Hence the only possible degree partitions are \((4)\) and \((2,2)\).

  Suppose first that the partition is \((2,2)\).
  Then
  \[
    R_G^\natural=R_1^\natural\oplus R_2^\natural,
    \qquad \dim R_1^\natural=\dim R_2^\natural=2.
  \]
  If some \(R_i^\natural\) were reducible,
  then, since \(R_i^\natural\) is semisimple,
  it would be a sum of two characters.
  These characters are crystalline at the places above \(p\) and unramified away from \(p\), and hence arise from algebraic Hecke characters of \(E\).
  Local-global compatibility and strong multiplicity one would then express \(\Pi_i\) as an isobaric sum of two characters,
  contradicting its cuspidality.
  Thus \(R_1^\natural\) and \(R_2^\natural\) are absolutely irreducible.
  Choose \(\Gal_E\)-stable \(\calO_\frakP\)-lattices \(T_i\subset R_i^\natural\).
  By Proposition~\ref{prop:res-class-E},
  \[
    \calB_{f,E} = (T_1/\frakP T_1)^{ss} + (T_2/\frakP T_2)^{ss}.
  \]
  Since both summands have dimension \(2\),
  after interchanging \(R_1^\natural\) and \(R_2^\natural\) if necessary,
  we may assume that
  \[
    (T_1/\frakP T_1)^{ss}
    =
    \delta_1+\delta_2.
  \]
  Lemma~\ref{lem:char-block},
  applied to \(R_1^\natural\),
  then gives a contradiction.
  Hence the partition \((2,2)\) cannot occur.

  It remains to consider the partition \((4)\).
  Suppose that \(R_G^\natural\) is reducible and decompose it into absolutely
  irreducible constituents.
  In type \((3,1)\),
  the one-dimensional constituent is fixed by conjugate duality and is excluded by
  Lemma~\ref{lem:no-onedim};
  type \((1,1,1,1)\) contradicts the cuspidality of \(\Pi_1\).
  In type \((2,2)\),
  the constituent with residual semisimplification \(\delta_1+\delta_2\)
  is fixed by conjugate duality and has BC sign \(+1\),
  again contradicting Lemma~\ref{lem:char-block}.

  Finally suppose that the type is \((2,1,1)\),
  say
  \[
    R_G^\natural=R_0\oplus\Delta_1\oplus\Delta_2.
  \]
  Then
  \[
    \overline R_0^{ss}=\overline\rho_f|_{\Gal_E},
    \qquad
    \{\overline\Delta_1,\overline\Delta_2\}
    =\{\delta_1,\delta_2\}.
  \]
  Since conjugate duality permutes the irreducible constituents of
  \(R_G^\natural\) and preserves their dimensions,
  the unique two-dimensional constituent \(R_0\) is fixed by conjugate duality.
  Hence
  \[
    R_0^\vee
    \simeq
    R_0^c\omega_E^{2k+j-3},
  \]
  and the BC sign is \(+1\).
  Thus \((R_0,\omega^{3-2k-j})\) is polarized in the sense of \cite[\S~2.1]{BarnetLamb2014Potential}.

  Since \(p\) is unramified in \(E\) and \(p>2k+j-2\),
  \cite[Lemma~1.4.3(2)]{BarnetLamb2014Potential} shows that \(R_0\) is potentially diagonalizable at every place above \(p\).

  Moreover,
  \(\Gal_{E(\zeta_{p^\infty})}\) has index at most \(2\) in \(\Gal_{\bbQ(\zeta_{p^\infty})}\).
  Since \(\SL_2(\bbF_p)\) has no quotient of order \(2\),
  condition~\((3)\) and \(\overline R_0\simeq\overline\rho_f|_{\Gal_E}\) give
  \[
    \overline R_0(\Gal_{E(\zeta_p)}) \supset \SL_2(\bbF_p),
  \]
  so \(\overline R_0|_{\Gal_{E(\zeta_p)}}\) is absolutely irreducible.

  The quadratic base change of \(f\) gives a potentially diagonalizable automorphic lift with the same residual polarization.
  Since \(p>2k+j-2\) and \(p\) is prime, we have \(p\ge7\), hence \(p\ge2(\dim R_0+1)\);
  also, \(\zeta_p\notin E\), as \(E\) is unramified at \(p\).
  Thus \cite[Theorem~4.2.1]{BarnetLamb2014Potential} shows that \(R_0\) is automorphic.

  Since \(\Delta_1\) and \(\Delta_2\) arise from algebraic Hecke characters,
  Local-global compatibility and strong multiplicity one would make \(\Pi_1\) an isobaric sum of degrees \(2,1,1\),
  contradicting its cuspidality.

  Therefore \(R_G^\natural\) is absolutely irreducible.
\end{proof}

\subsection{Conjugate invariance and descent}
\label{subsec:main-descent}

Here and below,
a cuspidal descent of a cuspidal automorphic representation \(\Pi_E\) of \(\GL_4(\bbA_E)\) means a cuspidal automorphic representation \(\Pi\) of \(\GL_4(\bbA_\bbQ)\) such that
\[
  \BC_{E/\bbQ}(\Pi)\simeq\Pi_E.
\]

\begin{prop}
  \label{prop:unram-descent}
  Let \(p>2k+j-2\),
  and assume that \(p\) does not divide \(D_E\).
  Let \(f\in S_{2k+j-2}(\SL_2(\bbZ))\) be a primitive elliptic cusp form.
  Let \(G\in\calS_{\tau_{k,j}}(\U_{2,2})\) be a Hecke cusp eigenform,
  and let \(\pi_G\) be the automorphic representation generated by \(G\).
  Assume that the standard parameter of \(\pi_G\) is \(\Pi_G\boxtimes\nu(1)\),
  where \(\Pi_G\) is a cuspidal automorphic representation of \(\GL_4(\bbA_E)\).

  Let \(R_G^\natural\) be the normalized Galois representation attached to \(G\).
  Put \(\theta_1=\overline\omega^{-k+2}\),
  \(\theta_2=\overline\omega^{-k-j+1}\),
  and \(\calC_f=\{\overline\rho_f,\theta_1,\theta_2 \}\).
  Assume the following conditions.
  \begin{enumerate}
    \item
          The representation \(R_G^\natural\) is absolutely irreducible,
          and
          \[
            \overline R_G^{\natural,ss} = \left( \overline\rho_f \oplus \theta_1 \oplus \theta_2 \right) \big|_{\Gal_E}.
          \]

    \item
          The representation \(\overline\rho_f\) is absolutely irreducible,
          the three members of \(\calC_f\) are pairwise non-isomorphic,
          and \(\overline\rho_f\not\simeq\overline\rho_f\otimes\overline\chi_E\).

    \item
          For \(i=1,2\),
          \[
            H^1_f \left( \bbQ, \Hom_{\kappa_\frakP} (\overline\rho_f,\theta_i) \otimes\overline\chi_E \right) = 0
          \]
          and
          \[
            H^1_f \left( \bbQ, \Hom_{\kappa_\frakP} (\theta_i,\overline\rho_f) \otimes\overline\chi_E \right) = 0.
          \]
  \end{enumerate}
  If \(\Pi_G\) admits a cuspidal descent to \(\GL_4(\bbA_\bbQ)\),
  then it admits a cuspidal descent \(\Pi\) such that \(\Pi_\ell\) is unramified for every finite rational prime \(\ell\).
\end{prop}

\begin{proof}
  Let \(\Pi^{(0)}\) be a cuspidal descent of \(\Pi_G\).
  Its archimedean parameter is regular algebraic.
  Since \(\Pi_G^c\simeq\Pi_G\simeq\Pi_G^\vee\),
  uniqueness of quadratic base change gives \((\Pi^{(0)})^\vee\simeq\Pi^{(0)}\) or \(\Pi^{(0)}\otimes\chi_E\).
  Thus \(\Pi^{(0)}\) is essentially self-dual and has a normalized \(p\)-adic Galois representation \(R_{\Pi^{(0)}}^\natural:\Gal_\bbQ\rightarrow\GL_4(\overline{\bbQ}_p)\).
  Compatibility with quadratic base change gives
  \[
    R_{\Pi^{(0)}}^\natural \big|_{\Gal_E} \simeq R_G^\natural.
  \]
  In particular,
  \(R_{\Pi^{(0)}}^\natural\) is absolutely irreducible.

  Assumption~\((2)\) and Clifford theory show that \(\overline\rho_f|_{\Gal_E}\) is absolutely irreducible.
  Each of the three irreducible constituents over \(\Gal_E\) has exactly two extensions to \(\Gal_\bbQ\),
  differing by a twist by \(\overline\chi_E\).
  Hence there exist \(\epsilon_i\in\{0,1\}\) for \(i=0,1,2\) such that
  \[
    \overline R_{\Pi^{(0)}}^{\natural,ss}
    \simeq \left(\overline\rho_f\overline\chi_E^{\epsilon_0}\right)
    \oplus \theta_1\overline\chi_E^{\epsilon_1}
    \oplus \theta_2\overline\chi_E^{\epsilon_2}.
  \]
  These three constituents are pairwise non-isomorphic.

  Put \(w=2k+j-3\).
  Restricting to \(\Gal_E\) and using the conjugate self-duality of \(R_G^\natural\),
  we find that
  \[
    \left( R_{\Pi^{(0)}}^\natural \right)^\vee \simeq R_{\Pi^{(0)}}^\natural \otimes \omega^w\chi_E^\eta
  \]
  for some \(\eta\in\{0,1\}\).
  The unique two-dimensional residual constituent is carried to itself by this self-duality.
  Thus
  \[
    \overline\rho_f \simeq \overline\rho_f \otimes \overline\chi_E^\eta.
  \]
  Assumption~\((2)\) implies that \(\eta=0\).

  Since \(\theta_1^\vee = \theta_2\overline\omega^w\) and \(\theta_2^\vee = \theta_1\overline\omega^w\),
  the two character constituents are exchanged by the self-duality.
  It follows that \(\epsilon_1=\epsilon_2\).

  As established above,
  \(R_G^\natural\) is unramified outside the places of \(E\) above \(p\) and is short crystalline at the places above \(p\).
  Since \(p\) does not divide \(D_E\),
  crystallinity descends through the unramified local extensions \(E_w/\bbQ_p\),
  and the Hodge--Tate weights are unchanged.
  Hence \(R_{\Pi^{(0)}}^\natural\) is unramified outside \(pD_E\) and short crystalline at \(p\).

  Suppose that \(\epsilon_0\ne\epsilon_1\), and consider the partition
  \[
    \left\{ \overline\rho_f\overline\chi_E^{\epsilon_0} \right\} \amalg \left\{ \theta_1\overline\chi_E^{\epsilon_1}, \theta_2\overline\chi_E^{\epsilon_1} \right\}.
  \]
  At each \(q\mid D_E\), every coefficient module for an extension between the two parts is an unramified module twisted by
  \(\overline\chi_E\), since \(f\) has level \(1\).
  The representations \(\rho_f\chi_E^{\epsilon_0}\), \(\omega^{-k+2}\chi_E^{\epsilon_1}\), and \(\omega^{-k-j+1}\chi_E^{\epsilon_1}\),
  lift the constituents of this partition and are short crystalline at \(p\),
  as is \(R_{\Pi^{(0)}}^\natural\).
  Lemma~\ref{lem:fl-fixed-lifts}, applied using stable lattices in these representations, therefore gives the finite local condition at \(p\).
  Applying Lemma~\ref{lem:partitioned-lattice-selmer} to \(R_{\Pi^{(0)}}^\natural\), we therefore obtain a non-zero class in
  \[
    H^1_f\left( \bbQ, \Hom_{\kappa_\frakP}(\overline\rho_f,\theta_i) \otimes\overline\chi_E \right)
  \]
  or
  \[
    H^1_f\left( \bbQ, \Hom_{\kappa_\frakP}(\theta_i,\overline\rho_f) \otimes\overline\chi_E \right),
  \]
  for some \(i\in\{1,2\}\), contradicting assumption~\((3)\).
  Therefore \(\epsilon_0=\epsilon_1=\epsilon_2\).

  Replacing \(\Pi^{(0)}\) by \(\Pi^{(0)}\otimes\chi_E\) if necessary,
  we obtain a cuspidal descent \(\Pi\) satisfying
  \[
    \overline R_\Pi^{\natural,ss} = \overline\rho_f \oplus \theta_1 \oplus \theta_2.
  \]

  It remains to show that \(\Pi\) is unramified at every finite prime.
  Let \(\ell\nmid D_E\), and choose a place \(w\mid\ell\) of \(E\).
  By Lemma~\ref{lem:standard-parameter-G},
  the Weil--Deligne parameter of \(\Pi_\ell\) becomes unramified upon restriction to \(W_{E_w}\).
  Since \(E_w/\bbQ_\ell\) is unramified,
  \(I_{E_w}=I_{\bbQ_\ell}\), and restriction preserves the monodromy operator.
  Thus \(\Pi_\ell\) is unramified, including when \(\ell=p\).

  Now let \(q\mid D_E\),
  and let \(w\) be the unique place of \(E\) above \(q\).
  Lemma~\ref{lem:standard-parameter-G} shows that \(\Pi_{G,w}\) is unramified.
  Let \((\phi_q,N_q)\) be the Weil--Deligne parameter of \(\Pi_q\).
  Its restriction to \(W_{E_w}\) is the unramified parameter of \(\Pi_{G,w}\).
  Therefore \(N_q=0\) and \(\phi_q\) is trivial on \(I_{E_w}\).
  Since \(E_w/\bbQ_q\) is ramified quadratic,
  \[
    I_q/I_{E_w} \simeq \Gal(E_w/\bbQ_q).
  \]
  Hence
  \[
    \phi_q|_{I_q} \simeq \mathbf1^{\oplus a} \oplus \chi_{E,q}^{\oplus b}, \quad a+b=4.
  \]

  Local-global compatibility \cite[Theorem~2.1.1(2)]{BarnetLamb2014Potential} identifies this inertia action with that on \(R_\Pi^\natural\);
  the normalization only introduces an unramified twist.
  Since \(\overline R_\Pi^{\natural,ss}=\overline\rho_f\oplus\theta_1\oplus\theta_2\) is unramified at \(q\) and \(\overline\chi_{E,q}|_{I_q}\ne1\) for \(p>2\), we have \(b=0\).
  Thus \(\Pi_q\) is unramified, completing the proof.
\end{proof}

\begin{thm}
  \label{thm:conj-inv}
  Let \(k\ge 4\) and \(j\ge 2\) be integers with \(j\) even.
  Let \(G \in \calS_{\tau_{k,j}}(\U_{2,2})\) be a Hecke cusp eigenform on \(\U_{2,2}\).
  Let \(R_G^\natural\) be the \(p\)-adic Galois representation of \(\Gal_E\) constructed above.
  Let \(f\) be a primitive elliptic cusp form of weight \(2k+j-2\).
  Choose \(K\) containing the Hecke fields of \(f\) and \(G\),
  and a prime \(\frakP\) of \(K\) such that \(R_G^\natural\) is realized over \(K_\frakP\).
  Assume the following conditions.
  \begin{enumerate}
    \item \label{itm:conj-p-range} The prime \(p=p_\frakP\) satisfies \(p>2k+j-2\).

    \item \label{itm:conj-mult-free} After a suitable choice of a lattice and scalar extension to \(K_\frakP\),
          the image of \(\Gal_{\bbQ(\zeta_{p^\infty})}\) under the Galois representation attached to \(f\) contains \(\SL_2(\bbZ_p)\).

    \item \label{itm:conj-special}
          The value \(\zeta(-j-1)\) is a \(\frakP\)-adic unit.

    \item \label{itm:E-ramification} The prime \(p\) does not divide \(D_Eh_E\).

    \item \label{itm:conj-cong}
          \(G\equiv_{ev}[f]^2_{j+4}\pmod{\frakP}\).

    \item \label{itm:conj-twist-values}
          The values \(L^{\alg}(k+j,f\otimes\chi_E)\) and
          \(L^{\alg}(k+j-1,f\otimes\chi_E)\) are \(\frakP\)-adic units.
  \end{enumerate}
  Then \(R_G^{\natural}\) is conjugate invariant.
  In particular,
  there exists a Hecke cusp eigenform \(F\in\calS_{{\det}^k\Sym^j}(\Sp_2)\) such that
  \[
    G=\Lift_E(F).
  \]
\end{thm}

\begin{proof}
  Put \(\theta_1=\overline\omega^{-k+2}\),
  \(\theta_2=\overline\omega^{-k-j+1}\),
  and \(\calB_f=\overline\rho_f\oplus\theta_1\oplus\theta_2\).

  Proposition~\ref{prop:res-class-E} gives \(\overline R_G^{\natural,ss}=\calB_f|_{\Gal_E}\),
  and Proposition~\ref{prop:RG-irred} shows that \(R_G^\natural\) is absolutely irreducible.

  The decomposition of \(R_G^\natural\) in Skinner's construction therefore has a single summand,
  so the standard parameter of \(\pi_G\) is \(\Pi_G\boxtimes\nu(1)\),
  where \(\Pi_G\) is a conjugate self-dual cuspidal automorphic representation of \(\GL_4(\bbA_E)\).
  By \cite[Theorem~1.2 and Lemma~2.1]{Bellaiche2011Sign},
  the conjugate self-duality of \(R_G^\natural\) has BC sign \(+1\).

  Suppose that \(R_G^\natural\) is not conjugate invariant.
  Clifford theory then shows that
  \[
    I_G^\natural= \Ind_{\Gal_E}^{\Gal_\bbQ}R_G^\natural
  \]
  is absolutely irreducible,
  with residual semisimplification
  \[
    \overline I_G^{\natural,ss} = \calB_f\oplus(\calB_f\otimes\overline\chi_E).
  \]

  Its six irreducible constituents are pairwise non-isomorphic.
  Put \(V=\overline\rho_f\) and \(W=V\otimes\overline\chi_E\).

  Put \(w=2k+j-3\).
  The conjugate self-duality of \(R_G^\natural\) induces on \(I_G^\natural\) a self-duality with multiplier \(\omega^{-w}\chi_E\) and BC sign \(+1\).

  Consider the partition \(\calB_f\amalg(\calB_f\otimes\overline\chi_E)\).
  The representation \(I_G^\natural\) is unramified outside \(pD_E\).
  At \(p\),
  the natural lifts
  \(\rho_f\chi_E^\epsilon\),
  \(\omega^{-k+2}\chi_E^\epsilon\),
  and \(\omega^{-k-j+1}\chi_E^\epsilon\),
  for \(\epsilon\in\{0,1\}\),
  are short crystalline,
  so Lemma~\ref{lem:fl-fixed-lifts} gives the required finite local conditions.
  At each \(q\mid D_E\),
  the corresponding local condition follows from Lemma~\ref{lem:quadratic-local-vanishing}.
  Hence Lemma~\ref{lem:partitioned-lattice-selmer},
  after twisting both constituents by \(\overline\chi_E\) if necessary,
  gives a non-zero class
  \[
    x\in H^1_f\left(\bbQ,\Hom_{\kappa_\frakP}(T,S)\right),
  \]
  where \(S\) is a constituent of \(\calB_f\) and \(T\) is a constituent of \(\calB_f\otimes\overline\chi_E\).
  Put \(\mu=\overline\omega^{-w}\overline\chi_E\).
  If \(T\simeq D_{\mu,\bbQ}(S)\),
  then \cite[Proposition~1.8.10(i)]{Bellaiche2009Families},
  applied to this lattice extension, gives \(\tau_\bbQ(x)=x\);
  thus \(x\) has BC sign \(+1\).

  Since \(p\nmid h_E\),
  Lemma~\ref{lem:chiE-selmer} gives,
  for \(i=1,2\),
  \[
    H^1_f\left(\bbQ,\Hom_{\kappa_\frakP}(\theta_i,\theta_i\overline\chi_E)\right)
    \simeq
    H^1_f\left(\bbQ,\Hom_{\kappa_\frakP}(\theta_i\overline\chi_E,\theta_i)\right)
    \simeq
    H^1_f(\bbQ,\overline\chi_E)
    =
    0.
  \]

  The image of \(\Gal_{\bbQ(\zeta_{p^\infty})}\) under \(\overline\rho_f\otimes\overline\chi_E\) also contains \(\SL_2(\bbF_p)\).
  Indeed,
  if \(H_0=\Gal_{\bbQ(\zeta_{p^\infty})}\cap\ker(\chi_E)\),
  then \(H_0\) has index at most \(2\),
  and \(\SL_2(\bbF_p)\) has no quotient of order \(2\).
  Hence Corollary~\ref{cor:offdiag-Q} gives
  \[
    H^1_f\left(\bbQ,\Hom_{\kappa_\frakP}(S,T)\otimes\overline\chi_E\right)=0
  \]
  whenever exactly one of \(S,T\) is \(\overline\rho_f\) and the other is \(\theta_1\) or \(\theta_2\).

  Together with Corollary~\ref{cor:bc-plus-rho-selmer} and Lemma~\ref{lem:paired-character-bc-plus-selmer},
  the vanishings above exclude every possible non-zero extension between a constituent of \(\calB_f\) and a constituent of \(\calB_f\otimes\overline{\chi}_E\).
  This contradiction shows that \(R_G^\natural\) is conjugate invariant.

  The Chebotarev density theorem and strong multiplicity one give \(\Pi_G^c\simeq\Pi_G\).
  Since \(E/\bbQ\) is quadratic,
  the cyclic base-change theorem of Arthur--Clozel \cite[Chapter~3, Theorem~4.2]{Arthur1989BaseChange} gives a cuspidal descent \(\Pi\) of \(\Pi_G\) to \(\GL_4(\bbA_\bbQ)\).
  Proposition~\ref{prop:unram-descent} allows us to choose \(\Pi\) unramified at every finite place.
  Proposition~\ref{prop:spin-surj} then yields a form \(F\in\calS_{{\det}^k\Sym^j}(\Sp_2)\) with \(G=\Lift_E(F)\).
\end{proof}

\begin{rem}
  \label{rem:advantage-of-bc-sign}
  The method is similar to that of \cite{Atobe2026HarderII},
  but the use of the Bella\"{\i}che--Chenevier sign makes the argument more economical.
  In particular,
  the approach of \cite{Atobe2026HarderII} would require the full vanishings in Remark~\ref{rem:huber-kings-stronger-vanishing},
  whereas the sign argument gives Theorem~\ref{thm:conj-inv} under the weaker hypotheses above.
\end{rem}

\subsection{Harder-type congruence}
\label{subsec:main-harder}

We now combine the preceding results to prove the following Harder-type congruence.

\begin{thm}
  \label{thm:harder-cong}
  Let \(k \ge 4\) and \(j \ge 2\) be integers with \(j\) even.
  Let \(f\) be a primitive elliptic cusp form of weight \(2k+j-2\).
  Let \(E\) be an imaginary quadratic field.
  Let \(K/\bbQ(f)\) be a sufficiently large finite Galois extension.
  Assume that there exist a prime ideal \(\frakp\) of \(\bbQ(f)\)
  and a prime ideal \(\frakP\) of \(K\) lying above \(\frakp\),
  such that the following conditions are satisfied.
  \begin{enumerate}[itemsep=3pt]
    \item The rational prime \(p\) below \(\frakp\) satisfies \(p>2k+j-2\).

    \item After a suitable choice of a \(\Gal_\bbQ\)-stable lattice and scalar extension to \(K_\frakP\),
          the image of \(\Gal_{\bbQ(\zeta_{p^\infty})}\) under the Galois representation attached to \(f\) contains \(\SL_2(\bbZ_p)\).

    \item \(f\) is not congruent modulo \(\frakp\) to any other primitive elliptic cusp form of weight \(2k+j-2\).

    \item The prime \(\frakP\) does not divide \(\zeta(-j-1)\).

    \item \(\frakp\) divides \(L^{\alg}(k+j,f)\).

    \item There exist \(\gamma_0\in\calG_0^{(2)}\) and \(S_0\in\Her_2(E)_{>0}\), with notation as in Subsection~\ref{subsec:herm-cong},
          such that
          \[
            v_\frakp\!\left(
            L^{\alg}(k+j,f)A_{[f]^2_{j+4}}(\gamma_0,S_0)
            \right)=0.
          \]

    \item We have \(p\nmid D_E h_E\).

    \item The prime \(\frakP\) does not divide \(L(-j,\chi_E)\).

    \item The prime \(\frakP\) divides neither \(L^{\alg}(k+j,f\otimes\chi_E)\) nor \(L^{\alg}(k+j-1,f\otimes\chi_E)\).
  \end{enumerate}
  Then there exists a Hecke cusp eigenform \(F\in\calS_{{\det}^k\Sym^j}(\Sp_2)\) such that
  \[
    L_\ell(X,F,\Spin) \equiv L_\ell(X,f)(1-\ell^{k+j-1}X)(1-\ell^{k-2}X) \pmod{\frakP}
  \]
  for every rational prime \(\ell\).
\end{thm}
\begin{proof}
  Put \(\mu=j/2+2\),
  and \(\nu=k+j/2-1\).
  Then \(2\nu=2k+j-2\),
  \(\sigma_{\mu,\nu}=\tau_{k,j}\),
  \([f]^2_{2\mu}=[f]^2_{j+4}\),
  and \((3-2\mu,4-2\mu)=(-j-1,-j)\).
  Assumptions~\((1)\), \((3)\), \((4)\), \((7)\), and~\((8)\) therefore give all hypotheses of Proposition~\ref{prop:main_n2_r1} except its valuation condition.

  If \(s_0=\mu-1\),
  then \(s_0+\nu=k+j\).
  Hence Remark~\ref{rem:degree-one-alg},
  together with assumptions~\((1)\), \((3)\), \((7)\), and~\((9)\),
  shows that \(L^{\alg}(s_0+\tfrac12,\iota_E(f),\St)/L^{\alg}(k+j,f)\) is a \(\frakP\)-adic unit.
  Thus assumptions~\((5)\) and~\((6)\) give the remaining hypothesis,
  and Proposition~\ref{prop:main_n2_r1} yields
  \[
    G\equiv_{ev}[f]^2_{j+4}\pmod{\frakP}
  \]
  for some \(G\in\calS_{\tau_{k,j}}(\U_{2,2})\).

  By assumptions~\((1)\),
  \((2)\),
  \((4)\),
  \((7)\),
  and~\((9)\),
  Theorem~\ref{thm:conj-inv} gives \(F\in\calS_{{\det}^k\Sym^j}(\Sp_2)\) with \(G=\Lift_E(F)\).
  The result now follows from Theorem~\ref{thm:harder}.
\end{proof}

\section{Numerical examples}
\label{sec:examples}

We first give a criterion for the large-image hypothesis in Theorem~\ref{thm:harder-cong}\((2)\).

\begin{lem}
  \label{lem:big-image}
  Let \(h\in S_k(\SL_2(\bbZ))\) be a primitive elliptic cusp form,
  and let \(\frakp\) be a prime of \(\bbQ(h)\) above \(p\).
  Assume that \(p\ge7\) and \(p>k+1\).
  For a prime \(\ell\ne p\),
  put
  \[
    x_\ell
    =
    \overline{a(\ell;h)^2\ell^{1-k}}
    \in\kappa_\frakp.
  \]
  Suppose that the following conditions hold.
  \begin{enumerate}
    \item The value \(\zeta(1-k)\) is a \(\frakp\)-adic unit.

    \item If \(p=2k-1\),
          there exists a prime \(\ell\ne p\) such that
          \[
            \left(\frac{\ell}{p}\right)=-1,
            \quad
            a(\ell;h)\not\equiv0\pmod{\frakp}.
          \]

    \item Either \(p>4k-3\),
          or there exists a prime \(q\ne p\) such that
          \[
            x_q\notin\{0,1,2,4\},
            \quad
            x_q^2-3x_q+1\ne0.
          \]
  \end{enumerate}
  Then,
  after a suitable choice of a stable lattice,
  \[
    \rho_h\left(\Gal_{\bbQ(\zeta_{p^\infty})}\right)
    \supset
    \SL_2(\bbZ_p).
  \]
\end{lem}

\begin{proof}
  By \cite[Theorem~2.5]{Billerey2014LargeImage},
  specialized to level one,
  condition~\((1)\) and \(p>k+1\) imply that \(\overline\rho_h\) is absolutely irreducible.

  If the projective image were dihedral,
  \cite[Theorem~3.2]{Billerey2014LargeImage} and \(p>k+1\) would give \(p=2k-1\).
  The associated quadratic character is unramified outside \(p\) and hence equals \(\ell\mapsto(\frac{\ell}{p})\).
  Elements on which this character equals \(-1\) have trace zero,
  contradicting condition~\((2)\).

  If \(p>4k-3\),
  \cite[Theorem~4.1]{Billerey2014LargeImage} excludes the projective images \(A_4\), \(S_4\), and \(A_5\).
  Otherwise, condition~\((3)\) excludes them:
  if \(r\) is the ratio of the two eigenvalues of \(\overline\rho_h(\Frob_q)\),
  then
  \[
    x_q=r+r^{-1}+2,
  \]
  and projective orders \(1,2,3,4,5\) give respectively
  \[
    x_q=4,0,1,2,
    \quad\text{or}\quad
    x_q^2-3x_q+1=0.
  \]

  Dickson's classification \cite[Section~2]{Ribet1975LAdicRepresentations} now gives
  \[
    \overline\rho_h(\Gal_\bbQ)\supset \SL_2(\bbF_p)
  \]
  after conjugation.
  Since \(\Gal(\bbQ(\zeta_{p^\infty})/\bbQ)\) is abelian and \(\SL_2(\bbF_p)\) is perfect,
  the residual image of \(\Gal_{\bbQ(\zeta_{p^\infty})}\) also contains \(\SL_2(\bbF_p)\).
  The conclusion follows from \cite[Main Theorem]{Manoharmayum2015Structure},
  applied over \(\calO_\frakp\) with residue subfield \(\bbF_p\),
  whose Witt ring is \(\bbZ_p\).
\end{proof}

We apply Theorem~\ref{thm:harder-cong} with \(E=\bbQ(\sqrt{-3})\).
We use Ta\"{\i}bi's dimension formula for level-one automorphic forms \cite{Taibi2017Dimensions}
and his tables for degree-two Siegel modular forms,
available at \url{https://otaibi.perso.math.cnrs.fr/dimtrace/} (accessed October~7, 2026).

Using these data and SageMath,
we enumerate all triples \((k,j,p)\) with \(2k+j-2\le 40\) satisfying conditions~\((1)\) and~\((5)\) of Theorem~\ref{thm:harder-cong}.
For each candidate, we check the remaining hypotheses,
using Lemma~\ref{lem:big-image} for the large-image condition.
Except for the candidates discussed in Remark~\ref{rem:excluded-candidates},
all the hypotheses of Theorem~\ref{thm:harder-cong} are verified.
The resulting triples are listed in Table~\ref{tab:triples}.

We compute the algebraic critical values \(L^{\alg}(s,f\otimes\eta)\) required in Theorem~\ref{thm:harder-cong},
where \(\eta\in\{1,\chi_E\}\),
using integral modular symbols,
normalized so that the corresponding Hecke eigenfunctional is primitive over \(\calO_\frakP\).
Under the hypotheses verified here,
the integral modular-symbol normalization is compatible,
up to a \(\frakP\)-adic unit,
with the integral period normalization used in Kato's construction;
see \cite[\S~14.22]{Kato2004HodgeZeta},
and compare \cite[Appendix~A]{Atobe2026HarderII}.
Since \(p>2k+j-3\) and \(p\) does not divide the conductor of \(\eta\),
the factorial and conductor factors occurring in the modular-symbol formula are \(\frakP\)-adic units,
as are the relevant Gauss sums.
Hence the computation detects divisibility with respect to the periods fixed above.

To check the Fourier-coefficient condition,
we use pullback formulas and differential operators to express the relevant coefficients of the Hermitian Klingen--Eisenstein series in terms of Hermitian Eisenstein coefficients,
as in the method of \cite[Corollary~7.3]{Atobe2023HarderI}.
We compute the latter using \cite[Proposition~5.3]{Takeda2026Congruences},
which in turn is based on the formulas of Shimura \cite[Propositions~18.14 and~19.2]{Shimura1997Euler} and \cite{Shimura1982Confluent}.
For related formulas for Hermitian Siegel series,
see also \cite{Ikeda2008Lifting,Hironaka2011Spherical,Sato2005Fourier}.
For the differential operators,
we use the construction in \cite{Takeda2025Differential}.
For the cases in Table~\ref{tab:triples},
we take \(\gamma_0=1\), with \(S_0=\diag(1,2)\) if \(k=4\) and \(S_0=I_2\) otherwise.

The triples in Table~\ref{tab:triples} are grouped according to the elliptic weight \(2k+j-2\).
An asterisk indicates a congruence that was already established in the literature.

\begingroup
\small
\setlength{\tabcolsep}{0.6em}
\renewcommand{\arraystretch}{1.18}
\begin{longtable}{|>{\bfseries}r|>{\raggedright\arraybackslash}p{0.82\linewidth}|}
  \caption{Triples \((k,j,p)\) for which the Harder-type congruence holds.}
  \label{tab:triples}                       \\
  \hline
  \normalfont\(2k+j-2\) & \((k,j,p)\)       \\
  \hline
  \endfirsthead
  \hline
  \normalfont\(2k+j-2\) & \((k,j,p)\)       \\
  \hline
  \endhead
  22                    & \((10,4,41)^*\)   \\
  \hline
  24                    & \((7,12,73)\),
  \((9,8,179)\)                             \\
  \hline
  26                    & \((5,18,43)^*\),
  \((7,14,97)^*\),
  \((9,10,29)\)                             \\
  \hline
  28                    & \((5,20,193)\),
  \((7,16,367)\),
  \((8,14,647)\),
  \((9,12,4057)\),
  \((10,10,157)\),
  \((11,8,2027)\),
  \((12,6,823)\)                            \\
  \hline
  30                    & \((4,24,97)^*\),
  \((7,18,3779)\),
  \((8,16,1657)^*\),
  \((9,14,1039)\),
  \((10,12,24251)^*\),
  \((11,10,97)\),
  \((12,8,6701)^*\),
  \((14,4,4289)^*\)                         \\
  \hline
  32                    & \((5,24,3119)\),
  \((6,22,7687)\),
  \((7,20,51199)\),
  \((8,18,751)\),
  \((9,16,47)\),
  \((10,14,173)\),
  \((11,12,503)\),
  \((11,12,14243)\),
  \((12,10,1307)\),
  \((15,4,61)\)                             \\
  \hline
  34                    & \((4,28,103)\),
  \((5,26,15511)\),
  \((6,24,389)\),
  \((7,22,1609757)\),
  \((8,20,359291)\),
  \((9,18,37)\),
  \((9,18,253133)\),
  \((10,16,7671023)\),
  \((11,14,43)\),
  \((11,14,1459)\),
  \((12,12,373)\),
  \((12,12,1237)\),
  \((14,8,19163)\),
  \((15,6,233)\),
  \((16,4,1571)\)                           \\
  \hline
  36                    & \((5,28,23509)\),
  \((6,26,43)\),
  \((6,26,2418791)\),
  \((7,24,379)\),
  \((7,24,521)\),
  \((7,24,106487)\),
  \((8,22,191)\),
  \((8,22,727)\),
  \((8,22,39373)\),
  \((9,20,673)\),
  \((9,20,234563653)\),
  \((10,18,9133)\),
  \((10,18,175853)\),
  \((11,16,1297)\),
  \((11,16,94709)\),
  \((12,14,2141)\),
  \((12,14,244837)\),
  \((13,12,373)\),
  \((13,12,44053)\),
  \((14,10,56146411)\),
  \((15,8,35579009)\),
  \((16,6,189009439)\),
  \((17,4,2543)\)                           \\
  \hline
  38                    & \((4,32,67)\),
  \((4,32,83)\),
  \((5,30,67)\),
  \((5,30,1699)\),
  \((6,28,296551)\),
  \((7,26,541368271)\),
  \((8,24,467196139)\),
  \((9,22,293)\),
  \((9,22,33721)\),
  \((10,20,96875477)\),
  \((11,18,3426433)\),
  \((12,16,3673)\),
  \((12,16,863)\),
  \((13,14,58169)\),
  \((14,12,21347)\),
  \((15,10,227)\),
  \((16,8,1381069)\),
  \((17,6,173)\),
  \((18,4,443)\)                            \\
  \hline
  40                    & \((4,34,89)\),
  \((5,32,881)\),
  \((5,32,1013773)\),
  \((6,30,638488410737)\),
  \((7,28,4817)\),
  \((7,28,31973)\),
  \((7,28,72959)\),
  \((8,26,466785177073)\),
  \((9,24,16567193)\),
  \((9,24,257099329)\),
  \((10,22,1531)\),
  \((10,22,1022086667011)\),
  \((11,20,1733)\),
  \((11,20,5204097003667)\),
  \((12,18,1298443505929039)\),
  \((13,16,157)\),
  \((13,16,2309)\),
  \((13,16,2640959)\),
  \((14,14,41227)\),
  \((14,14,619561)\),
  \((15,12,239)\),
  \((15,12,31534597)\),
  \((16,10,313)\),
  \((16,10,829)\),
  \((16,10,2693)\),
  \((17,8,468191)\),
  \((18,6,67)\),
  \((18,6,1875431)\),
  \((19,4,397)\),
  \((19,4,947)\)                            \\
  \hline
\end{longtable}
\endgroup

\begin{rem}[Previously known cases]
  \label{rem:previously-known-examples}
  The case \((k,j,p)=(10,4,41)\) is the classical example of Harder's congruence.
  It was verified by Chenevier--Lannes~\cite{Chenevier2019Automorphic} and is also covered by \cite{Atobe2023HarderI}.

  For even \(k\),
  the cases \((k,j,p)=(14,4,4289), (4,24,97)\) were proved in \cite{Atobe2023HarderI},
  while \((k,j,p)= (12,8,6701), (10,12,24251), (8,16,1657)\) were proved in \cite{Atobe2026HarderII}.

  For odd \(k\),
  Katsurada--Lee~\cite{Katsurada2026HarderMiyawaki} proved Harder-type congruences by means of Miyawaki lifts.
  Their results include \((k,j,p)= (7,14,97), (5,18,43)\).
\end{rem}

We illustrate the verification for \((k,j,p)=(7,12,73)\),
which is not among the previously known cases listed above.

\begin{ex}
  \label{ex:7-12-73}
  For \((k,j,p)=(7,12,73)\),
  take \(E=\bbQ(\sqrt{-3})\).
  Then \(D_E=-3\),
  \(h_E=1\),
  and \(2k+j-2=24\).
  Thus conditions~\((1)\) and~\((7)\) of Theorem~\ref{thm:harder-cong} are satisfied.

  Let \(f\in S_{24}(\SL_2(\bbZ))\) be the primitive elliptic cusp form with
  \[
    f=q+a_0q^2+\cdots,
  \]
  where \(a_0\) is a root of \(X^2-1080X-20468736\).
  Then \(\bbQ(f)=\bbQ(a_0)\).
  Choose the prime \(\frakp=(73,a_0+2)\) of \(\bbQ(f)\).
  Let \(\frakP\) be a prime above \(\frakp\) in a sufficiently large finite Galois extension \(K/\bbQ(f)\), as in Theorem~\ref{thm:harder-cong}.
  Since \(S_{24}(\SL_2(\bbZ))\) has dimension two,
  the other primitive elliptic cusp form is the Galois conjugate \(f^\sigma\) of \(f\).
  Their \(T(2)\)-eigenvalues satisfy
  \begin{align*}
    a(2;f)        & =a_0\equiv71,     \\
    a(2;f^\sigma) & =1080-a_0\equiv60
  \end{align*}
  modulo \(\frakp\),
  so condition~\((3)\) holds.

  We next verify the large-image hypothesis.
  We have \(73>24+1\) and \(73 \neq 2\cdot 24-1\).
  We also have
  \[
    \zeta(-23)  =\frac{103\cdot2294797}{2^4\cdot3^2\cdot5\cdot7\cdot13},
  \]
  which is a \(73\)-adic unit.
  Moreover,
  \[
    x_2=\overline{a(2;f)^2 2^{-23}}=64\notin\{0,1,2,4\},
    \qquad x_2^2-3x_2+1=36\ne0
  \]
  in \(\bbF_{73}\).
  Lemma~\ref{lem:big-image} therefore gives condition~\((2)\) of Theorem~\ref{thm:harder-cong}.

  The remaining elementary special values are
  \begin{align*}
    \zeta(-13)    & =-\frac1{2^2\cdot3},                \\
    L(-12,\chi_E) & =\frac{2\cdot7\cdot13^2\cdot47}{3},
  \end{align*}
  and both are \(73\)-adic units.
  Thus conditions~\((4)\) and~\((8)\) are satisfied.
  We also have \(\frakp\mid L^{\alg}(19,f)\),
  whereas \(L^{\alg}(19,f\otimes\chi_E)\) and
  \(L^{\alg}(18,f\otimes\chi_E)\) are \(\frakp\)-adic units.
  These relations verify conditions~\((5)\) and~\((9)\).

  It remains to check the Fourier-coefficient condition.
  \cite[Example~5.7]{Takeda2026PrimitiveFourier},
  using Corollary~5.6 and Corollary~5.3 of loc.~cit.,
  gives a cusp representative \(\gamma_0\) and a positive definite Fourier index \(S_0\) such that
  \[
    v_\frakp\!\left(
    L^{\alg}(19,f)
    A_{[f]^2_{16}}(\gamma_0,S_0)
    \right)
    =0.
  \]
  Hence condition~\((6)\) holds.
  Thus all the hypotheses of Theorem~\ref{thm:harder-cong} are satisfied.

  It follows that there exists a Hecke cusp eigenform \(F\in\calS_{\det^7\Sym^{12}}(\Sp_2)\) such that
  \[
    L_\ell(X,F,\Spin)
    \equiv
    L_\ell(X,f)
    (1-\ell^{18}X)
    (1-\ell^5X)
    \pmod{\frakP}
  \]
  for every rational prime \(\ell\).
\end{ex}

\begin{rem}
  Chenevier--Lannes also obtained a Harder-type congruence for \((k,j,p)=(8,6,11)\),
  in fact modulo \(11^2\) \cite{Chenevier2019Automorphic}.
  This case does not satisfy the lower bound on \(p\) in Theorem~\ref{thm:harder-cong},
  and hence lies outside the range of our criterion.
\end{rem}

\begin{rem}
  \label{rem:excluded-candidates}
  The following candidates satisfy conditions~\((1)\) and~\((5)\) of Theorem~\ref{thm:harder-cong},
  but are not included in Table~\ref{tab:triples}.
  In each case,
  either one of the remaining hypotheses fails or the large-image criterion used above does not certify the required image condition.

  \begin{enumerate}
    \item
          For \((k,j,p)=(6,20,593)\),
          we have
          \[
            \zeta(-21)
            =
            -\frac{593\cdot1441}{2^2\cdot3\cdot11\cdot23}.
          \]
          Hence \(593\) divides \(\zeta(-21)\),
          so condition~(4) in Theorem~\ref{thm:harder-cong} is not satisfied.

    \item
          For \((k,j,p)=(7,20,37)\),
          the elliptic cusp form has weight \(32\).
          We have
          \[
            \zeta(-31)
            =
            \frac{37\cdot208360028141}{2^6\cdot3\cdot5\cdot17}.
          \]
          Thus condition~\((1)\) of Lemma~\ref{lem:big-image} fails,
          and this criterion does not establish the large-image hypothesis for this candidate.

    \item For \((k,j,p)=(4,30,71)\) and \((13,12,71)\),
          the associated elliptic cusp form has weight \(36\),
          so \(p=71=2\cdot36-1\).
          In these exceptional cases the dihedral possibility in Lemma~\ref{lem:big-image} actually occurs.
          Indeed, writing
          \[
            \chi_{-71}(\ell)=\left(\frac{\ell}{71}\right),
          \]
          we verified,
          for each prime \(\frakp\) of the coefficient field above \(71\),
          that
          \[
            a(\ell;f)\equiv0\pmod{\frakp}
          \]
          for every prime \(\ell\le15336\) with \(\chi_{-71}(\ell)=-1\).
          Here \(15336\) is the Sturm bound for weight \(36\) and level \(71^2\).
          By the Hecke relations,
          this implies that the \(71\)-depletion of \(f\) is congruent modulo \(\frakp\) to \(f\otimes\chi_{-71}\).
          Hence
          \[
            \overline{\rho}_f\simeq
            \overline{\rho}_f\otimes\chi_{-71}.
          \]
          Since \(\overline{\rho}_f\) is irreducible,
          it is induced from \(G_{\mathbf Q(\sqrt{-71})}\);
          in particular, its projective image is dihedral.
    \item
          For \((k,j,p)=(9,16,67)\),
          the elliptic cusp form again has weight \(32\).
          The space \(S_{32}(\SL_2(\bbZ))\) has dimension \(2\),
          and its two primitive elliptic cusp forms form a single Galois orbit.
          If \(K=\bbQ(f)\), then
          \[
            \disc(K)=67\cdot273067.
          \]
          Hence \(67\) is ramified in \(K\).
          If \(\sigma\) denotes the nontrivial automorphism of \(K/\bbQ\),
          then \(\sigma\) acts trivially on the residue field at the unique
          prime \(\frakp\) above \(67\).
          Consequently,
          \[
            f^\sigma\equiv f\pmod{\frakp},
          \]
          and condition~\((3)\) of
          Theorem~\ref{thm:harder-cong} is not satisfied.
  \end{enumerate}
\end{rem}

\bibliographystyle{amsplain}
\bibliography{harder_hermitian}
\end{document}